\documentclass{amsart}

\usepackage[T1]{fontenc}
\usepackage{enumerate, amsmath, amsfonts, amssymb, amsthm, mathrsfs, wasysym, graphics, graphicx, xcolor, url, hyperref, hypcap, xargs, multicol, pdflscape, multirow, hvfloat, array, ae, aecompl, pifont, mathtools, a4wide, float, blkarray, overpic, nicefrac}
\usepackage[noabbrev,capitalise]{cleveref}
\usepackage[normalem]{ulem}
\usepackage{marginnote}
\hypersetup{colorlinks=true, citecolor=darkblue, linkcolor=darkblue}
\usepackage[all]{xy}
\usepackage{tikz}
\usepackage{tikz-cd}
\usetikzlibrary{trees, decorations, decorations.pathmorphing, decorations.markings, decorations.shapes, shapes, arrows, matrix, calc, fit, intersections, patterns, angles}
\graphicspath{{figures/}{figures/diagonals/}{figures/walks/}{figures/tubes/}}
\makeatletter\def\input@path{{figures/}}\makeatother
\usepackage{caption}
\usepackage[export]{adjustbox}

\newtheorem{theorem}{Theorem}[section]
\newtheorem{corollary}[theorem]{Corollary}
\newtheorem{proposition}[theorem]{Proposition}
\newtheorem{lemma}[theorem]{Lemma}
\newtheorem{conjecture}[theorem]{Conjecture}
\newtheorem*{theorem*}{Theorem}

\theoremstyle{definition}
\newtheorem{definition}[theorem]{Definition}
\newtheorem{example}[theorem]{Example}
\newtheorem{remark}[theorem]{Remark}

\newtheorem{notation}[theorem]{Notation}

\crefname{equation}{Equation}{Equations}

\newcommand{\R}{\mathbb{R}} 
\newcommand{\Q}{\mathbb{Q}} 
\newcommand{\N}{\mathbb{N}} 
\newcommand{\Z}{\mathbb{Z}} 
\newcommand{\HH}{\mathbb{H}} 
\newcommand{\cal}[1]{{\mathcal{#1}}} 
\renewcommand{\b}[1]{{\boldsymbol{#1}}} 

\newcommand{\set}[2]{\left\{ #1 \;\middle|\; #2 \right\}} 
\newcommand{\bigset}[2]{\big\{ #1 \;\big|\; #2 \big\}} 
\newcommand{\Bigset}[2]{\Big\{ #1 \;\Big|\; #2 \Big\}} 
\newcommand{\ssm}{\smallsetminus} 
\newcommand{\dotprod}[2]{\left\langle \, #1 \; \middle| \; #2 \, \right\rangle} 
\newcommand{\bigdotprod}[2]{\big\langle \, #1 \; \big| \; #2 \, \big\rangle} 
\newcommand{\eqdef}{\mbox{\,\raisebox{0.2ex}{\scriptsize\ensuremath{\mathrm:}}\ensuremath{=}\,}} 
\newcommand{\simplex}{\triangle} 
\newcommand{\transpose}[1]{{#1}^T} 

\DeclareMathOperator{\rank}{rank} 

\newcommand{\ie}{\textit{i.e.}~} 
\newcommand{\eg}{\textit{e.g.}~} 
\newcommand{\apriori}{\textit{a priori}} 
\definecolor{darkblue}{rgb}{0,0,0.7} 
\definecolor{green}{RGB}{57,181,74} 
\definecolor{violet}{RGB}{147,39,143} 
\newcommand{\red}{\color{red}} 
\newcommand{\blue}{\color{blue}} 
\newcommand{\orange}{\color{orange}} 
\newcommand{\green}{\color{green}} 
\newcommand{\darkblue}{\color{darkblue}} 
\newcommand{\defn}[1]{\textsl{\darkblue #1}} 
\newcommand{\para}[1]{\medskip\noindent\uline{\textit{#1.}}} 
\newcommand*\circled[1]{\tikz[baseline=(char.base)]{\node[shape=circle, draw, inner sep=1.5pt, scale=.7] (char) {#1};}}
\newcommand{\compactVectorD}[2]{\begin{bmatrix} #1 \\ #2 \end{bmatrix}}
\newcommand{\compactVectorT}[3]{\begin{bmatrix} #1 \\[-.1cm] #2 \\[-.1cm] #3 \end{bmatrix}}

\usepackage{todonotes}

\newcommandx{\Asso}[2][1=n,2={}]{\mathsf{Asso}^{#2}(#1)} 
\newcommandx{\Zono}[2][1=n,2={}]{\mathsf{Zono}^{#2}(#1)} 
\newcommandx{\Fan}[1][1=F]{\mathcal{#1}} 
\newcommand{\multiplicityVector}{\b{m}} 
\newcommandx{\ray}[1][1=r]{\b{#1}} 
\newcommandx{\rays}[1][1=R]{\b{#1}} 
\newcommand{\gvector}[1]{\b{g}(#1)} 
\newcommand{\gvectorFull}[2]{\b{g}(#1,#2)} 
\newcommand{\gvectors}[1]{\b{g}(#1)} 
\newcommand{\gvectorsFull}[2]{\b{g}(#1,#2)} 
\newcommandx{\gvectorFan}[1][1=\quiver]{\mathcal{F}(#1)} 
\newcommand{\cvector}[2]{\mathbf{c}(#2 \in #1)} 
\newcommand{\cvectorFull}[3]{\mathbf{c}(#1,#3 \in #2)} 
\newcommand{\cvectors}[1]{\mathbf{c}(#1)} 
\newcommand{\cvectorsFull}[2]{\mathbf{c}(#1,#2)} 

\newcommand{\typeCone}{\mathbb{TC}} 
\newcommand{\ctypeCone}{\smash{\overline{\mathbb{TC}}}} 
\newcommandx{\coefficient}[3][1={\ray[s]}, 2=\ray, 3=\ray']{\alpha_{#2,#3}(#1)} 

\newcommand{\Trop}[1]{\mathbb{P}_{#1}} 
\newcommand{\positiveExponents}[1]{\left\{#1\right\}_+} 
\newcommand{\negativeExponents}[1]{\left\{#1\right\}_-} 
\newcommand{\seed}{\Sigma} 
\newcommand{\cluster}{\mathrm{X}} 
\newcommandx{\clusters}{\mathcal{X}} 
\newcommand{\coefficients}{\mathrm{P}} 
\newcommandx{\variables}[2][1=\B_\circ, 2 = \coefficients_\circ]{\mathcal{V}(#1,#2)} 
\newcommandx{\clusterAlgebra}[2][1=\B_\circ, 2=\coefficients_\circ]{\mathcal{A}\!\left(#1,#2\right)} 
\newcommandx{\principalClusterAlgebra}[1][1=\B_\circ]{\mathcal{A}\left(#1\right)} 
\newcommandx{\principalVariables}[1][1=\B_\circ]{\mathcal{V}(#1)} 
\newcommandx{\coefficientFreeClusterAlgebra}[1][1=\B_\circ]{\mathcal{A}_{\mathrm{fr}}(#1)} 
\newcommandx{\clusterComplex}[1][1=\B_\circ]{\Delta(#1)} 
\newcommand{\B}{\mathrm{B}} 
\newcommand{\A}[1]{\mathrm{A}({#1})} 
\newcommand{\D}{\mathrm{D}} 
\newcommand{\simpleRoot}{\b{\alpha}} 
\newcommand{\fundamentalWeight}{\b{\omega}} 
\newcommandx{\meshes}[1][1=\B_\circ]{\mathcal{M}(#1)} 

\newcommand{\quiver}{\bar Q} 
\newcommand{\blossom}{^\text{\ding{96}}} 
\newcommandx{\strings}[1][1=\quiver]{\mathcal{S}(#1)} 
\newcommandx{\distinguishableStrings}[1][1=\quiver]{\mathcal{S}_\mathrm{dist}(#1)} 
\newcommandx{\walks}[1][1=\quiver]{\mathcal{W}(#1)} 
\newcommandx{\properWalks}[1][1=\quiver]{\mathcal{W}_\mathrm{prop}(#1)} 
\renewcommand{\top}{\mathrm{top}} 
\newcommand{\bottom}{\mathrm{bot}} 
\newcommandx{\NKC}[1][1=\quiver]{\mathcal{NK}(#1)} 
\newcommandx{\RNKC}[1][1=\quiver]{\mathcal{RNK}(#1)} 
\newcommand{\peaks}[1]{\mathsf{peaks}(#1)} 
\newcommand{\deeps}[1]{\mathsf{deeps}(#1)} 
\newcommand{\KN}{\textsc{kn}} 
\newcommand{\hL}{\text{\rotatebox[origin=c]{180}{\checked}}}
\newcommand{\hR}{\text{\rotatebox[origin=c]{180}{\reflectbox{\checked}}}}
\newcommand{\cL}{\text{\reflectbox{\checked}}}
\newcommand{\cR}{\text{\checked}}
\newcommand{\hh}[1]{\hL#1\hR} 
\newcommand{\cc}[1]{\cL#1\cR} 
\newcommand{\hc}[1]{\hL#1\cR} 
\newcommand{\ch}[1]{\cL#1\hR} 
\newcommand{\distinguishedString}[2]{\mathsf{ds}(#1,#2)} 
\newcommand{\distinguishedSign}[2]{\varepsilon(#1,#2)} 

\newcommandx{\graphG}[1][1=G]{#1} 
\newcommandx{\tube}[1][1=t]{\mathsf{#1}} 
\newcommandx{\tubes}[1][1=\graphG]{\mathcal{T}(#1)} 
\newcommandx{\tubing}[1][1=T]{\mathsf{#1}} 

\newcommand{\field}{\mathbb{K}}
\newcommand{\cat}{\mathcal{C}}
\newcommand{\Hom}[1]{\operatorname{Hom}_{#1}}
\newcommand{\susp}{\Sigma}
\newcommand{\add}{\operatorname{add}}
\newcommand{\MOD}{\operatorname{mod}}
\newcommand{\End}[1]{\operatorname{End}_{#1}}
\newcommand{\spl}{\operatorname{sp}}
\newcommand{\Ksp}{K_0^{\spl}}

\newcommand{\ind}{\operatorname{ind}}
\newcommand{\coker}{\operatorname{coker}}
\newcommand{\CC}{\phi}
\newcommand{\proj}{\operatorname{proj}}
\newcommand{\bsm}{\begin{smallmatrix}}
\newcommand{\esm}{\end{smallmatrix}}

\newcommand{\tc}{\mathcal{T}}
\newcommand{\ec}{\mathcal{E}}
\newcommand{\infl}{\rightarrowtail}
\newcommand{\defl}{\twoheadrightarrow}
\newcommand{\Modt}{\operatorname{Mod}\tc}
\newcommand{\modt}{\MOD\tc}
\newcommand{\kzero}[1]{K_0(#1)}
\newcommand{\eps}{\varepsilon}

\makeatletter
\def\part{\@startsection{part}{1}%
\z@{.7\linespacing\@plus\linespacing}{.8\linespacing}%
{\LARGE\sffamily\centering}}
\makeatother

\makeatletter
\def\l@section{\@tocline{1}{5pt}{0pc}{}{}}
\makeatother
\let\oldtocpart=\tocpart
\renewcommand{\tocpart}[2]{\sc\large\oldtocpart{#1}{#2}}
\let\oldtocsection=\tocsection
\renewcommand{\tocsection}[2]{\bf\oldtocsection{#1}{#2}}
\let\oldtocsubsubsection=\tocsubsubsection
\renewcommand{\tocsubsubsection}[2]{\quad\oldtocsubsubsection{#1}{#2}}

\title[Generalized associahedra and minimal relations between $\b{g}$-vectors]{Associahedra for finite type cluster algebras \\ and minimal relations between $\b{g}$-vectors}

\thanks{YP, VP and PGP were partially supported by the French ANR grant SC3A~(15\,CE40\,0004\,01). AP and VP were partially supported by the French ANR grant CAPPS~(17\,CE40\,0018).}


\author{Arnau Padrol}
\address[Arnau Padrol]{Institut de Math\'{e}matiques de Jussieu - Paris Rive Gauche, Sorbonne Universit\'e, Paris}
\email{arnau.padrol@imj-prg.fr}
\urladdr{\url{https://webusers.imj-prg.fr/~arnau.padrol/}}

\author{Yann Palu}
\address[Yann Palu]{LAMFA, Universit\'e Picardie Jules Verne, Amiens}
\email{yann.palu@u-picardie.fr}
\urladdr{\url{http://www.lamfa.u-picardie.fr/palu/}}

\author{Vincent Pilaud}
\address[Vincent Pilaud]{CNRS \& LIX, \'Ecole Polytechnique, Palaiseau}
\email{vincent.pilaud@lix.polytechnique.fr}
\urladdr{\url{http://www.lix.polytechnique.fr/~pilaud/}}

\author{Pierre-Guy Plamondon}
\address[Pierre-Guy Plamondon]{Laboratoire de Math\'ematiques de Versailles, UVSQ, CNRS, Universit\'e Paris-Saclay, Versailles}
\email{pierre-guy.plamondon@uvsq.fr}
\urladdr{\url{https://www.imo.universite-paris-saclay.fr/~plamondon/}}

\begin{document}

\begin{abstract}
We show that the mesh mutations are the minimal relations among the $\b{g}$-vectors with respect to any initial seed in any finite type cluster algebra.
We then use this algebraic result to derive geometric properties of the $\b{g}$-vector fan: we show that the space of all its polytopal realizations is a simplicial cone, and we then observe that this property implies that all its realizations can be described as the intersection of a high dimensional positive orthant with well-chosen affine spaces.
This sheds a new light on and extends earlier results of N.~Arkani-Hamed, Y.~Bai, S.~He, and G.~Yan in type~$A$ and of V.~Bazier-Matte, G.~Douville, K.~Mousavand, H.~Thomas and E.~Y\i ld\i r\i m for acyclic initial seeds.

Moreover, we use a similar approach to study the space of polytopal realizations of the \mbox{$\b{g}$-vector} fans of another generalization of the associahedron: non-kissing complexes (a.k.a. support $\tau$-tilting complexes) of gentle algebras. We show that the space of realizations of the non-kissing fan is simplicial when the gentle bound quiver is brick and $2$-acyclic, and we describe in this case its facet-defining inequalities in terms of mesh mutations.

Along the way, we prove algebraic results on~$2$-Calabi--Yau triangulated categories, and on extriangulated categories that are of independent interest.
In particular, we prove, in those two setups, an analogue of a result of M.~Auslander on minimal relations for Grothendieck groups of module categories.
\end{abstract}

\maketitle

%
\enlargethispage{.3cm}
\tableofcontents
\vspace{-.5cm}


\section*{Introduction}

The associahedron is a ``mythical polytope''~\cite{Haiman} whose facial structure encodes Catalan families: its vertices correspond to parenthesizations of a non-associative product, triangulations of a convex polygon, or binary trees; its edges correspond to applications of the associativity rule, diagonal flips, or edge rotations; and in general its faces correspond to partial parenthesizations, diagonal dissections, or Schr\"oder trees.
Its combinatorial and topological structure was introduced in early works of D.~Tamari~\cite{Tamari} and J.~Stasheff~\cite{Stasheff}, and a first $3$-dimensional polytopal model was realized by J.~Milnor for the PhD defense of J.~Stasheff.
The first systematic polytopal realizations were constructed by M.~Haiman~\cite{Haiman} and C.~Lee~\cite{Lee}.
Since then, the associahedron has been largely ``demystified'' with several polytopal constructions and generalizations (some are discussed below).
Its realizations can be classified into three species: the secondary polytopes~\cite{GelfandKapranovZelevinsky, BilleraFillimanSturmfels}, the $\b{g}$-vectors realizations~\cite{Loday, HohlwegLange, LangePilaud, HohlwegLangeThomas, HohlwegPilaudStella, Postnikov, PilaudSantos-brickPolytope, PilaudStump-brickPolytope} and the $\b{d}$-vector realizations~\cite{ChapotonFominZelevinsky, CeballosSantosZiegler}. See~\cite{CeballosSantosZiegler} for a discussion of some of these realizations and their connections.

The associahedron appears as a fundamental structure in several mathematical theories, such as moduli spaces and topology~\cite{Stasheff, Keller-AinfinityAlgebras}, operads and rewriting theory~\cite{Street, MasudaThomasTonksVallette}, combinatorial Hopf algebras~\cite{LodayRonco, ChatelPilaud, Pilaud-brickAlgebra}, diagonal harmonics~\cite{BergeronPrevilleRatelle, PrevilleRatelleViennot}, mathematical physics~\cite{ArkaniHamedBaiHeYan}, etc.
Another striking illustration of the ubiquity of the associahedron lies in the theory of cluster algebras introduced by S.~Fomin and A.~Zelevinsky in~\cite{FominZelevinsky-ClusterAlgebrasI} with motivation coming from total positivity and canonical bases.
Cluster algebras are commutative rings generated by variables obtained from an initial seed by the discrete dynamical process of seed mutations.
Finite type cluster algebras, whose mutation graph is finite, were classified in~\cite{FominZelevinsky-ClusterAlgebrasII} using the Cartan-Killing classification for crystallographic root systems.
Finite type cluster complexes were geometrically realized first by F.~Chapoton, S.~Fomin and A.~Zelevinsky~\cite{ChapotonFominZelevinsky} based on the $\b{d}$-vector fans~\cite{FominZelevinsky-YSystems, FominZelevinsky-ClusterAlgebrasII} of the bipartite initial seed, then by~C.~Hohlweg, C.~Lange and H.~Thomas~\cite{HohlwegLangeThomas} based on the Cambrian fans~\cite{Reading-CambrianLattices, ReadingSpeyer} of acyclic initial seeds, and finally by C.~Hohlweg, V.~Pilaud and S.~Stella~\cite{HohlwegPilaudStella} based on the $\b{g}$-vector fans~\cite{FominZelevinsky-ClusterAlgebrasIV} with respect to an arbitrary initial seed.
The resulting polytopes are known as generalized associahedra, and the classical associahedra corresponds to the type~$A$ cluster algebras.

This paper focuses on a surprising construction of the associahedron that recently appeared in mathematical physics.
Motivated by the prediction of the behavior of scattering particles, N.~Arkani-Hamed, Y.~Bai, S.~He, and G.~Yan recently described in~\cite[Sect.~3.2]{ArkaniHamedBaiHeYan} the kinematic associahedron.
It is a class of polytopal realizations of the classical associahedron obtained as sections of a high-dimensional positive orthant with well-chosen affine subspaces.
This construction provides a large degree of freedom in the choice of the parameters defining these affine subspaces, and actually produces all polytopes whose normal fan is affinely equivalent to that of J.-L.~Loday's associahedron~\cite{Loday} (see \cref{subsec:typeConeAsso}).
These realizations were then extended by V.~Bazier-Matte, G.~Douville, K.~Mousavand, H.~Thomas and E.~Y\i ld\i r\i m~\cite{BazierMatteDouvilleMousavandThomasYildirim} in the context of finite type cluster algebras using tools from representation theory of quivers.
More precisely, they fix a finite type cluster algebra~$\cal{A}$ and consider the real euclidean space~$\R^\cal{V}$ indexed by the set~$\cal{V}$ of cluster variables of~$\cal{A}$.
Starting from an acyclic initial seed~$\seed_\circ$, they consider a set~$\cal{M}$ of mesh mutations and compute the corresponding linear dependences between the $\b{g}$-vectors, which can be interpreted as linear spaces in~$\R^\cal{V}$.
Finally, they perturb these mesh linear spaces by a collection~$\b{\ell} \in \R_{>0}^{\cal{M}}$ of positive parameters and intersect the resulting perturbed mesh affine spaces with the positive orthant~$\R_{\ge 0}^\cal{V}$ (see \cref{subsec:typeConeCA}).
They show that the resulting polytope is always a generalized associahedron, and that its normal fan is the $\b{g}$-vector fan of~$\cal{A}$ with respect to the initial acyclic seed~$\seed_\circ$.
An implicit by-product of their construction is that the space of all polytopal realizations of the $\b{g}$-vector fan of~$\cal{A}$ with respect to the initial acyclic \mbox{seed~$\seed_\circ$ is a simplicial cone}.

\medskip
In this paper, we revisit, extend and explore further this construction using a reversed approach.
Given a complete simplicial fan~$\Fan$, we consider the space~$\typeCone(\Fan)$ of all its polytopal realizations.
This space was called type cone in~\cite{McMullen-typeCone} and deformation cone in~\cite{Postnikov, PostnikovReinerWilliams}, who studied the case when~$\Fan$ is the braid arrangement leading to the rich theory of generalized permutahedra (recently extended to all finite Coxeter arrangements in~\cite{ArdilaCastilloEurPostnikov}).
The type cone is known to be a polyhedral cone defined by a collection of inequalities corresponding to the linear dependences among the rays of~$\Fan$ contained in pairs of adjacent maximal cones of~$\Fan$ (see \cref{def:typeCone}).
Our approach is based on an elementary but powerful observation: for any fan~$\Fan$, all polytopal realizations of~$\Fan$ can be described as sections of a high dimensional positive orthant with a collection of affine subspaces parametrized by the type cone~$\typeCone(\Fan)$ (see \cref{prop:alternativePolytopalRealization}); if moreover the type cone~$\typeCone(\Fan)$ is a simplicial cone, it leads to a simple parametrization of all polytopal realizations of~$\Fan$ by a positive orthant corresponding to the facets of the type cone~$\typeCone(\Fan)$ (see \cref{coro:simplicialTypeCone}).
To prove that the type cone~$\typeCone(\Fan)$ is simplicial, we just need to identify which pairs of adjacent maximal cones of~$\Fan$ correspond to the facets of~$\typeCone(\Fan)$ and to show that the corresponding linear dependences among their rays positively span the linear dependence among the rays of any pair of adjacent maximal cones of~$\Fan$.
When applied to the $\b{g}$-vector fans of cluster algebras (see \cref{subsec:typeConeCA}), the type cone approach yields all polytopal realizations of the $\b{g}$-vector fans and thus efficiently revisits and extends the construction of~\cite{BazierMatteDouvilleMousavandThomasYildirim}.
This new perspective has several advantages, as our proof uniformly~applies~to:

\smallskip
\begin{itemize}
\item \textbf{any initial seed}, regardless of whether it is acyclic or not (see \cref{exm:typeConeCA,fig:clusterFans,fig:labelFacetDefiningInequalititiesCA1,fig:labelFacetDefiningInequalititiesCA2} for examples in types~$A_3$ and~$C_3$~cyclic). In contrast, the proof of~\cite{BazierMatteDouvilleMousavandThomasYildirim} only treats acyclic initial seeds (although H.~Thomas announced in a lecture given at RIMS in June 2019 that this restriction will disappear in a future version of~\cite{BazierMatteDouvilleMousavandThomasYildirim}).

\smallskip
\item \textbf{all finite type cluster algebras}, regardless of whether it is simply-laced or not (see \cref{exm:typeConeCA} for an example in type~$C_3$). In contrast, even with the acyclicity restriction, the result of~\cite{BazierMatteDouvilleMousavandThomasYildirim} is first proved in simply-laced cases ($A$, $D$ and~$E$) using representation theory of quivers. Extension to the non simply-laced cases is then argued by a folding argument. Although not presented in detail, this argument is subtle and technical.

\smallskip
\item \textbf{any positive real-valued parameters}, regardless of whether they are rational or not. In contrast, the proof of~\cite{BazierMatteDouvilleMousavandThomasYildirim} naturally applies to rational parameters and then requires an approximation argument.
\end{itemize}

\smallskip
\noindent
These advantages all follow from one essential feature of the type cone approach.
Namely, it enables to completely separate the algebraic aspects from the geometric aspects of the problem:
\begin{itemize}
\item On the \textbf{algebraic side}, we study the relations between the $\b{g}$-vectors of a finite type cluster algebra and we show that the mesh mutations minimally generate these relations.
For this, we use representation theory in the setting of~$2$-Calabi--Yau triangulated categories, and we adapt M.~Auslander's proof that the relations in the Grothendieck group of an Artin algebra are generated by the ones given by almost-split sequences precisely when the algebra is of finite representation type~\cite{Auslander1984} (see \cref{sec:clusterCategories}).

\smallskip
\item On the \textbf{geometric side}, we use elementary manipulations to show that a fan whose type cone is simplicial is affinely equivalent to the normal fan of sections of a high dimensional positive orthant with well-chosen affine subspaces (see \cref{sec:typeCone}). We then observe that our algebraic result precisely states that all type cones of $\b{g}$-vector fans of cluster algebras are simplicial from which we derive all our geometric consequences (see \cref{sec:applications}).
\end{itemize}

\smallskip
\noindent
These two aspects are interesting in their own right and might be useful in different communities: the geometric side enables to simply describe all polytopal realizations of a fan with simplicial type cone (which is sometimes easier to check than to find an explicit realization), while the algebraic side provides fundamental new results on the $\b{g}$-vectors of cluster algebras and on (an extriangulated variant of) the Grothendieck groups of~$2$-Calabi--Yau triangulated categories with cluster tilting objects.
We have therefore deliberately split our presentation into two clearly marked parts that can be read independently.
The geometric side, which uses elementary techniques, is presented in \cref{part:geometry} together with all its consequences on polytopal realizations.
The algebraic side, which is of independent interest, is presented in \cref{part:algebra}. 

\medskip
\enlargethispage{.4cm}
Besides revisiting the construction of~\cite{ArkaniHamedBaiHeYan, BazierMatteDouvilleMousavandThomasYildirim} and extending it to any initial seed (acyclic or not) in any finite type cluster algebra (simply-laced or not), our type cone approach is also successful when applied to the $\b{g}$-vector fans of other families of generalizations of the associahedron.
In the present paper, we explore specifically \textbf{gentle associahedra}. These complexes where constructed in~\cite{PaluPilaudPlamondon-nonkissing} as polytopal realizations of finite non-kissing complexes of gentle bound quivers. They encompass two families of simplicial complexes studied independently in the literature: on the one hand the grid associahedra introduced by T.~K.~Petersen, P.~Pylyavskyy and D.~Speyer in~\cite{PetersenPylyavskyySpeyer} for a staircase shape, studied by F.~Santos, C.~Stump and V.~Welker~\cite{SantosStumpWelker} for rectangular shapes, and extended by T.~McConville in~\cite{McConville} for arbitrary grid shapes; and on the other hand the Stokes polytopes and accordion associahedra studied by Y.~Baryshnikov~\cite{Baryshnikov}, F.~Chapoton~\cite{Chapoton-quadrangulations}, A.~Garver and T.~McConville~\cite{GarverMcConville} and T.~Manneville and V.~Pilaud~\cite{MannevillePilaud-accordion}. The latter, sometimes called accordiohedra, have also~recently appeared in the mathematical physics litterature concerning scattering amplitudes~\cite{BanerjeeLaddhaRaman, Raman}, with realizations inspired by those in~\cite{ArkaniHamedBaiHeYan}. It was shown in~\cite{PaluPilaudPlamondon-nonkissing, BrustleDouvilleMousavandThomasYildirim} that the non-kissing complex of a gentle bound quiver provides a combinatorial model for the support $\tau$-tilting complex~\cite{AdachiIyamaReiten} of the associated gentle algebra. Even for finite non-kissing complexes, it turns out that the type cone of the non-kissing fan is not always simplicial (see \cref{exm:nonSimplicialTypeConeNKC}). We prove however that the type cone is always simplicial when the gentle bound quiver is brick (at least $2$ relations in any cycle) and $2$-acyclic (no oriented cycle of length~$2$). Moreover, we precisely describe the facet-defining inequalities of the type cone in terms of linear dependences of $\b{g}$-vectors involved in certain mesh relations. We thus automatically derive a construction similar to that of~\cite{ArkaniHamedBaiHeYan, BazierMatteDouvilleMousavandThomasYildirim} describing all polytopal realizations of the $\b{g}$-vector fans of these brick and $2$-acyclic gentle bound quivers as sections of a positive orthant by well-chosen affine spaces (see \cref{subsec:typeConeCA}). Again, this result mainly relies on finding which relations minimally generate the relations between the $\b{g}$-vectors in the non-kissing complex. We present both a purely combinatorial proof (see \cref{subsec:typeConeCA}) and a purely algebraic proof for (non-necessarily gentle) brick algebras making use of extriangulated categories (see \cref{sec:extricats}).
Our algebraic proof is based on an analogue of a result of M.~Auslander~\cite{Auslander1984} relating to minimal generators of Grothendieck groups.
Along the way, we need to generalize, for extriangulated categories, several results from cluster-tilting theory on indices~\cite{DehyKeller} and on abelian quotients~\cite{BuanMarshReiten, KellerReiten} that were known for triangulated categories. We note, however, that our assumptions are not quite the expected ones. This is related to the fact that the Grothendieck groups of cluster categories are badly behaved (\eg the Grothendieck group of a cluster category of type~$A_2$ is trivial). So as to overcome this difficulty, we replace the triangulated structure of cluster categories by some relative extriangulated structure, before considering Grothendieck groups. This explains why we consider projective objects rather than cluster-tilting objects for extriangulated categories.

\medskip
\enlargethispage{.5cm}
Further families of generalizations of the associahedron certainly deserve to be studied with our type cone approach.
We are investigating in particular graph associahedra\footnote{A preliminary version of this paper contained the full facet description of the type cones of graph associahedra. As they are not simplicial, they lie slightly beyond the scope of this paper, and will be studied elsewhere.}~\cite{CarrDevadoss, Postnikov, FeichtnerSturmfels, Zelevinsky}, brick polytopes~\cite{PilaudSantos-brickPolytope, PilaudStump-brickPolytope} and quotientopes~\cite{Reading-latticeCongruences, Reading-CambrianLattices, PilaudSantos-quotientopes}.
In fact, as already mentioned, it seems sometimes easier to find all polytopal realizations of a fan (by describing its type cone in terms of certain specific pairs of adjacent maximal cones) than to identify one specific explicit realization.
This is clearly the case when the type cone is simplicial as illustrated by the results of this paper: we naturally obtained all polytopal realizations of the $\b{g}$-vector fans of cluster algebras from cyclic seeds and of gentle algebras whose polytopality was only established very recently~\cite{HohlwegPilaudStella, PaluPilaudPlamondon-nonkissing}.
We believe that this approach gives reasonable hope that the polytopality of further fans could in the future be established using our type cone approach.
Tempting candidates are quotientopes for hyperplane arrangements beyond type~$A$, which are particularly interesting since the coarsenings of a fan essentially correspond to the faces of its type cone (see \cref{subsec:facesTypeCone}).

\medskip
Last but not least, the type cone approach naturally opens many research questions.
First, it raised the fundamental problem of describing the minimal relations among $\b{g}$-vectors in $2$-Calabi--Yau and extriangulated categories that we treat in \cref{part:algebra}.
It also motivates the study of relevant objects revealed by the type cone dictionary between geometry and algebra:

\smallskip
\begin{itemize}
\item The facet-defining inequalities of the type cone of~$\Fan$ correspond to very specific pairs of adjacent maximal cones of~$\Fan$ whose corresponding linear dependences minimally generate all such dependences. While they correspond to mesh mutations for cluster algebras and brick and $2$-acyclic gentle algebras, they are not understood in general (see for instance \cref{fig:labelFacetDefiningInequalititiesNKC}\,(right)). It would be particularly interesting to fully understand these specific mutations for arbitrary (non-kissing finite) gentle algebras.

\smallskip
\item The rays of the type cone of~$\Fan$ correspond to specific polytopes which form a positive Minkowski basis of all realizations of~$\Fan$ (see \cref{subsec:MinkowskiSums}). In particular, when the type cone is simplicial, any polytopal realization of~$\Fan$ has a unique representation (up to translation) as a Minkowski sum of positive dilates of these polytopes. For cluster algebras, it follows from~\cite[Sect.~6]{BazierMatteDouvilleMousavandThomasYildirim} that these polytopes are Newton polytopes of the $F$-polynomials. It would be interesting to obtain similar algebraic interpretations for the rays of the type cones of the non-kissing fans of gentle~algebras.

\smallskip
\item Once we get a representation of a polytope~$P$ as a signed Minkowski sum~$P = \sum_{i \in [k]} \alpha_i Q_i$, we get a formula for the volume of $P$ as a multivariate polynomial in $\alpha_1, \dots, \alpha_k$ whose coefficients are the mixed volumes of $Q_1, \dots, Q_k$~\cite{ArdilaBenedettiDoker, McMullen-Valuations}. These encode rich combinatorial information when representing generalized permutahedra as positive Minkowski sums of coordinate simplices or hypersimplices~\cite{Postnikov}, and in particular for matroid polytopes~\cite{ArdilaBenedettiDoker}. It would be interesting to get such volume formulas for polytopal realizations of $\b{g}$-vector fans of cluster algebras and gentle algebras, and combinatorial interpretations for their coefficients.
\end{itemize}

\smallskip
\noindent
Another question related to this paper is to understand the complete realization space of the combinatorics of the associahedra.
The associahedra constructed in~\cite{ArkaniHamedBaiHeYan, BazierMatteDouvilleMousavandThomasYildirim} and the present paper describe all polytopal realizations of the $\b{g}$-vector fans, but there are other polytopal realizations of the combinatorial type of the associahedron with different normal fans.
Is it possible to similarly understand the complete realization space of the cluster complex, including all realizations not having the~$\b{g}$-vector fan as their normal fan?

\medskip
\enlargethispage{.2cm}
To conclude, it is our hope that the present paper will participate to the interactions between combinatorial geometry, representation theory, and mathematical physics; in particular through further developments of the geometric approach to scattering amplitudes that has flourished during the last years~\cite{ArkaniHamedTrnka-Amplituhedron,ArkaniHamedBourjailyCachazoGoncharovPostnikovTrnka-GrassmannianGeometryScatteringAmplitudes}. The kinematic associahedron from~\cite{ArkaniHamedBaiHeYan} is one of the most recent ``positive geometries''~\cite{ArkaniHamedBaiLam-PositiveGeometries} that has emerged in the study of scattering amplitudes.


\section*{Structure of the paper: main results and logical dependencies}

We now give a more detailed overview of our main results and the structure of the article. In particular, some of the geometric realizations in~\cref{part:geometry} depend on algebraic results whose proofs are delayed until~\cref{part:algebra}. The goal of this section is to clarify their logical dependencies, summarized in the schematic diagram depicted at the end. 

\para{Type cone approach}
(\cref{sec:typeCone})

\noindent
Let~$\Fan$ be an essential complete polyhedral fan in~$\R^n$ (all definitions are recalled in~\cref{ssec:polytopes and fans}).
The type cone~$\typeCone(\Fan)$ of~$\Fan$ (\cref{ssec:type cone}), introduced by P. McMullen~\cite{McMullen-typeCone}, is a polyhedral cone that parametrizes the set of all possible polytopal realizations of~$\Fan$. These realizations can be described explicitly as affine sections of the non-negative orthant~$\R^N_{\geq0}$, where~$N$ is the number of rays of~$\Fan$ (\cref{prop:alternativePolytopalRealization}).  
If $\Fan$ is a simplicial fan and has the unique exchange relation property (\cref{def:uerp}), then its type cone is easier to describe and study.
And if moreover the type cone~$\typeCone(\Fan)$ is simplicial (\cref{rem:dimTypeCone}), then all possible polytopal realizations of~$\Fan$ can be described explicitely in terms of~$N-n$ positive real parameters as follows

\newtheorem*{coro:simplicialTypeCone}{\cref{coro:simplicialTypeCone}}
\begin{coro:simplicialTypeCone}
Assume that the type cone~$\typeCone(\Fan)$ is simplicial and let~$\b{K}$ be the $(N-n) \times N$-matrix whose rows are the inner normal vectors of the facets of~$\typeCone(\Fan)$. Then the polytope
\[
R_\b{\ell} \eqdef \set{\b{z} \in \R^N}{\b{K}\b{z} = \b{\ell} \text{ and } \b{z} \ge 0}
\]
is a realization of the fan~$\Fan$ for any positive vector~$\b{\ell} \in \R_{>0}^{N-n}$.
Moreover, the polytopes~$R_\b{\ell}$ for~$\b{\ell} \in \R_{>0}^{N-n}$ describe all polytopal realizations of~$\Fan$.
\end{coro:simplicialTypeCone}

This gives a new point of view on~\cite{ArkaniHamedBaiHeYan,BazierMatteDouvilleMousavandThomasYildirim}, illustrated in \cref{subsec:typeConeAsso} for classical associahedra, that we then apply to two families of~$\b{g}$-vector fans generalizing the~$\b{g}$-vector fans of classical associahedra: the~$\b{g}$-vector fans of cluster algebras of finite type (generalized associahedra) in~\cref{subsec:typeConeCA}, and the~$\b{g}$-vector fans of gentle algebras (non-kissing associahedra, non-crossing associahedra, generalized accordiohedra) in~\cref{subsec:typeConeNKC}.

\para{Cluster algebras of finite type and generalized associahedra}
(\cref{subsec:typeConeCA})

\noindent
We first consider the~$\b{g}$-vector fan~$\Fan(\B_\circ)$ of a skew-symmetrizable cluster algebra of finite type, with respect to any initial exchange matrix~$\B_\circ$ (acyclic or not). We provide the following polytopal realizations of~$\Fan(\B_\circ)$.

\newtheorem*{thm:allPolytopalRealizationsCA}{\cref{thm:allPolytopalRealizationsCA}}
\begin{thm:allPolytopalRealizationsCA}
For any~$\b{\ell} \in \R_{>0}^{\meshes}$, the polytope
\[
R_{\b{\ell}}(\B_\circ) \eqdef \Bigset{\b{z} \in \R^{\principalVariables}}{\b{z} \ge 0 \text{ and } \b{z}_x + \b{z}_{x'} - \sum_{y \in \principalVariables} \coefficient[y][x][x'] \, \b{z}_{y} = \b{\ell}_{\{x,x'\}} \text{ for all } \{x,x'\} \in \meshes}
\]
is a generalized associahedron, whose normal fan is the cluster fan~$\gvectorFan[\B_\circ]$.
Moreover, the polytopes~$R_\b{\ell}(\B_\circ)$ for~$\b{\ell} \in \R_{>0}^{\meshes}$ describe all polytopal realizations of~$\gvectorFan[\B_\circ]$.
\end{thm:allPolytopalRealizationsCA}

Here, $\meshes$ denotes the set of all pairs~$\{x,x'\}$ related by non-initial mesh mutations (\cref{def:meshMutation}), and~$\principalVariables$ the set of cluster variables.
For $\{x,x'\} \in \meshes$ and $y \in \principalVariables$, the coefficient $\coefficient[y][x][x']$ is defined as follows: $\coefficient[y][x][x']$ is set to be~$|b_{xy}|$ if~$y\in X\cap X'$ and $0$ otherwise, where~$X,X'$ are any two clusters such that~$X \ssm \{x\} = X' \ssm \{x'\}$. By the unique exchange relation property, the coefficients~$\coefficient[y][x][x']$ do not depend on the specific choice of clusters~$X,X'$.

\medskip

Some crucial pieces of the proof of~\cref{thm:allPolytopalRealizationsCA} use cluster categories (2-Calabi--Yau triangulated categories) and are deferred to~\cref{sec:clusterCategories}. Namely, a reformulation of~\cite[Thm.~7.5]{BuanMarshReinekeReitenTodorov}, proved in~\cref{sec:clusterCategories}, gives the first step towards applying our type cone strategy.

\newtheorem*{prop:uniqueExchangePropertyCA}{\cref{prop:uniqueExchangePropertyCA}}
\begin{prop:uniqueExchangePropertyCA}
The cluster fan~$\Fan(\B_\circ)$ has the unique exchange relation property.
\end{prop:uniqueExchangePropertyCA}

And the second step follows from our main result in \cref{sec:clusterCategories}, which makes use of representation theory of finite dimensional algebras and additive categorification of cluster algebras.

\newtheorem*{coro:simplicialTypeConeCA}{\cref{coro:simplicialTypeConeCA}}
\begin{coro:simplicialTypeConeCA}
The type cone~$\typeCone \big( \Fan(\B_\circ) \big)$ is simplicial.
\end{coro:simplicialTypeConeCA}

\para{Gentle algebras, non-kissing complexes, and generalized accordiohedra}
(\cref{subsec:typeConeNKC})

\noindent
The same strategy is applied to a class of finite dimensional algebras, called gentle algebras, whose representations are combinatorially well understood.
We note that the~$\tau$-tilting theory of gentle algebras is the algebraic counterpart of the combinatorics of accordions on surfaces~\cite{PaluPilaudPlamondon-nonkissing,BrustleDouvilleMousavandThomasYildirim,PaluPilaudPlamondon-surfaces}.
We identify two conditions, called brick and 2-acyclicity (\cref{def:brick2acyclic}), that enable us to make use of the type cone approach.

Let~$\quiver = (Q,I)$ be a gentle bound quiver, and let~$\Fan(\quiver)$ be the fan of its~$\b{g}$-vectors (\cref{subsubsec:nonkissingComplex}), also called non-kissing fan.
We provide the following polytopal realizations of~$\Fan(\quiver)$.

\newtheorem*{thm:allPolytopalRealizationsNKC}{\cref{thm:allPolytopalRealizationsNKC}}
\begin{thm:allPolytopalRealizationsNKC}
For any brick and $2$-acyclic gentle quiver~$\quiver$ and any~$\b{\ell} \in \R_{>0}^{\strings}$, the polytope
\[
R_\b{\ell}(\quiver) \eqdef \set{\b{z} \in \R^{\walks}}{\begin{array}{@{}l@{}} \b{z} \ge 0, \quad \b{z}_\omega = 0 \text{ for any straight or self-kissing walk } \omega, \\ \text{and } \b{z}_{\cc{\sigma}} + \b{z}_{\hh{\sigma}} - \b{z}_{\hc{\sigma}} - \b{z}_{\ch{\sigma}} = \b{\ell}_\sigma \text{ for all } \sigma \in \strings\end{array}}
\]
is a realization of the non-kissing fan~$\gvectorFan[\quiver]$.
Moreover, the polytopes~$R_\b{\ell}(\quiver)$ for~$\b{\ell} \in \R_{>0}^{\strings}$ describe all polytopal realizations of~$\gvectorFan[\quiver]$.
\end{thm:allPolytopalRealizationsNKC}

Here, $\strings$ is the set of all strings of~$\quiver$, and~$\walks$ the set of all walks on~$\quiver$ (\ie of all strings of the blossoming bound quiver of~$\quiver$ that join two blossom vertices).
If~$\sigma$ is a string of~$\quiver$, we make use of the evocative notations~$\cc{\sigma}, \hh{\sigma}, \hc{\sigma}, \ch{\sigma}$ for the walks obtained by adding, in the blossoming bound quiver, hooks or cohooks at each end of~$\sigma$.

\medskip

The two main ingredients are similar to those for cluster algebras.
As a first step, we obtain the following consequence of \cref{prop:exchangeablePairsNKC} (ii), which is proved in \cref{sec:Kbproj} by representation-theoretic methods.

\newtheorem*{prop:exchangeablePairsNKC}{\cref{prop:exchangeablePairsNKC}\,{\normalfont (iv)}}
\begin{prop:exchangeablePairsNKC}
The non-kissing fan~$\Fan(\quiver)$ has the unique exchange relation property.
\end{prop:exchangeablePairsNKC}

The second step is the following statement, for which we provide a purely combinatorial proof in \cref{subsubsec:simplicialTypeConeNKC} and an algebraic proof in \cref{sec:Kbproj}, using the main result of \cref{sec:extricats}.

\newtheorem*{coro:simplicialTypeConeNKC}{\cref{coro:simplicialTypeConeNKC}}
\begin{coro:simplicialTypeConeNKC}
If the gentle bound quiver~$\quiver$ is brick and 2-acyclic, then the type cone~$\typeCone \big( \Fan(\quiver) \big)$ is simplicial.
\end{coro:simplicialTypeConeNKC}

\para{Cluster categories}
(\cref{sec:clusterCategories})

\noindent
We prove an analogue for cluster categories of a result of M. Auslander for module categories of Artin algebras~\cite{Auslander1984}. 
We are mainly motivated by the fact that it implies that the type cones of the cluster fans of finite type are simplicial.
However, this result is of independent interest and sheds new lights on the Grothendieck groups of cluster categories.

Let~$\cat$ be a cluster category with finitely many isomorphism classes of indecomposable objects (see \cref{sec:setting} for the precise and more general setting in which the theorem is proven).
Let~$\Ksp(\cat)$ be the split Grothendieck group of~$\cat$.
Fix a cluster-tilting object~$T\in\cat$, and let $K_0(\cat ; T)$ be the quotient of~$\Ksp(\cat)$ by the relations~$[X]+[Z]-[Y]$ for all triangles ${X\xrightarrow{} Y \xrightarrow{} Z \xrightarrow{h} \susp X}$ with~$h \in (\susp T)$.
Denote by~$\b{g}:\Ksp(\cat) \to K_0(\cat ; T)$ the canonical projection.
For any indecomposable~$X\in\cat$, write~$\ell_X = [X]+[\susp^{-1}X]-[E]$, where~$E$ is the middle term of an almost split triangle starting at~$X$.
For any objects~$X,Y\in\cat$, we write~$\langle X, Y \rangle$ for~$\dim_\field\Hom{\Lambda}(FX,FY)$, where~$\Lambda$ is the cluster-tilted algebra~$\End{\cat}(T)$ and~$F$ is the equivalence of categories~$\cat(T,-):\cat/(\susp T) \to \MOD \Lambda$.

\newtheorem*{thm:relations-g-vecteurs}{\cref{thm:relations-g-vecteurs}}
\begin{thm:relations-g-vecteurs}
The set~$L_\cat \eqdef \set{\ell_X}{X \in \ind(\cat) \ssm \add(\susp T)}$ is a basis of the kernel of~$\b{g}$ and, for any~$x \in \ker(\b{g})$, we have
\[
x = \sum_{A \in \ind(\cat) \ssm \add(\susp T)} \frac{\langle x, A \rangle}{\langle \ell_A, A \rangle} \ell_A.
\]
\end{thm:relations-g-vecteurs}

\para{Extriangulated categories}
(\cref{sec:extricats})

\noindent
We prove a more general version of \cref{thm:relations-g-vecteurs} that applies not only to cluster fans, but also to non-kissing fans.
Once more, our main motivation is to prove that the type cone of the non-kissing fan of a brick and 2-acyclic gentle algebra is simplicial, but the results are of independent interest, and might be useful for studying the type cones of other families of simplicial fans.

Let~$\cat$ be an extriangulated category (\cref{sec:recollections extricats}) with a fixed projective object~$T$ such that the morphism~$T\to 0$ is an inflation and, for each~$X\in\cat$, there is an extriangle~$T_1^X\infl T_0^X\defl X\overset{\delta_X}{\dashrightarrow}$.
Let~$\susp T$ be the cone of the inflation~$T\infl 0$.

\newtheorem*{prop:KRextricat}{\cref{prop:KRextricat}}
\begin{prop:KRextricat}
The functor~$\cat(T,-)$ induces an equivalence~$F:\cat/(\susp T) \to \MOD \End{\cat}(T)$.
\end{prop:KRextricat}

Our assumptions allow for a well-defined notion of index:~$\ind_T(X)=[T_0^X]-[T_1^X] \in \kzero{\add T}$.

\newtheorem*{prop:index iso extricat}{\cref{prop:index iso extricat}}
\begin{prop:index iso extricat}
The index induces an isomorphism of abelian groups from~$\kzero{\cat}$ to~$\kzero{\add T}$.
\end{prop:index iso extricat}

The previous two statements also hold when~$T$ is an additive subcategory rather than a mere object.
Assume moreover that~$\cat$ is~$\field$-linear, Ext-finite, Krull--Schmidt and has Auslander--Reiten--Serre duality.
Then, with notations analogous to those used for cluster categories above, we obtain the following statement.

\newtheorem*{thm:extricats}{\cref{thm:extricats}}
\begin{thm:extricats}
The extriangulated category~$\cat$ has only finitely many isomorphism classes of indecomposable objects if and only if the set~$L_\cat$ generates the kernel of the canonical projection ${\b{g}:\Ksp(\cat)\to\kzero{\cat}}$.
In that case~$L_\cat$ is a basis of the kernel of~$\b{g}$ and any~$x\in\ker(\b{g})$ decomposes~as
\[
 x = \sum_{A \in \ind(\cat) \ssm \add(\susp T)} \frac{\langle x, A \rangle}{\langle \ell_A, A \rangle} \ell_A.
\]
\end{thm:extricats}

\para{Schematic diagram of logical dependencies}

\noindent
We conclude with a schematic diagram representing the logical dependencies between the different sections of the paper.

\vspace{.5cm}
\tikzset{
  double arrow/.style={
    -latex, line width=3pt, black, 
    postaction={draw,line width=2pt, white, shorten <= 1,shorten >= 2}, 
  }
}
\centerline{
\begin{tikzpicture}[every node/.style={inner sep=6, outer sep=8}]
	\node[diamond,aspect=3, text width=5cm, minimum height=1cm, align=center, draw=blue, fill=blue!10, inner sep=-2] (I1) at (0,1.5) {\ref{part:geometry}.\ref{sec:typeCone} -- Simplicial type cone strategy};
	\node[rectangle, draw=black, fill=black!10, align=center, minimum width = 5.3cm, minimum height=1.2cm] (I21) at (-6,-1.5) {\ref{part:geometry}.\ref{subsec:typeConeAsso} -- Toy example \\ Classical associahedra};
	\node[rectangle, draw=orange, fill=orange!10, align=center, text width=5.3cm, minimum width = 5cm, minimum height=1.2cm] (I22) at (0,-1.5) {\ref{part:geometry}.\ref{subsec:typeConeCA} -- Cluster algebras \\ \& generalized associahedra};
	\node[rectangle, draw=green, fill=green!10, align=center, text width=5.3cm, minimum width = 5cm, minimum height=1.2cm] (I23) at (6,-1.5) {\ref{part:geometry}.\ref{subsec:typeConeNKC} -- Gentle algebras \\ \& generalized accordiohedra};
	\node[ellipse, draw=orange, fill=orange!10, align=center, text width=5cm, inner sep=2, minimum height=1cm] (II1) at (-4.5,-4.2) {\ref{part:algebra}.\ref{sec:clusterCategories} -- Cluster categories};
	\node[ellipse, draw=green, fill=green!10, align=center, text width=5cm, inner sep=0, minimum height=1cm] (II2) at (4.5,-4.2) {\ref{part:algebra}.\ref{sec:extricats} -- Extriangulated categories};
	\draw[decorate, decoration={shape backgrounds, shape=dart, shape size=2mm, shape evenly spread=8}, draw=blue, fill=blue!10] (I1.180) to [bend right, looseness=1.1] node[label={[shift={(-1.6,-1.5)}]\textit{applied to}}]{} (I21.90);
	\draw[decorate, decoration={shape backgrounds, shape=dart, shape size=2mm, shape evenly spread=3}, draw=blue, fill=blue!10] (I1.270) to node[label={[shift={(1,-1.)}]\textit{applied to}}]{} (I22.90);
	\draw[decorate, decoration={shape backgrounds, shape=dart, shape size=2mm, shape evenly spread=8}, draw=blue, fill=blue!10] (I1.360) to [bend left, looseness=1.1] node[label={[shift={(1.6,-1.5)}]\textit{applied to}}]{} (I23.90);
	\draw[thick, double distance=2pt, -implies, in=270] (I22.270)+(-1.5,-1.2) to node[semithick, right, outer sep=0, pos=0.55]{\textit{used for}} ([yshift=.2cm] I22.270);
	\draw[thick, double distance=2pt, -implies, in=270] (I23.270)+(-1.5,-1.2) to node[semithick, right, outer sep=0, pos=0.55]{\textit{used for}} ([yshift=.2cm] I23.270);
	\draw[thick, ->] (II1.0) to node[semithick, below, outer sep=0]{\textit{generalizes to}} (II2.180);
	\draw let \p{top}=(I1.90), \p{left}=(I21.180), \p{bottom}=([shift={(0,-.15)}]I22.270), \p{right}=(I23.0) in (\x{left}, \y{top}) -- (\x{left}, \y{bottom}) -- (\x{right}, \y{bottom}) -- (\x{right}, \y{top}) -- (\x{left}, \y{top});
	\draw let \p{top}=([shift={(0,.2)}]II1.90), \p{left}=(I21.180), \p{bottom}=(II1.270), \p{right}=(I23.0) in (\x{left}, \y{top}) -- (\x{left}, \y{bottom}) -- (\x{right}, \y{bottom}) -- (\x{right}, \y{top}) -- (\x{left}, \y{top});
	\path let \p{top}=(I1.90) in node[rectangle, draw, fill=white, align=left, inner sep=4] (I) at (-6.15, \y{top}) {\ref{part:geometry} -- Discrete geometry};
	\path let \p{top}=([shift={(0,.2)}]II1.90) in node[rectangle, draw, fill=white, align=left, inner sep=4] (I) at (-5.8, \y{top}) {\ref{part:algebra} -- Representation theory};
\end{tikzpicture}
}


\newpage
\addtocontents{toc}{\vspace{-.2cm}}
\part{Type cones of $\b{g}$-vector fans}
\label{part:geometry}


\section{Polytopal realizations and type cone of a simplicial fan}
\label{sec:typeCone}


\subsection{Polytopes and fans}
\label{ssec:polytopes and fans}

We briefly recall basic definitions and properties of polyhedral fans and polytopes, and refer to~\cite{Ziegler-polytopes} for a classical textbook on this topic.

A hyperplane~$H \subset \R^n$ is a \defn{supporting hyperplane} of a set~$X \subset \R^n$ if~$H \cap X \ne \varnothing$ and~$X$ is contained in one of the two closed half-spaces of~$\R^n$ defined by~$H$.

We denote by~$\R_{\ge0}\rays \eqdef \set{\sum_{\ray \in \rays} \lambda_{\ray} \, \ray}{\lambda_{\ray} \in \R_{\ge0}}$ the \defn{positive span} of a set~$\rays$ of vectors of~$\R^n$.
A \defn{polyhedral cone} is a subset of~$\R^n$ defined equivalently as the positive span of finitely many vectors or as the intersection of finitely many closed linear halfspaces.
The \defn{faces} of a cone~$C$ are the intersections of~$C$ with the supporting hyperplanes of~$C$.
The $1$-dimensional (resp.~codimension~$1$) faces of~$C$ are called~\defn{rays} (resp.~\defn{facets}) of~$C$.
A cone is \defn{simplicial} if it is generated by a set of linearly independent vectors.

A \defn{polyhedral fan} is a collection~$\Fan$ of polyhedral cones such that
\begin{itemize}
\item if~$C \in \Fan$ and~$F$ is a face of~$C$, then~$F \in \Fan$,
\item the intersection of any two cones of~$\Fan$ is a face of both.
\end{itemize}
A fan is \defn{simplicial} if all its cones are simplicial, \defn{complete} if the union of its cones covers the ambient space~$\R^n$, and \defn{essential} if it contains the cone~$\{\b{0}\}$. Note that every complete fan is the product of an essential fan with its lineality space (the largest linear subspace contained in all the cones).
For two fans~$\Fan, \cal{G}$ in~$\R^n$, we say that~$\Fan$ \defn{refines}~$\cal{G}$ (and that~$\cal{G}$ \defn{coarsens}~$\Fan$) if every cone of~$\Fan$ is contained in a cone of~$\cal{G}$.

A \defn{polytope} is a subset~$P$ of~$\R^n$ defined equivalently as the convex hull of finitely many points or as a bounded intersection of finitely many closed affine halfspaces.
The \defn{dimension}~$\dim(P)$ is the dimension of the affine hull of~$P$.
The \defn{faces} of~$P$ are the intersections of~$P$ with its supporting hyperplanes.
The dimension~$0$ (resp.~dimension~$1$, resp.~codimension~$1$) faces are called \defn{vertices} (resp.~\defn{edges}, resp.~\defn{facets}) of~$P$.
A polytope is \defn{simple} if each vertex is incident to $\dim(P)$ facets (or equivalently to $\dim(P)$ edges).

The (outer) \defn{normal cone} of a face~$F$ of~$P$ is the cone generated by the outer normal vectors of the facets of~$P$ containing~$F$.
In other words, it is the cone of vectors~$\b{c}$ such that the linear form~${\b{x} \mapsto \dotprod{\b{c}}{\b{x}}}$ on~$P$ is maximized by all points of the face~$F$.
The (outer) \defn{normal fan} of~$P$ is the collection of the (outer) normal cones of all its faces.
We say that a complete polyhedral fan~$\Fan$ in~$\R^n$ is \defn{polytopal} when it is the normal fan of a polytope~$P$ of~$\R^n$, and that~$P$ is a \defn{polytopal realization} of~$\Fan$.


\subsection{Type cone}
\label{ssec:type cone}

Fix an essential complete simplicial fan~$\Fan$ in~$\R^n$, with an (arbitrary) indexing of its rays by $[N] \eqdef \{1,\dots,N\}$. Let~$\b{G}$ be the $N \times n$-matrix whose rows are representative vectors of the rays of~$\Fan$. Let~$\b{K}$ be a $(N-n) \times N$-matrix that spans the left kernel of~$\b{G}$ (\ie $\b{K}\b{G} = 0$ and $\rank(\b{K})=N-n$). For any height vector~$\b{h} \in \R^N$, we define the polytope
\[
P_\b{h} \eqdef \set{\b{x} \in \R^n}{\b{G}\b{x} \le \b{h}},
\]
where $\b{G}\b{x} \le \b{h}$ is the standard shorthand for the corresponding system of inequalities.

We say that $\b{h}$ is \defn{$\Fan$-admissible} if $P_\b{h}$ is a polytopal realization of~$\Fan$.
The following classical statement characterizes the $\Fan$-admissible height vectors.
It is a reformulation of regularity of triangulations of vector configurations, introduced in the theory of secondary polytopes~\cite{GelfandKapranovZelevinsky}, see also~\cite{DeLoeraRambauSantos}.
We present here a convenient formulation from~\cite[Lem.~2.1]{ChapotonFominZelevinsky}.

\begin{proposition}
\label{prop:characterizationPolytopalFan}
Let~$\Fan$ be an essential complete simplicial fan in~$\R^n$. Then the following are equivalent for any height vector~$\b{h}\in \R^N$:
\begin{enumerate}
\item The fan~$\Fan$ is the normal fan of the polytope~$P_\b{h} \eqdef \set{\b{x} \in \R^n}{\b{G}\b{x} \le \b{h}}$.
\item For any two adjacent maximal cones~$\R_{\ge0}\rays$ and~$\R_{\ge0}\rays'$ of~$\Fan$ with~$\rays \ssm \{\ray\} = \rays' \ssm \{\ray'\}$,
\[
\alpha \, \b{h}_{\ray} + \alpha' \, \b{h}_{\ray'} + \sum_{\ray[s] \in \rays \cap \rays'} \beta_{\ray[s]} \, \b{h}_{\ray[s]} > 0,
\]
where
\[
\alpha \, \ray + \alpha' \, \ray' + \sum_{\ray[s] \in \rays \cap \rays'} \beta_{\ray[s]} \, \ray[s] = 0
\]
is the unique (up to rescaling) linear dependence with~$\alpha, \alpha' > 0$ between the rays of~$\rays \cup \rays'$.
\end{enumerate}
\end{proposition}

\begin{notation}
For any adjacent maximal cones~$\R_{\ge0}\rays$ and~$\R_{\ge0}\rays'$ of~$\Fan$ with ${\rays \ssm \{\ray\} = \rays' \ssm \{\ray'\}}$, we denote by~$\coefficient[{\ray[s]}][\rays][\rays']$ the coefficient of~$\ray[s]$ in the unique linear dependence between the rays of~$\rays \cup \rays'$, \ie such that
\[
\sum_{\ray[s] \in \rays \cup \rays'} \coefficient[{\ray[s]}][\rays][\rays'] \, \ray[s] = 0.
\]
These coefficients are \apriori{} defined up to rescaling, but we additionally fix the rescaling so that~${\coefficient[\ray][\rays][\rays'] + \coefficient[\ray'][\rays][\rays'] = 2}$ (this convention is arbitrary, but will be convenient in \cref{sec:applications}).
\end{notation}

In this paper, we are interested in the set of all possible realizations of~$\Fan$ as the normal fan of a polytope~$P_\b{h}$. This was studied by P.~McMullen in~\cite{McMullen-typeCone} (see~\cite[Sect.~9.5]{DeLoeraRambauSantos} for a formulation in terms of chambers of triangulations of vector configurations).

\begin{definition}
\label{def:typeCone}
The \defn{type cone} of~$\Fan$ is the cone~$\typeCone(\Fan)$ of all $\Fan$-admissible height vectors~$\b{h} \in \R^N$:
\begin{align*}
\typeCone(\Fan) & \eqdef \set{\b{h} \in \R^N}{\Fan \text{ is the normal fan of } P_\b{h}} \\
& = \Bigset{\b{h} \in \R^N}{\sum_{\ray[s] \in \rays \cup \rays'} \coefficient[{\ray[s]}][\rays][\rays'] \, \b{h}_{\ray[s]} > 0 \; \begin{array}{l} \text{for any adjacent maximal} \\ \text{cones~$\R_{\ge0}\rays$ and~$\R_{\ge0}\rays'$ of~$\Fan$} \end{array}}.
\end{align*}
\end{definition}

Note that the type cone is an open cone and contains a lineality subspace of dimension~$n$ (it is invariant by translation in~$\b{G} \R^n$). 
It is sometimes useful to get rid of the lineality space by considering the projection~$\b{K}\typeCone(\Fan)$. 

We denote by $\ctypeCone(\Fan)$ the closure of $\typeCone(\Fan)$, and call it the \defn{closed type cone} of~$\Fan$. It is the closed polyhedral cone defined by the inequalities $\sum_{\ray[s] \in \rays \cup \rays'} \coefficient[{\ray[s]}][\rays][\rays'] \, \b{h}_{\ray[s]} \ge 0$ for any adjacent maximal cones~$\R_{\ge0}\rays$ and~$\R_{\ge0}\rays'$. If $\Fan$ is the normal fan of the polytope~$P$, then $\ctypeCone(\Fan)$ is called the \defn{deformation cone} of $P$ in~\cite{Postnikov}, see also~\cite{PostnikovReinerWilliams}.

\begin{example}
\label{exm:typeCone}
\begin{figure}[b]
	\capstart
	\centerline{{\includegraphics[scale=1]{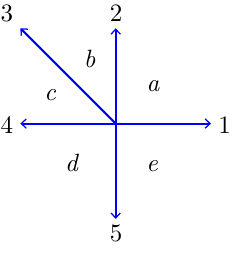}} \quad {\includegraphics[scale=1]{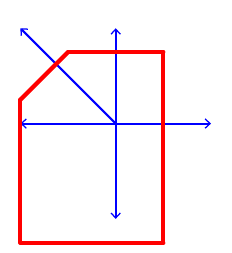}} \quad \raisebox{.5cm}{\includegraphics[scale=.45]{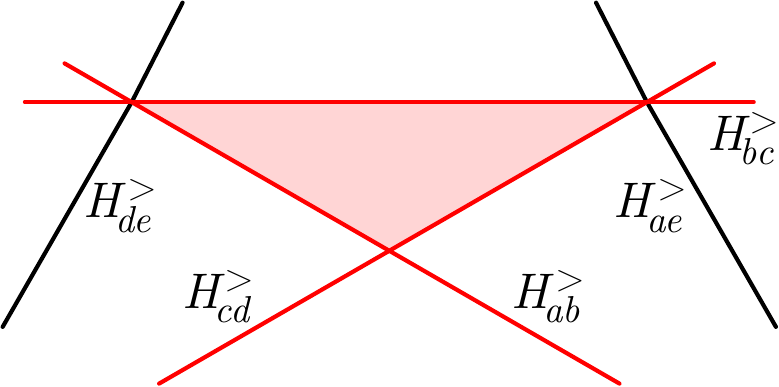}}}
	\caption{A $2$-dimensional fan~$\Fan$ with five rays $1, \dots, 5$ and five maximal cones~$a, \dots, e$ (left), its polytopal realization corresponding to the height vector~$(\nicefrac{1}{2}, \nicefrac{3}{4}, 1, 1, \nicefrac{5}{4})$ (middle), and the intersection of $\b{K}\typeCone(\Fan)$ with a hyperplane to get a $2$-dimensional slice (right).}
	\label{fig:typeCone}
\end{figure}
Consider the $2$-dimensional fan~$\Fan$ depicted in \cref{fig:typeCone} (left) (we will encounter this fan again in \cref{exm:typeConeCoarsenings,exm:typeConeMinkowskiSum,exm:typeConeAsso}).
It has five rays labeled~$1, \dots, 5$ and five maximal cones labeled~$a,\dots, e$.
For the matrices $\b{G}$ and $\b{K}$, we consider
\[
\b{G} = \begin{bmatrix} 1 & 0 \\ 0 & 1 \\ -1 & 1 \\ -1 & 0 \\ 0 & -1 \end{bmatrix}
\qquad\text{and}\qquad
\b{K} = \begin{bmatrix} 1 & 0 & 0 & 1 & 0 \\ 0 & 1 & 0 & 0 & 1 \\ 1 & 0 & 1 & 0 & 1 \end{bmatrix}.
\]
The type cone~$\typeCone(\Fan)$ lies in~$\R^5$, but has a $2$-dimensional lineality space.
The five pairs of adjacent maximal cones of~$\Fan$ give rise to following five defining inequalities for~$\typeCone(\Fan)$:
\begin{gather*}
H^>_{ab}: \; \b{h}_1 + \b{h}_3 - \b{h}_2 > 0
\qquad
H^>_{bc}: \; \b{h}_2 + \b{h}_4 - \b{h}_3 > 0
\qquad
H^>_{cd}: \; \b{h}_3 + \b{h}_5 - \b{h}_4 > 0
\\
H^>_{de}: \; \b{h}_1 + \b{h}_4 > 0
\qquad
H^>_{ae}: \; \b{h}_2 + \b{h}_5 > 0,
\end{gather*}
where $H^>_{ij}$ denotes the halfspace defined by the inequality corresponding to the two adjacent maximal cones~$i$ and~$j$.
Note that the inequalities~$H^>_{ae}$ and~$H^>_{de}$ are redundant.
For example, the height vector~$(\nicefrac{1}{2}, \nicefrac{3}{4}, \nicefrac{5}{4}, 1, \nicefrac{5}{4})$ belongs to the type cone~$\typeCone(\Fan)$, and the corresponding polytope is represented in \cref{fig:typeCone} (middle).
To represent this type cone, it suffices to represent~$\b{K}\typeCone(\Fan)$ which is an essential simplicial cone in~$\R^3$, given as the intersection of the following five open halfspaces:
\begin{gather*}
\b{K}H^>_{ab}: \; {x} - {y} +2 {z} > 0
\qquad
\b{K}H^>_{bc}: \; {x} + {y} - {z} > 0
\qquad
\b{K}H^>_{cd}: \; -{x} + {y} +2 {z} > 0
\\
\b{K}H^>_{de}: \; 2{x} + {z} > 0
\qquad
\b{K}H^>_{ae}: \; 2{y} + {z} > 0.
\end{gather*}
One can further reduce the dimension by intersecting with a transversal hyperplane to get a $2$-dimensional representation.
This is the red triangle depicted in \cref{fig:typeCone} (right).
\end{example}

\begin{definition}
An \defn{extremal adjacent pair} of~$\Fan$ is a pair of adjacent maximal cones~$\{\R_{\ge0}\rays, \R_{\ge0}\rays'\}$ of~$\Fan$ such that the corresponding inequality $\sum_{\ray[s] \in \rays \cup \rays'} \coefficient[{\ray[s]}][\rays][\rays'] \, \b{h}_{\ray[s]} \ge 0$ in the definition of the type cone~$\ctypeCone(\Fan)$ actually defines a facet of~$\ctypeCone(\Fan)$.
\end{definition}

In other words, extremal adjacent pairs define the extremal rays of the polar of the closed type cone~$\ctypeCone(\Fan)$.
Understanding the extremal adjacent pairs of~$\Fan$ enables to describe its type cone~$\typeCone(\Fan)$ and thus all its polytopal realizations.
For instance, for the $2$-dimensional fan of \cref{exm:typeCone}, the extremal adjacent pairs are~$\{a,b\}$, $\{b,c\}$ and~$\{c,d\}$.

\begin{remark}
\label{rem:dimTypeCone}
Since the type cone is an~$N$-dimensional cone with a lineality subspace of dimension~$n$, it has at least~$N-n$ facets (thus~$N-n$ extremal adjacent pairs).
The type cone is \defn{simplicial} when it has precisely $N-n$ facets.
\end{remark}

Although we will only deal with simplicial fans, note for completeness that \cref{prop:characterizationPolytopalFan,def:typeCone} can be easily adapted to non-simplicial fans. To describe the type cone of an arbitrary complete fan, it suffices to consider any simplicial refinement, and to set some of the strict inequalities of the definition of the type cone to equality (see~\cite[Prop.~3]{PilaudSantos-quotientopes} for a proof).

\begin{proposition}
\label{prop:non-simplicial}
Let $\Fan$ be a complete fan that coarsens the essential complete simplicial fan~$\Fan[G]$.
The \defn{type cone}~$\typeCone(\Fan)$ of all $\Fan$-admissible height vectors~$\b{h} \in \R^N$ is
\begin{align*}
\typeCone(\Fan) 
& = \left\{\b{h} \in \R^N\;\middle|\;
\begin{array}{ll}
\sum_{\ray[s] \in \rays \cup \rays'} \coefficient[{\ray[s]}][\rays][\rays'] \, \b{h}_{\ray[s]} = 0
& \begin{array}{l} \text{for any adjacent maximal cones} \\ \text{$\R_{\ge0}\rays$ and~$\R_{\ge0}\rays'$ of~$\Fan[G]$ belonging}\\\text{to \textbf{the same} maximal cone of $\Fan$,}\end{array}
\\[.7cm]
\sum_{\ray[s] \in \rays \cup \rays'} \coefficient[{\ray[s]}][\rays][\rays'] \, \b{h}_{\ray[s]} > 0
& \begin{array}{l} \text{for any adjacent maximal cones} \\ \text{$\R_{\ge0}\rays$ and~$\R_{\ge0}\rays'$ of~$\Fan[G]$ belonging}\\\text{to \textbf{distinct} maximal cones of $\Fan$}\end{array}
\end{array}
\right\}.
\end{align*}
\end{proposition}


\subsection{Unique exchange relation property}
\label{subsec:uniqueExchangeProperty}

Two rays~$\ray$ and~$\ray'$ of~$\Fan$ are \defn{exchangeable} if there are two adjacent maximal cones~$\R_{\ge0}\rays$ and~$\R_{\ge0}\rays'$ of~$\Fan$ with ${\rays \ssm \{\ray\} = \rays' \ssm \{\ray'\}}$.

\begin{definition}
\label{def:uerp}
We say that two exchangeable rays~$\ray, \ray'$ of~$\Fan$ admit a \defn{unique exchange relation} when the linear dependence
\[
\sum_{\ray[s] \in \rays \cup \rays'} \coefficient[{\ray[s]}][\rays][\rays'] \, \ray[s] = 0.
\]
does not depend on the pair~$\{\rays, \rays'\}$ of adjacent maximal cones of~$\Fan$ with $\rays \ssm \{\ray\} = \rays' \ssm \{\ray'\}$, but only on the pair of rays~$\{\ray, \ray'\}$.
This implies in particular that the rays~$\ray[s]$ for which~$\coefficient[{\ray[s]}][\rays][\rays'] \ne 0$ belong to~$\rays \cup \rays'$ for any pair~$\{\rays, \rays'\}$ of adjacent maximal cones of~$\Fan$ with $\rays \ssm \{\ray\} = \rays' \ssm \{\ray'\}$. 

We say that the fan~$\Fan$ has the \defn{unique exchange relation property} if any two exchangeable rays~$\ray, \ray'$ of~$\Fan$ admit a unique exchange relation.
\end{definition}

When~$\Fan$ has the unique exchange relation property, we change the notation~$\coefficient[{\ray[s]}][\rays][\rays']$ to~$\coefficient$ and we obtain that the type cone of~$\Fan$ is expressed as
\[
\typeCone(\Fan) = \Bigset{\b{h} \in \R^N}{\sum_{\ray[s]} \coefficient \, \b{h}_{\ray[s]} > 0 \text{ for any exchangeable rays~$\ray$ and~$\ray'$ of~$\Fan$}}.
\]

\begin{definition}
In a fan~$\Fan$ with the unique exchange relation property, an \defn{extremal exchangeable pair} is a pair of exchangeable rays~$\{\ray,\ray'\}$ such that the corresponding inequality~${\sum_{\ray[s]} \coefficient \, \b{h}_{\ray[s]} \ge  0}$ defines a facet of the closed type cone~$\ctypeCone(\Fan)$.
\end{definition}

In this paper, we will only consider fans with the unique exchange relation property, and our objective will be to describe their extremal exchangeable pairs.


\subsection{Alternative polytopal realizations}

In this section, we provide alternative polytopal realizations of the fan~$\Fan$.
We also discuss the behavior of these realizations in the situation when the type cone~$\typeCone(\Fan)$ is simplicial.

We still consider an essential complete simplicial fan~$\Fan$ in~$\R^n$, the $N \times n$-matrix~$\b{G}$ whose rows are the rays of~$\Fan$, and an $(N-n) \times N$-matrix~$\b{K}$ that spans the left kernel of~$\b{G}$ (\ie $\b{K}\b{G} = 0$ and $\rank(\b{K})=N-n$).

\begin{proposition}
\label{prop:alternativePolytopalRealization}
The affine map~$\Psi: \R^n \to \R^N$ defined by~$\Psi(\b{x}) = \b{h} - \b{G}\b{x}$ sends the polytope
\[
P_\b{h} \eqdef \set{\b{x} \in \R^n}{\b{G}\b{x} \le \b{h}}
\]
to the polytope
\[
Q_\b{h} \eqdef \set{\b{z} \in \R^N}{\b{K}\b{z} = \b{K}\b{h} \text{ and } \b{z} \ge 0}.
\]
\end{proposition}

\begin{proof}
This result is standard and is proved for instance in~\cite[Coro.~9.5.7]{DeLoeraRambauSantos}.
We include a short argument here for the convenience of the reader.
For~$\b{x}$ in~$P_\b{h}$, we have $\Psi(\b{x}) \ge 0$ by definition and~$\b{K}\Psi(\b{x}) = \b{K}\b{h} - \b{K}\b{G}\b{x} = \b{K}\b{h}$ since~$\b{K}$ is the left kernel of~$\b{G}$. Therefore~$\Psi(\b{x}) \in Q_\b{h}$.
Moreover, the map~$\Psi : P_\b{h} \to Q_\b{h}$ is:
\begin{itemize}
\item injective: Indeed, $\Psi(\b{x}) = \Psi(\b{x}')$ implies~$\b{G}(\b{x} - \b{x}') = 0$ and~$\b{G}$ has full rank since~$\Fan$ is essential and complete.
\item surjective: Indeed, for~$\b{z} \in Q_\b{h}$, we have~$\b{K}(\b{h}-\b{z}) = 0$ so that~$\b{h}-\b{z}$ belongs to the right kernel of~$\b{K}$, which is the image of~$\b{G}$ (because $\b{G}$ is of full rank).
Hence, there exists~$\b{x} \in \R^n$ such that~$\b{h}-\b{z} = \b{G}\b{x}$. We have $\b{z} = \Psi(\b{x})$ and~$\b{x} \in P_\b{h}$ since~$\b{h} - \b{G}\b{x} = \b{z} \ge 0$.\qedhere
\end{itemize}
\end{proof}

Note that~$\b{K}$ is only defined up to linear transformation, as it depends on the choice of a basis of the left kernel of~$\b{G}$. Note also that, by construction, the coefficients of the normal vectors of the facets of~$\typeCone(\Fan)$ arise from a linear dependence among the rays of~$\Fan$, and hence belong to this kernel. When~$\typeCone(\Fan)$ is simplicial, its $(N-n)$ facets are necessarily linearly independent, and hence form a basis. They provide a particularly interesting choice of~$\b{K}$.  

\begin{corollary}
\label{coro:simplicialTypeCone}
Assume that the type cone~$\typeCone(\Fan)$ is simplicial and let~$\b{K}$ be the $(N-n) \times N$-matrix whose rows are the inner normal vectors of the facets of~$\typeCone(\Fan)$. Then the polytope
\[
R_\b{\ell} \eqdef \set{\b{z} \in \R^N}{\b{K}\b{z} = \b{\ell} \text{ and } \b{z} \ge 0}
\]
is a realization of the fan~$\Fan$ for any positive vector~$\b{\ell} \in \R_{>0}^{N-n}$.
Moreover, the polytopes~$R_\b{\ell}$ for~$\b{\ell} \in \R_{>0}^{N-n}$ describe all polytopal realizations of~$\Fan$.
\end{corollary}

\begin{proof}
Let~$\b{\ell} \in \R_{>0}^{N-n}$.
Since~$\b{K}$ has full rank there exists~$\b{h} \in \R^N$ such that~$\b{K}\b{h} = \b{\ell}$.
Since~$\b{K}\b{h} \ge 0$ and the rows of~$\b{K}$ are precisely all inner normal vectors of the facets of the type cone~$\typeCone(\Fan)$, we obtain that~$\b{h}$ belongs to~$\typeCone(\Fan)$.
Since $R_\b{\ell} = Q_\b{h} \sim P_\b{h}$ by \cref{prop:alternativePolytopalRealization}, we conclude that~$R_\b{\ell}$ is a polytopal realization of~$\Fan$.
Since~$\typeCone(\Fan)$ is simplicial, we have~$\b{K}\typeCone(\Fan) = \R_{>0}^{N-n}$, so that we obtain all polytopal realizations of~$\Fan$ in this way.
\end{proof}


\subsection{Faces of the type cone}
\label{subsec:facesTypeCone}

We now relate the faces of the closed type cone~$\ctypeCone(\Fan)$ with the coarsenings of the fan~$\Fan$.

\begin{theorem}
Let $\Fan$ be a polytopal fan. Then the relatively open faces of the closed type cone~$\ctypeCone(\Fan)$ are the type cones of all (polytopal) coarsenings of~$\Fan$. 
\end{theorem}

\begin{proof}
Let $\cal{G}$ be a simplicial refinement of $\Fan$. By \cref{prop:non-simplicial}, the type cone of any coarsening of $\Fan$ is a relatively open cell of the hyperplane arrangement defined by adjacent pairs of~$\cal{G}$. The type cone of the coarsening is not empty whenever it is polytopal. When non-empty, distinct coarsenings give rise to distinct cells (one will have an equality that is a strict inequality in the other), and all lie in the closure of~$\typeCone(\Fan)$ (since they are all in the closure of all the defining half-spaces).
Thus every coarsening describes a relatively open face of~$\ctypeCone(\Fan)$.
 
For the converse, observe that for $\b{h}\in\ctypeCone(\Fan)$, the normal fan of $P_{\b{h}}$ is the polytopal complete fan obtained by merging all pairs of adjacent maximal cones~$\R_{\ge0}\rays$ and~$\R_{\ge0}\rays'$ of~$\Fan$ such that $\sum_{\ray[s] \in \rays \cup \rays'} \coefficient[{\ray[s]}][\rays][\rays'] \, \b{h}_{\ray[s]} = 0$.
\end{proof}

\begin{remark}
Some observations are in place:

\begin{itemize}
\item Note that any extremal adjacent pair $\{\R_{\ge0}\rays, \R_{\ge0}\rays'\}$ of~$\Fan$ gives rise to a facet of $\ctypeCone(\Fan)$, but that this pair might not be unique for this facet. There might be several extremal adjacent pairs giving rise to the same facet. This happens often when we have the unique exchange relation property and there are many pairs of maximal cones separating the same exchangeable rays.
 
\item Mergings induced by extremal adjacent pairs are not the only ones that can take place in the boundary of~$\ctypeCone(\Fan)$. Indeed, if $\cal{G}$ is the coarsening defining a lower-dimensional face~$\typeCone(\cal{G})$ of~$\ctypeCone(\Fan)$, then to get $\Fan$ we have to perform all the coarsenings given by facets of~$\ctypeCone(\Fan)$ containing~$\typeCone(\cal{G})$. But we might have to perform more coarsenings arising from exchangeable pairs  $\{\R_{\ge0}\rays, \R_{\ge0}\rays'\}$ whose inequality is redundant for~$\ctypeCone(\Fan)$ but tight for all points in~$\typeCone(\cal{G})$. This situation arises when, in order to have a polytopal polyhedral fan, more pairs of adjacent cones must be merged.

\item Even if $\Fan$ is essential, not all its polytopal coarsenings need to be. Indeed, it can happen that there is a certain polytopal coarsening $\cal{G}$ in which all the cells have a common lineality space. In this case, $\cal{G}$ is the normal fan of a lower dimensional polytope (whose codimension is the dimension of the lineality space of $\cal{G}$). For this to happen, it is necessary (but not sufficient) that there is a positive linear dependence between some rays of $\Fan$. 
\end{itemize}
\end{remark}

\begin{example}
\label{exm:typeConeCoarsenings}
Consider the $2$-dimensional fan~$\Fan$ of \cref{fig:typeCone}\,(left) whose type cone has been described in \cref{exm:typeCone} and represented in \cref{fig:typeCone}\,(right).
The relative interior of the three facets of~$\ctypeCone(\Fan)$ are the type cones of three coarsenings of~$\Fan$, obtained by merging certain adjacent maximal cones (those that are extremal). One of the rays generating~$\b{K}\ctypeCone(\Fan)$ corresponds to a triangle, and its normal fan is obtained by merging two pairs of extremal adjacent maximal cones. The two remaining rays lie not only in hyperplanes~$H_{ij}$ defining facets of the type cone, but also in one extra such hyperplane. Their normal fans have only two cones, one arising as a merging of three cones and one arising as a merging of two cones. These fans are not essential, they have a one-dimensional lineality space. Hence, the corresponding polytopes are not full-dimensional, they are of codimension~$1$.

\begin{figure}[htpb]
	\capstart
	\centerline{\includegraphics[scale=.7]{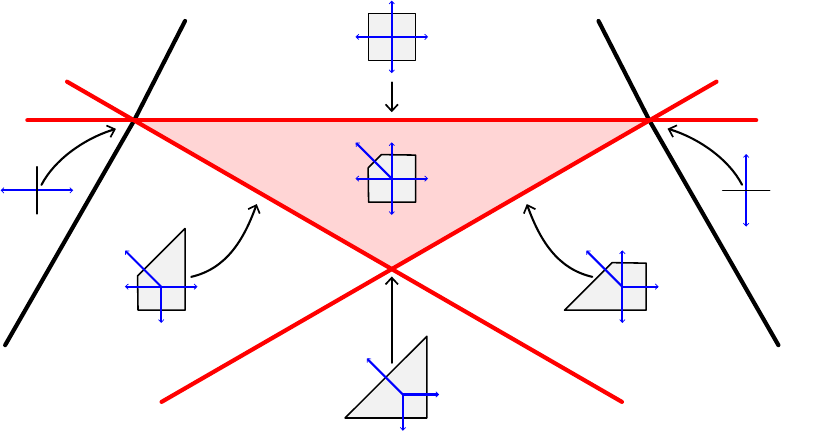}}
	\caption{The type cone~$\ctypeCone(\Fan)$ of the fan~$\Fan$ of \cref{fig:typeCone}\,(left). On top of each face of $\ctypeCone(\Fan)$, we depict an example of polytope with the corresponding normal fan.}
	\label{fig:typeConeCoarsenings}
\end{figure}
\end{example}


\subsection{Minkowski sums}
\label{subsec:MinkowskiSums}

We now discuss the connection between convex combinations in the type cone and Minkowski sums of polytopes.
Recall that the \defn{Minkowski sum} of two polytopes~$P$ and~$Q$ is the polytope $P + Q \eqdef \set{p + q}{p \in P\text{ and } q \in Q}$. 

\begin{lemma}
\label{lem:MinkowskiSum}
For~$\b{h}, \b{h}' \! \in \R^N$, the polytope~$P_{\b{h} + \b{h'}}$ is the Minkowski sum of the polytopes~$P_\b{h}$~and~$P_{\b{h}'}$.
\end{lemma}

\begin{proof}
This follows from the definition of~$P_\b{h}$ and of Minkowski sums:
\begin{align*}
P_\b{h} + P_{\b{h}'} & = \set{\b{x} + \b{x}'}{\b{x} \in P_\b{h} \text{ and } \b{x}' \in P_{\b{h}'}} = \set{\b{x} + \b{x}'}{\b{G}\b{x} \le \b{h} \text{ and } \b{G}\b{x}' \le \b{h}'} \\
& = \set{\b{y} \in \R^n}{\b{G}\b{y} \le \b{h} + \b{h}'} = P_{\b{h} + \b{h}'}.
\qedhere
\end{align*}
\end{proof}

We say that~$P$ is a \defn{Minkowski summand} of~$R$ if there is a polytope~$Q$ such that $P+Q = R$, and a \defn{weak Minkowski summand} if there is a $\lambda \ge 0$ and a polytope~$Q$ such that ${P + Q = \lambda R}$. 
\cref{lem:MinkowskiSum} ensures that convex combinations in the type cone correspond to Minkowski combinations of polytopes. Actually, the set of weak Minkowski summands of a polytope $P$ with normal fan~$\Fan$ is isomorphic to the closed type cone~$\ctypeCone(\Fan)$ modulo lineality (the cone of weak Minkowski summands was studied by W.~Meyer in~\cite{Meyer}, to see that it is equivalent to the type cone and other formulations, see for example the appendix of~\cite{PostnikovReinerWilliams}).

This provides natural Minkowski summands for all polytopal realizations of~$\Fan$: the rays of the closed type cone. These correspond to indecomposable polytopes, see~\cite{Meyer} or~\cite{McMullen-typeCone}. A polytope $P \subseteq \R^n$ is called \defn{indecomposable} if all its weak Minkowski summands are of the form $\lambda P+t$ for some $\lambda \ge 0$ and $t \in \R^n$.

\begin{corollary}
Any polytope in~$\ctypeCone(\Fan)$ is the Minkowski sum of at most~$N-n$ indecomposable polytopes corresponding to some rays of~$\b{K}\ctypeCone(\Fan)$.
\end{corollary}

If $P = Q + R$, then we say that~$R$ is the \defn{Minkowski difference} of~$P$ and~$Q$, denoted~${R = P - Q}$. Note that this is only defined when~$Q$ is a Minkowski summand of~$P$.
For instance, we have ${P_\b{h} - P_{\b{h}'} = P_{\b{h} - \b{h'}}}$ when $P_{\b{h}'}$ is a Minkowski summand of $P_\b{h}$.
There are two natural ways this construction generalizes to arbitrary pairs of polyhedra.
On the one hand, one can consider that ${P - Q = \set{\b{x} \in \R^n}{\b{x} + Q \subseteq P}}$.
On the other hand, considering arbitrary differences of support functions gives rise to \defn{virtual polyhedra}, which is the Grothendieck group of the semigroup of polytopes with Minkowski sums~\cite{PukhlikovKhovanskii}.

\begin{corollary}
Let $\b{h}^{(1)},\dots, \b{h}^{(N-n)} \in \ctypeCone(\Fan)$ form a basis of the left kernel of~$\b{G}$. Then for any~$\b{h} \in \ctypeCone(\Fan)$, the polytope~$P_\b{h}$ has a unique representation (up to translation) as a signed Minkowski sum of dilates of $P_{\b{h}^{(1)}}, \dots, P_{\b{h}^{(N-n)}}$.
\end{corollary}

\begin{proof}
It suffises to express the projection of $\b{h}$ onto the left kernel of~$\b{G}$ in terms of $\b{h}^{(1)}, \dots, \b{h}^{(N-n)}$ to obtain the coefficients of~$P_\b{h}$ as a signed Minkowski sum of dilates of $P_{\b{h}^{(1)}}, \dots, P_{\b{h}^{(N-n)}}$.
\end{proof}

This is used in~\cite{ArdilaBenedettiDoker} to show that every generalized permutahedron can be written uniquely as a signed Minkowski sum of simplices. 
In the particular case when the type cone is simplicial, then we can get rid of the Minkowski differences.

\begin{corollary}
If the type cone~$\typeCone(\Fan)$ of $\Fan$ is simplicial, then every $P_\b{h}$ with $\b{h} \in \typeCone(\Fan)$ (resp.~with $\b{h} \in \ctypeCone(\Fan)$) has a unique representation, up to translation, as a Minkowski sum of positive (resp.~non-negative) dilates of the $N-n$ indecomposable polytopes $P_{\b{h}^{(1)}}, \dots, P_{\b{h}^{(N-n)}}$ arising from the rays~$\b{h}^{(1)}, \dots, \b{h}^{(N-n)}$ of $\ctypeCone(\Fan)$.
\end{corollary}

Note that such a family of polytopes giving rise to unique positive/non-negative representations only exists when the type cone is simplicial. We get unicity with any basis, but in general we cannot guarantee non-negativity unless we are in the simplicial cone they span. If we take all rays of the type cone, we get non-negativity, but the representation is not unique unless the cone is simplicial.

\begin{example}
\label{exm:typeConeMinkowskiSum}
Consider the $2$-dimensional fan~$\Fan$ of \cref{fig:typeCone}\,(left) whose type cone has been described in \cref{exm:typeCone} and represented in \cref{fig:typeCone}\,(right).
As illustrated in \cref{fig:typeConeCoarsenings}, the rays of the closed type cone~$\ctypeCone(\Fan)$ are directed by the height vectors~$(1,2,1,2,1)$ (corresponding to a triangle), $(0,1,1,0,1)$ (corresponding to a vertical segment) and~$(1,0,1,1,0)$ (corresponding to a horizontal segment).
Therefore, any polytope in~$\ctypeCone(\Fan)$ is (up to translation) a positive Minkowski combination of a triangle and two edges.
\cref{fig:typeConeMinkowskiSum} illustrates this property for the polytope of \cref{fig:typeCone}\,(middle) corresponding to the height vector~$(\nicefrac{1}{2}, \nicefrac{3}{4}, \nicefrac{5}{4}, 1, \nicefrac{5}{4})$.
Note that~$\Fan$ is the normal fan of J.-L.~Loday's $2$-dimensional associahedron, whose realization as a Minkowski sum of faces of a simplex was given by A.~Postnikov in~\cite{Postnikov}.

\begin{figure}
	\capstart
	\centerline{
	\begin{tabular}{c@{\;$=$\;}c@{\,$+$\,}c@{\,$+$\,}c@{\,}l@{\,}l}
	$(\frac{1}{2}, \frac{3}{4}, \frac{5}{4}, 1, \frac{5}{4})$ & $\frac{1}{6}(1,2,1,2,1)$ & $\frac{3}{4}(0,1,1,0,1)$ & $\frac{1}{2}(1,0,1,1,0)$ & $-\,\frac{1}{6}(1,0,-1,-1,0)$ & $-\,\frac{1}{3}(0,1,1,0,-1)$ \\
	\raisebox{-0.5\height}{\includegraphics[scale=.5]{typeConePolytope}} & \raisebox{-0.5\height}{\includegraphics[scale=.5]{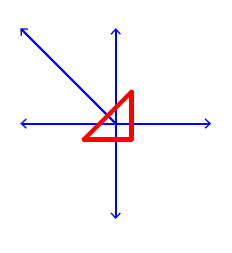}} & \raisebox{-0.5\height}{\includegraphics[scale=.5]{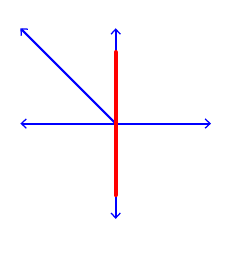}} & \raisebox{-0.5\height}{\includegraphics[scale=.5]{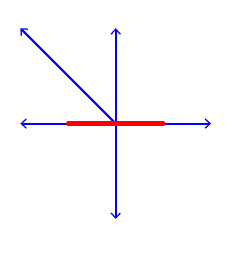}} & $+$ translation.
	\end{tabular}
	}
	\caption{The polytope of \cref{fig:typeCone}\,(middle) is a Minkowski sum of a triangle and two edges.}
	\label{fig:typeConeMinkowskiSum}
\end{figure}
\end{example}


\section{Applications to two generalizations of the associahedron}
\label{sec:applications}

In this section, we study the type cones of complete simplicial fans arising as normal fans of two families of generalizations of the associahedron: the generalized associahedra of finite type cluster algebras~\cite{FominZelevinsky-ClusterAlgebrasI, FominZelevinsky-ClusterAlgebrasII, FominZelevinsky-ClusterAlgebrasIV, HohlwegLangeThomas, HohlwegPilaudStella} and the gentle associahedra~\cite{PaluPilaudPlamondon-nonkissing}.
Both families contain the classical associahedra~$\Asso[n]$ constructed in~\cite{ShniderSternberg, Loday}.


\subsection{Classical associahedra}
\label{subsec:typeConeAsso}

We first describe the associahedra of~\cite{ShniderSternberg, Loday} and their type cones as they are the prototypes of our constructions.


\subsubsection{Associahedra and their normal fans}

We quickly recall the combinatorics and the geometric construction of~\cite{ShniderSternberg, Loday} for the associahedron~$\Asso[n]$.
The face lattice of~$\Asso[n]$ is the reverse inclusion lattice of dissections (\ie pairwise non-crossing subsets of diagonals) of a convex~$(n+3)$-gon.
In particular, its vertices correspond to triangulations of the $(n+3)$-gon and its facets correspond to internal diagonals of the $(n+3)$-gon.
Equivalently, its vertices correspond to rooted binary trees with $(n+1)$ internal nodes, and its facets correspond to proper intervals of~$[n+1]\eqdef\{1,\dots,n+1\}$ (\ie intervals distinct from~$\varnothing$ and~$[n+1]$).
These bijections become clear when the convex $(n+3)$-gon is drawn with its vertices on a concave curve labeled by~$0, \dots, n+2$.
The following statement provides three equivalent geometric constructions of~$\Asso[n]$.

\begin{theorem}
\label{thm:associahedronLoday}
The associahedron~$\Asso[n]$ can be described equivalently as:
\begin{itemize}
\item the convex hull of the points~$L(T) \in \R^{n+1}$ for all rooted binary trees~$T$ with~$n+1$ internal nodes, where the $i$th coordinate of~$L(T)$ is the product of the number of leaves in the left subtree by the number of leaves in the right subtree of the $i$th node of~$T$ in inorder~\cite{Loday},
\item the intersection of the hyperplane~$\bigset{\b{x} \in \R^{n+1}}{\sum_{\ell \in [n+1]} \b{x}_\ell = \binom{n+2}{2}}$ with the halfspaces $\bigset{\b{x} \in \R^{n+1}}{\sum_{a < \ell < b} \b{x}_\ell \ge \binom{b-a}{2}}$ for all internal diagonals~$(a,b)$ of \mbox{the $(n+3)$-gon~\cite{ShniderSternberg}},
\item the Minkowski sum of the faces~$\simplex_{[i,j]}$ of the standard $n$-dimensional simplex~$\simplex_{[n+1]}$ corresponding to all proper intervals~$[i,j]$ of~$[n+1]$,~\cite{Postnikov}.
\end{itemize}
\end{theorem}

\pagebreak
We now focus on the normal fan of~$\Asso[n]$.
Since~$\Asso[n]$ lives in a hyperplane of~$\R^{n+1}$, its normal fan has a one-dimensional lineality.
Let~$\HH \eqdef \bigset{\b{x} \in \R^{n+1}}{\sum_{\ell \in [n+1]} \b{x}_\ell = 0}$ and ${\pi : \R^{n+1} \to \HH}$ denote the orthogonal projection. We denote by~$\gvectorFan[n]$ the intersection of the normal fan of~$\Asso[n]$ with~$\HH$, which is an essential complete simplicial fan.

\begin{theorem}
\label{thm:normalFanLoday}
In the fan~$\gvectorFan[n]$, 
\begin{itemize}
\item the normal vector of the facet corresponding to an internal diagonal~$(a,b)$ of the $(n+3)$-gon is the vector~$\gvector{a,b} \eqdef \pi\big( \sum_{a < \ell < b} \b{e}_\ell \big) = (n+1) \sum_{a < \ell < b} \b{e}_\ell - (b-a-1) \sum_{1 \le \ell \le n+1} \b{e}_\ell$.
\item the normal cone of the vertex corresponding to a rooted binary tree~$T$ is the incidence cone~$\bigset{\b{x} \in \HH}{\b{x}_i \le \b{x}_j \text{ for all edges $i \to j$ in $T$}}$.
\end{itemize}
\end{theorem}

\cref{thm:normalFanLoday} is illustrated in \cref{fig:lodayFans} and \cref{exm:typeConeAsso} in dimension~$2$ and~$3$.
Note that in \cref{fig:lodayFans} and \cref{exm:typeConeAsso} we express the $\b{g}$-vectors in the basis given by $\gvector{0,i+1}$ for $i\in [n]$ (to be coherent with the upcoming \cref{def:gvectorCA}).

\begin{figure}
	\capstart
	\centerline{\raisebox{.5cm}{\includegraphics[scale=.45]{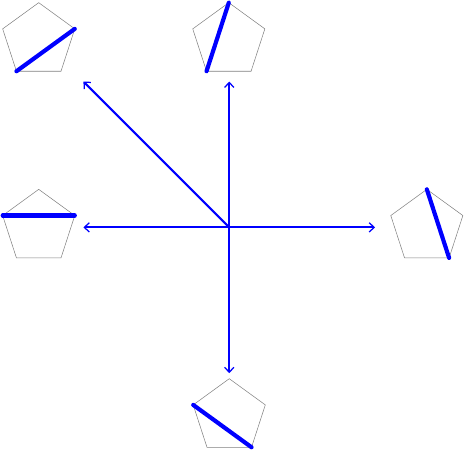}}\hspace{1.5cm}\includegraphics[scale=.45]{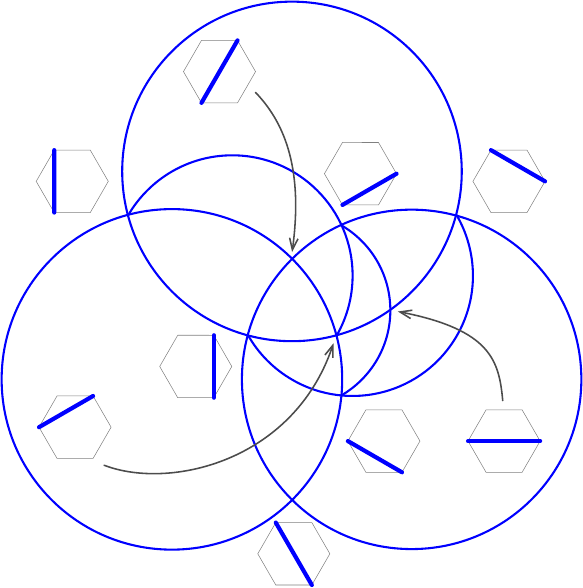}}
	\caption{The normal fans~$\gvectorFan[2]$ and~$\gvectorFan[3]$ of J.-L.~Loday's associahedra. As the rightmost fan is $3$-dimensional, we intersect it with the sphere and stereographically project it from the direction~$(-1,-1,-1)$.}
	\label{fig:lodayFans}
\end{figure}

Let us also recall the linear dependencies in this fan and observe that it has the unique exchange relation property discussed in \cref{subsec:uniqueExchangeProperty}.
From now on, we use the convention that~$\gvector{a,b} = 0$ when~$(a,b)$ is a boundary edge of the $(n+3)$-gon.

\begin{proposition}
\label{prop:exchangeablePairsAsso}
Let~$(a,b)$ and~$(a',b')$ be two crossing diagonals with~${0 \le a < a' < b < b' \le n+2}$, and let~$T$ and~$T'$ be any two triangulations such that~$T \ssm \{(a,b)\} = T' \ssm \{(a',b')\}$.
Then both triangulations~$T$ and~$T'$ contain the square $aa'bb'$, and the linear dependence between the $\b{g}$-vectors of~$T \cup T'$ is given by
\[
\gvector{a,b} + \gvector{a',b'} = \gvector{a,b'} + \gvector{a',b}.
\]
In particular, the fan~$\gvectorFan[n]$ has the unique exchange relation property.
\end{proposition}


\subsubsection{Type cones of associahedra}

From the linear dependencies of \cref{prop:exchangeablePairsAsso}, we obtain a redundant description of the type cone of the fan~$\gvectorFan[n]$.
To simplify the presentation and write the wall-crossing inequalities in a uniform way, we embed the type cone in a larger space, adding dummy variables for the boundary edges of the polygon.

\begin{corollary}
\label{coro:typeConeAsso}
Let~$n \in \N$ and~$X(n) \eqdef \set{(a,b)}{0 \le a < b \le n+2}$. Then the type cone of the normal fan~$\gvectorFan[n]$ of~$\Asso[n]$ is given by
\[
\typeCone \big( \gvectorFan[n] \big) = \set{\b{h} \in \R^{X(n)}}{\begin{array}{@{}l@{}} \b{h}_{(0,n+2)} = 0, \qquad\text{and}\qquad \b{h}_{(a,a+1)} = 0 \text{ for all } 0 \le a \le n+1 \\ \b{h}_{(a,b)} + \b{h}_{(a',b')} > \b{h}_{(a,b')} + \b{h}_{(a',b)} \text{ for all } 0 \le a < a' < b < b' \le n+2 \end{array}} \!.
\]
\end{corollary}

\pagebreak

\begin{example}
\label{exm:typeConeAsso}
Consider the fan~$\gvectorFan[2]$ and~$\gvectorFan[3]$ illustrated in \cref{fig:lodayFans}.
The type cone of~$\gvectorFan[2]$ has dimension~$5$ and a lineality space of dimension~$2$ (this is the type cone studied in \cref{exm:typeCone,exm:typeConeCoarsenings,exm:typeConeMinkowskiSum}).
It has $3$ facet-defining inequalities (given below), which correspond to the flips described in \cref{prop:extremalExchangeablePairsAsso} and illustrated in \cref{fig:labelFacetDefiningInequalititiesAsso}\,(left).

\[
\begin{array}{r|cccccc}
\text{diagonals} & \raisebox{-.25cm}{\;\includegraphics[scale=.7]{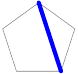}}  & \raisebox{-.25cm}{\;\includegraphics[scale=.7]{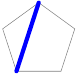}} & \raisebox{-.25cm}{\;\includegraphics[scale=.7]{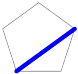}} & \raisebox{-.25cm}{\;\includegraphics[scale=.7]{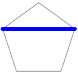}} & \raisebox{-.25cm}{\;\includegraphics[scale=.7]{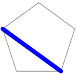}} \\[.3cm]
\text{$\b{g}$-vectors} & \compactVectorD{1}{0} & \compactVectorD{0}{1} & \compactVectorD{-1}{1} & \compactVectorD{-1}{0} & \compactVectorD{0}{-1} \\[.6cm]
\text{facet}		& 1 & -1 & 1 & 0 & 0 & \red \circled{A} \\
\text{defining} 	& 0 & 1 & -1 & 1 & 0 & \red \circled{B} \\
\text{inequalities}	& 0 & 0 & 1 & -1 & 1 & \red \circled{C} \\[.2cm]
\end{array}
\]

\medskip
\noindent
In all our tables, we just record the coefficients of the inequalities, which are then easily reconstructed.
For instance, the inequality~$\red \circled{A}$ above is given by~$\b{h}_{\includegraphics[scale=.3]{diagonalA1}} + \b{h}_{\includegraphics[scale=.3]{diagonalA3}} > \b{h}_{\includegraphics[scale=.3]{diagonalA2}}$.

\begin{figure}[h]
	\capstart
    \begin{adjustbox}{center}
        \begin{tikzpicture}
        	\matrix (m) [matrix of math nodes, row sep=.35cm, column sep=.2cm, nodes={anchor=center, align=center, inner sep=0pt}]{
        		& \;\includegraphics[scale=.7]{diagonalA2} & \node (b) {\red \circled{B}}; & \;\includegraphics[scale=.7]{diagonalA4} & \\
        		\;\includegraphics[scale=.7]{diagonalA1} & \node (a) {\red \circled{A}}; & \;\includegraphics[scale=.7]{diagonalA3} & \node (c) {\red \circled{C}}; & \;\includegraphics[scale=.7]{diagonalA5} \\};
        	\draw[->] (m-2-1) -- (m-1-2);
        	\draw[->] (m-1-2) -- (m-2-3);
        	\draw[->] (m-2-3) -- (m-1-4);
        	\draw[->] (m-1-4) -- (m-2-5);
        	\draw[red] (a) -- (m-2-1);
        	\draw[red] (a) -- (m-2-3);
        	\draw[red] (a) -- (m-1-2);
        	\draw[red] (b) -- (m-1-2);
        	\draw[red] (b) -- (m-1-4);
        	\draw[red] (b) -- (m-2-3);
        	\draw[red] (c) -- (m-2-3);
        	\draw[red] (c) -- (m-2-5);
        	\draw[red] (c) -- (m-1-4);
        \end{tikzpicture}
        \quad
        \begin{tikzpicture}
        	\matrix (m) [matrix of math nodes, row sep=.35cm, column sep=.2cm, nodes={anchor=center, align=center, inner sep=0pt}]{
        		&& \;\includegraphics[scale=.7]{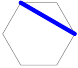} & \node (c) {\red \circled{C}}; & \;\includegraphics[scale=.7]{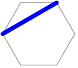} && \\
        		& \;\includegraphics[scale=.7]{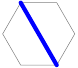} & \node (b) {\red \circled{B}}; & \;\includegraphics[scale=.7]{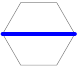} & \node (e) {\red \circled{E}}; & \;\includegraphics[scale=.7]{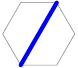} & \\
        		\;\includegraphics[scale=.7]{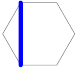} & \node (a) {\red \circled{A}}; & \;\includegraphics[scale=.7]{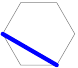} & \node (d) {\red \circled{D}}; & \;\includegraphics[scale=.7]{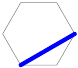} & \node (f) {\red \circled{F}}; & \;\includegraphics[scale=.7]{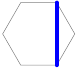} \\};
        	\draw[->] (m-3-1) -- (m-2-2);
        	\draw[->] (m-2-2) -- (m-1-3);
        	\draw[->] (m-2-2) -- (m-3-3);
        	\draw[->] (m-1-3) -- (m-2-4);
        	\draw[->] (m-3-3) -- (m-2-4);
        	\draw[->] (m-2-4) -- (m-1-5);
        	\draw[->] (m-2-4) -- (m-3-5);
        	\draw[->] (m-1-5) -- (m-2-6);
        	\draw[->] (m-3-5) -- (m-2-6);
        	\draw[->] (m-2-6) -- (m-3-7);
        	\draw[red] (a) -- (m-3-1);
        	\draw[red] (a) -- (m-3-3);
        	\draw[red] (a) -- (m-2-2);
        	\draw[red] (b) -- (m-2-2);
        	\draw[red] (b) -- (m-2-4);
        	\draw[red] (b) -- (m-1-3);
        	\draw[red] (b) -- (m-3-3);
        	\draw[red] (c) -- (m-1-3);
        	\draw[red] (c) -- (m-1-5);
        	\draw[red] (c) -- (m-2-4);
        	\draw[red] (d) -- (m-3-3);
        	\draw[red] (d) -- (m-3-5);
        	\draw[red] (d) -- (m-2-4);
        	\draw[red] (e) -- (m-2-4);
        	\draw[red] (e) -- (m-2-6);
        	\draw[red] (e) -- (m-1-5);
        	\draw[red] (e) -- (m-3-5);
        	\draw[red] (f) -- (m-3-5);
        	\draw[red] (f) -- (m-3-7);
        	\draw[red] (f) -- (m-2-6);
        \end{tikzpicture}
    \end{adjustbox}
	\caption{The facet-defining inequalities of the type cone~$\typeCone \big( \gvectorFan[n] \big)$ correspond to the flips described in \cref{prop:extremalExchangeablePairsAsso}. See \cref{subsec:typeConeCA,sec:clusterCategories} for a representation theoretic \mbox{interpretation}.}
	\label{fig:labelFacetDefiningInequalititiesAsso}
\end{figure}

The type cone of~$\gvectorFan[3]$ has dimension~$9$ and a lineality space of dimension~$3$.
It has $6$ facet-defining inequalities (given below), which correspond to the flips described in \cref{prop:extremalExchangeablePairsAsso} and illustrated in \cref{fig:labelFacetDefiningInequalititiesAsso}\,(right).

\[
\begin{array}{r|cccccccccc}
\text{diagonals} & \raisebox{-.25cm}{\;\includegraphics[scale=.7]{diagonalB1}}  & \raisebox{-.25cm}{\;\includegraphics[scale=.7]{diagonalB2}} & \raisebox{-.25cm}{\;\includegraphics[scale=.7]{diagonalB3}} & \raisebox{-.25cm}{\;\includegraphics[scale=.7]{diagonalB4}} & \raisebox{-.25cm}{\;\includegraphics[scale=.7]{diagonalB5}} & \raisebox{-.25cm}{\;\includegraphics[scale=.7]{diagonalB6}} & \raisebox{-.25cm}{\;\includegraphics[scale=.7]{diagonalB7}} & \raisebox{-.25cm}{\;\includegraphics[scale=.7]{diagonalB8}} & \raisebox{-.25cm}{\;\includegraphics[scale=.7]{diagonalB9}} \\[.3cm]
\text{$\b{g}$-vectors} & \compactVectorT{1}{0}{0} & \compactVectorT{0}{1}{0} & \compactVectorT{0}{0}{1} & \compactVectorT{-1}{1}{0} & \compactVectorT{0}{-1}{1} & \compactVectorT{-1}{0}{1} & \compactVectorT{-1}{0}{0} & \compactVectorT{0}{-1}{0} & \compactVectorT{0}{0}{-1} \\[.6cm]
\text{facet}		& 1 & -1 & 0 & 1 & 0 & 0 & 0 & 0 & 0 & \red \circled{A} \\
\text{defining} 	& 0 & 1 & -1 & -1 & 1 & 0 & 0 & 0 & 0 & \red \circled{B} \\
\text{inequalities}	& 0 & 0 & 1 & 0 & -1 & 0 & 1 & 0 & 0 & \red \circled{C} \\
					& 0 & 0 & 0 & 1 & -1 & 1 & 0 & 0 & 0 & \red \circled{D} \\
					& 0 & 0 & 0 & 0 & 1 & -1 & -1 & 1 & 0 & \red \circled{E} \\
					& 0 & 0 & 0 & 0 & 0 & 1 & 0 & -1 & 1 & \red \circled{F} \\[.2cm]
\end{array}
\]




\end{example}

\pagebreak

Motivated by \cref{exm:typeConeAsso}, we now describe the facets of this type cone~$\typeCone \big( \gvectorFan[n] \big)$.

\begin{proposition}
\label{prop:extremalExchangeablePairsAsso}
Two internal diagonals~$(a,b)$ and~$(a',b')$ of the~$(n+3)$-gon form an extremal exchangeable pair for the fan~$\gvectorFan[n]$ if and only if~$a = a'+1$ and~$b = b'+1$, or the opposite.
\end{proposition}

\begin{proof}
Let~$(\b{f}_{(a,b)})_{0 \le a < b \le n+2}$ be the canonical basis of~$\R^{\binom{n+3}{2}}$.
Consider two crossing internal diagonals~$(a,b)$ and~$(a',b')$ with~$0 \le a < a' < b < b' \le n+2$.
By \cref{prop:exchangeablePairsAsso}, the linear dependence between the corresponding $\b{g}$-vectors is given by
\[
\gvector{a,b} + \gvector{a',b'} = \gvector{a,b'} + \gvector{a',b}.
\]
Therefore, the inner normal vector of the corresponding inequality of the type cone~$\typeCone \big( \gvectorFan[n] \big)$ is
\[
\b{n}(a,b,a',b') \eqdef \b{f}_{(a,b)} + \b{f}_{(a',b')} - \b{f}_{(a,b')} - \b{f}_{(a',b)}.
\]
Denoting
\[
\b{m}(c,d) \eqdef \b{n}(c,d-1,c+1,d) = \b{f}_{(c,d-1)} + \b{f}_{(c+1,d)} - \b{f}_{(c,d)} - \b{f}_{(c+1,d-1)},
\]
we obtain that
\[
\b{n}(a,b,a',b') = \sum_{\substack{c \in {[a,a'[} \\ d \in {]b,b']}}} \b{m}(c,d).
\]
Indeed, on the right hand side, the basis vector~$\b{f}_{(c,d)}$ appears with a positive sign in~$\b{m}(c,d+1)$ for~$(c,d) \in {[a,a'[} \times {[b,b'[}$ and in~$\b{m}(c-1,d)$ for~$(c,d) \in {]a,a']} \times {]b,b']}$, and with a negative sign in~$\b{m}(c,d)$ for~$(c,d) \in {[a,a'[} \times {]b,b']}$ and in~$\b{m}(c-1,d+1)$ for~$(c,d) \in {]a,a']} \times {[b,b'[}$.
Therefore, these contributions all vanish except when~$(c,d)$ is one of the diagonals~$(a,b)$, $(a',b')$, $(a,b')$ or~$(a',b)$.
This shows that any exchange relation is a positive linear combination of the exchange relations corresponding to all pairs of diagonals~$(a,b)$ and~$(a',b')$ of the~$(n+3)$-gon such that~$a = a'+1$ and~$b = b'+1$ or the opposite.

Conversely, since~$\gvectorFan[n]$ has dimension~$n$ and~$n(n+3)/2$ rays (corresponding to the internal diagonals of the $(n+3)$-gon), we know from \cref{rem:dimTypeCone} that there are at least~$n(n+1)/2$ extremal exchangeable pairs. We thus conclude that all exchangeable pairs of diagonals~$\{(a,b-1), (a+1,b)\}$ for~$1 \le a < b-2 \le n$ are extremal.
\end{proof}

Our next statement follows from the end of the previous proof.

\begin{corollary}
\label{coro:simplicialTypeConeAsso}
The type cone~$\typeCone \big( \gvectorFan[n] \big)$ is simplicial.
\end{corollary}

Combining \cref{coro:simplicialTypeCone,coro:simplicialTypeConeAsso} and \cref{prop:extremalExchangeablePairsAsso}, we derive the following description of all polytopal realizations of the fan~$\gvectorFan[n]$, recovering all associahedra of~\cite[Sect.~3.2]{ArkaniHamedBaiHeYan}.
Note that the arguments of~\cite[Sect.~3.2]{ArkaniHamedBaiHeYan} were quite different from the present~approach.

\begin{corollary}[{\cite[Sect.~3.2]{ArkaniHamedBaiHeYan}}]
\label{coro:allPolytopalRealizationsAsso}
For~$n \in \N$, define~$X(n) \eqdef \set{(a,b)}{0 \le a < b \le n+2}$ and~$Y(n) \eqdef \set{(a,b)}{1 \le a < b \le n+1}$.
Then for any~$\b{\ell} \in \R_{>0}^{Y(n)}$, the polytope
\[
R_\b{\ell}(n) \eqdef \set{\b{z} \in \R^{X(n)}}{\begin{array}{l} \b{z} \ge 0, \quad \b{z}_{(0,n+2)} = 0 \quad\text{and}\quad \b{z}_{(a,a+1)} = 0 \text{ for all } 0 \le a \le n+1 \\ \b{z}_{(a-1,b)} + \b{z}_{(a,b+1)} - \b{z}_{(a,b)} - \b{z}_{(a-1,b+1)} = \b{\ell}_{(a,b)} \text{ for all } (a,b) \in Y(n) \end{array}}
\]
is an $n$-dimensional associahedron, whose normal fan is~$\gvectorFan[n]$.
Moreover, the polytopes~$R_\b{\ell}(n)$ for~$\b{\ell} \in \R_{>0}^{Y(n)}$ describe all polytopal realizations of the fan~$\gvectorFan[n]$.
\end{corollary}



\subsection{Finite type cluster complexes and generalized associahedra}
\label{subsec:typeConeCA}

Cluster algebras were introduced by S.~Fomin and A.~Zelevinsky~\cite{FominZelevinsky-ClusterAlgebrasI} with motivation coming from total positivity and canonical bases.
Here, we will focus on finite type cluster algebras~\cite{FominZelevinsky-ClusterAlgebrasII} and more specifically on properties of their $\b{g}$-vectors~\cite{FominZelevinsky-ClusterAlgebrasIV}.
These $\b{g}$-vectors support a complete simplicial fan, which is called \defn{$\b{g}$-vector fan} or \defn{cluster fan}. It is known to be the normal fan of a polytope called generalized associahedron. These polytopal realizations were first constructed for bipartite initial seeds by F.~Chapoton, S.~Fomin and A.~Zelevinsky~\cite{ChapotonFominZelevinsky} using the $\b{d}$-vector fans of~\cite{FominZelevinsky-YSystems}, then for acyclic initial seeds by C.~Hohlweg, C.~Lange and H.~Thomas~\cite{HohlwegLangeThomas} using Cambrian lattices and fans of N.~Reading and D.~Speyer~\cite{Reading-CambrianLattices, ReadingSpeyer}, then revisited by S.~Stella~\cite{Stella}, and by V.~Pilaud and C.~Stump~\cite{PilaudStump-brickPolytope} via brick polytopes, and finally for arbitrary initial seeds by C.~Hohlweg, V.~Pilaud and S.~Stella~\cite{HohlwegPilaudStella}.
More recently in~\cite{BazierMatteDouvilleMousavandThomasYildirim}, V.~Bazier-Matte, G.~Douville, K.~Mousavand, H.~Thomas and E.~Y\i ld\i r\i m extended the construction of associahedra of~\cite[Sect.~3.2]{ArkaniHamedBaiHeYan} to acyclic seeds of simply-laced finite type cluster algebras.
Here, we extend the construction of~\cite[Sect.~3.2]{ArkaniHamedBaiHeYan} to any finite type (simply-laced or not) cluster algebra with respect to any seed (acyclic or not) by providing a direct proof that the type cone of the cluster fan of any finite type cluster algebra with respect to any initial seed is simplicial.


\subsubsection{Cluster algebras and cluster fans}

We present some definitions and properties of finite type cluster algebras and their cluster fans, following the presentation of~\cite{HohlwegPilaudStella}.

\para{Cluster algebras}
Let~$\Q(x_1, \dots, x_n, p_1, \dots, p_m)$ be the field of rational expressions in~$n+m$ variables with rational coefficients, and~$\Trop{m}$ denote its abelian multiplicative subgroup generated by~$\{p_i\}_{i \in [m]}$.
For~$p = \prod\limits_{i \in [m]} p_i^{a_i} \in \Trop{m}$, define
\(
\positiveExponents{p} \eqdef \prod\limits_{i \in [m]} p_i^{\max(a_i,0)}
\)
and
\(
\negativeExponents{p} \eqdef \prod\limits_{i \in [m]} p_i^{-\min(a_i,0)}
\).

\medskip
A \defn{seed}~$\seed$ is a triple~$(\B, \coefficients, \cluster)$ where
\begin{itemize}
\item the \defn{exchange matrix}~$\B$ is an integer~$n \times n$ skew-symmetrizable matrix, \ie such that there exists a diagonal matrix~$\D$ with~$-\B\D = (\B\D)^T$,
\item the \defn{coefficient tuple}~$\coefficients$ is any subset of~$n$ elements of~$\Trop{m}$,
\item the \defn{cluster}~$\cluster$ is a set of~$n$ \defn{cluster variables} in~$\Q(x_1, \dots, x_n, p_1, \dots, p_m)$ algebraically independent over~$\Q(p_1, \dots, p_m)$.
\end{itemize}
To simplify our notations, we use the convention to label~$\B = (b_{xy})_{x,y \in \cluster}$ and~$\coefficients = \{p_x\}_{x \in \cluster}$ by the cluster variables of~$\cluster$.

For a seed~$\seed = (\B, \coefficients, \cluster)$ and a cluster variable~$x \in \cluster$, the \defn{mutation} in direction~$x$ creates a new seed~$\mu_x(\seed) = \seed' = (\B', \coefficients', \cluster')$ where:
\begin{itemize}
\item the new cluster $\cluster'$ is obtained from $\cluster$ by replacing $x$ with the cluster variable $x'$ defined by the following \defn{exchange relation}:
\[
x x' = \positiveExponents{p_x} \prod_{{y \in \cluster, \; b_{xy}  > 0}} y^{b_{xy}} + \negativeExponents{p_x} \prod_{{y \in \cluster, \; b_{xy}  <0}} y^{-b_{xy}}
\]
and leaving the remaining cluster variables unchanged so that $\cluster \ssm \{x\} = \cluster' \ssm \{x'\}$.

\item the row (resp.~column) of~$\B'$ indexed by~$x'$ is the negative of the row (resp.~column) of~$\B$ indexed by~$x$, while all other entries satisfy
\(
b'_{yz} = b_{yz} + \frac{1}{2}\big(|b_{yx}| b_{xz} + b_{yx}|b_{xz}|\big),
\)

\item the elements of the new coefficient tuple $\coefficients'$ are 
\[
p'_y =
\begin{cases}
	p_x^{-1}  & \text{if } y = x', \\
	p_y\negativeExponents{p_x}^{b_{xy}}  & \text{if } y \ne x' \text{ and } b_{xy} \leq 0, \\
	p_y\positiveExponents{p_x}^{b_{xy}}  & \text{if } y \ne x' \text{ and } b_{xy} > 0.
\end{cases}
\]
\end{itemize}
An important point is that mutations are involutions: $\mu_{x'} \big( \mu_{x}(\seed) \big) = \seed$.
We say that two seeds are \defn{adjacent} (resp.\ \defn{mutationally equivalent}) when they can be obtained from each other by a mutation (resp.\ a sequence of mutations). We also use the same terminology for exchange matrices.
We denote by~$\variables$ the collection of all cluster variables in the seeds mutationally equivalent to an initial seed~$\seed_\circ = (\B_\circ, \coefficients_\circ, \cluster_\circ)$ with cluster variables~$\cluster_\circ = \{x_1, \dots, x_n\}$.
The (geometric type) \defn{cluster algebra} $\clusterAlgebra$ is the $\Z\Trop{m}$-subring of~$\Q(x_1, \dots, x_n, p_1, \dots, p_m)$ generated by the cluster variables in~$\variables$.
The \defn{cluster complex} of~$\clusterAlgebra$ is the simplicial complex whose vertices are the cluster variables of~$\clusterAlgebra$ and whose facets are the clusters of~$\clusterAlgebra$.

\para{Finite type}
In this paper, we only consider \defn{finite type} cluster algebras, \ie those with a finite number~$N$ of cluster variables.
It turns out that these finite type cluster algebras were classified by S.~Fomin and A.~Zelevinsky~\cite{FominZelevinsky-ClusterAlgebrasII} using the Cartan-killing classification for crystallographic root systems.
Define the \defn{Cartan companion} of an exchange matrix $\B$ as the matrix~$\A{\B}$ given by~$a_{xy} = 2$ if~$x = y$ and~$a_{xy} = -|b_{xy}|$ otherwise.

\pagebreak

\begin{theorem}[{\cite[Thm.~1.4]{FominZelevinsky-ClusterAlgebrasII}}]
\label{thm:finiteTypeClassification}
The cluster algebra~$\clusterAlgebra$ is of finite type if and only~if $\B_\circ$ is mutationally equivalent to a matrix~$\B$ whose Cartan companion~$\A{\B}$ is a Cartan matrix of finite type. Moreover the type of $\A{\B}$ is determined~by~$B_\circ$.
\end{theorem}

A finite type exchange matrix~$\B_\circ$ is \defn{acyclic} if~$\A{\B_\circ}$ is already a Cartan matrix, and \defn{cyclic} otherwise.
An acyclic exchange matrix~$\B_\circ$ is \defn{bipartite} if each row of~$\B_\circ$ consists either of non-positive or non-negative entries.
We use the same terminology for the seed~$\seed_\circ$.

From now on, we fix a finite type cluster algebra~$\clusterAlgebra$ and we consider the root system of type~$\A{\B_\circ}$ (again, this root system is finite only when the initial seed~$\B_\circ$ is acyclic). We use the following classical bases of the underlying vector spaces:
\begin{itemize}
\item the simple roots~$\{\simpleRoot_x\}_{x \in \cluster_\circ}$ and the fundamental weights~$\{\fundamentalWeight_x\}_{x \in \cluster_\circ}$ are two bases of the same vector space~$V$ related by the Cartan matrix~$\A{\B_\circ}$,
\item the simple coroots~$\{\simpleRoot^\vee_x\}_{x \in \cluster_\circ}$ and the fundamental coweights~$\{\fundamentalWeight^\vee_x\}_{x \in \cluster_\circ}$ are two basis of the dual space~$V^\vee$ related by the transpose of the Cartan matrix~$\transpose{\A{\B_\circ}}$,
\end{itemize}
and the basis of simple roots is dual to the basis of fundamental coweights, while the basis of fundamental weights is dual to the basis of simple coroots.

The cluster complex of a finite type cluster algebra~$\clusterAlgebra$ is independent of the choice of coefficients and of the choice of the initial seed (as long at it remains in the same mutation class), and therefore only depends on the cluster type of~$\clusterAlgebra$ (\ie the Cartan type of~$\A{\B}$ in \cref{thm:finiteTypeClassification}).
Moreover, for~$\B_\circ^\vee \eqdef -\transpose{\B_\circ}$, then the map sending a cluster variable~$x$ in a seed~$\seed$ of~$\clusterAlgebra$ to the cluster variable~$x^\vee$ in the seed~$\seed^\vee$ of~$\clusterAlgebra[\B_\circ^\vee]$ obtained by the same sequence of mutations defines a natural isomorphism between the cluster complexes of~$\clusterAlgebra$ and~$\clusterAlgebra[\B_\circ^\vee]$.

\para{Principal coefficients and $\b{g}$- and $\b{c}$-vectors}
We now consider principal coefficients to define the $\b{g}$- and $\b{c}$-vectors.
As defined in~\cite[Def.~3.1]{FominZelevinsky-ClusterAlgebrasIV}, the cluster algebra with \defn{principal coefficients} at~$\B_\circ$ is the cluster algebra~$\clusterAlgebra[\B_\circ][\coefficients_\circ]$ in~$\Q(x_1, \dots, x_n, p_1, \dots, p_n)$, where~$\coefficients_\circ = \{p_x\}_{x \in \cluster_\circ}$ are precisely the generators~$p_1, \dots, p_n$ of~$\Trop{n}$ relabeled by~$\cluster_\circ$.
For simplicity, we drop the mention of the coefficients for principal cluster algebras in the notations for seeds~$\seed = (\B, \cluster)$, for cluster algebras~$\principalClusterAlgebra$ and for variables~$\principalVariables$.
Principal coefficients cluster algebras are $\Z^n$-graded (in the fundamental weight basis~$\{\fundamentalWeight_x\}_{x \in \cluster_\circ}$ of~$V$) for the degree function~$\deg(\B_\circ,\cdot)$ on $\principalClusterAlgebra$ obtained by setting
\(
\deg(\B_\circ,  x) \eqdef \fundamentalWeight_x
\)
and
\(
\deg(\B_\circ, p_x) \eqdef \sum_{y \in \cluster_\circ} -b_{yx} \fundamentalWeight_y
\)
for any $x\in\cluster_\circ$.

\begin{definition}[\cite{FominZelevinsky-ClusterAlgebrasIV}]\label{def:gvectorCA}
The \defn{$\b{g}$-vector}~$\gvector{x} = \gvectorFull{\B_\circ}{x}$ of a cluster variable~$x \in \principalClusterAlgebra$ is its degree.
We denote by $\gvectors{\seed} = \gvectorFull{\B_\circ}{\seed} \eqdef \set{\gvectorFull{\B_\circ}{x}}{x \in \seed}$ the set of $\b{g}$-vectors of a seed~$\seed$.
\end{definition}

The next definition gives another family of integer vectors, introduced implicitly in~\cite{FominZelevinsky-ClusterAlgebrasIV}, that are relevant in the structure of~$\principalClusterAlgebra$.

\begin{definition}[\cite{FominZelevinsky-ClusterAlgebrasIV}]
The \defn{$\b{c}$-vector} of a cluster variable $x$ in a seed~$\seed$ of~$\principalClusterAlgebra$ is the vector
\(
{\cvector{\seed}{x} = \cvectorFull{\B_\circ}{\seed}{x} \eqdef \sum_{y \in \cluster_\circ} c_{yx} \, \simpleRoot_y}
\)
of exponents of~$p_x = \prod_{y \in \cluster_\circ} (p_y)^{c_{yx}}$.
We denote by ${\cvectors{\seed} = \cvectorsFull{\B_\circ}{\seed} \eqdef \set{\cvectorFull{\B_\circ}{\seed}{x}}{x \in \seed}}$ the set of $\b{c}$-vectors of a seed~$\seed$.
\end{definition}

These two families of vectors are connected via the isomorphism~$x \mapsto x^\vee$ between the cluster complexes of~$\principalClusterAlgebra$ and~$\principalClusterAlgebra[\B_\circ^\vee]$ described above.

\begin{theorem}[{\cite[Thm~1.2]{NakanishiZelevinsky}}]
\label{prop:gvectorscvectorsDualBasesCA}
For any seed~$\seed$ of~$\principalClusterAlgebra$, let $\seed^\vee$ be its dual in $\principalClusterAlgebra[\B_\circ^\vee]$.
Then the set of $\b{g}$-vectors~$\gvectorFull{\B_\circ}{\seed}$ and the set of $\b{c}$-vectors~$\cvectorsFull{\B_\circ^\vee}{\seed^\vee}$ form dual bases, that is
\(
{\bigdotprod{\gvectorFull{\B_\circ}{x}}{\cvectorFull{\B_\circ^\vee}{\seed^\vee}{y^\vee}} = \delta_{x=y}}
\)
for any two cluster variables~$x,y \in \seed$.
\end{theorem}

\para{Cluster fan and generalized associahedron}
The following statement is well known and admits several possible proofs as discussed in~\cite[Sect.~4]{HohlwegPilaudStella}.
Examples are illustrated in \cref{fig:clusterFans}.

\begin{figure}[t]
	\capstart
	\begin{adjustbox}{center}
    	\begin{overpic}[scale=.45]{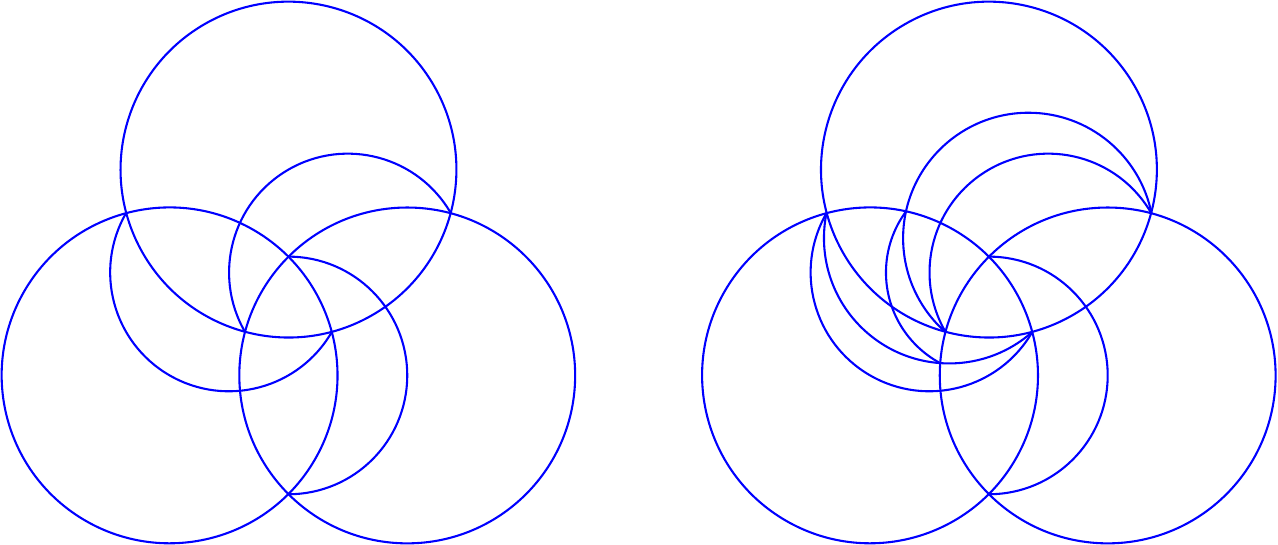}
    	\put(-19,30){$\left[\begin{array}{@{}c@{\;}c@{\;}c@{}} 0 & -1 & 1 \\ 1 & 0 & -1 \\ -1 & 1 & 0 \end{array}\right]$}
    	\put(98,30){$\left[\begin{array}{@{}c@{\;}c@{\;}c@{}} 0 & -1 & 2 \\ 1 & 0 & -2 \\ -1 & 1 & 0 \end{array}\right]$}
    	\end{overpic}
	\end{adjustbox}
	\caption{Two cluster fans~$\gvectorFan[\B_\circ]$ for the type~$A_3$~(left) and type~$C_3$~(right) cyclic initial exchange matrices. As the fans are $3$-dimensional, we intersect them with the sphere and stereographically project them from the direction~$(-1,-1,-1)$. Illustration from~\cite{HohlwegPilaudStella}.}
	\label{fig:clusterFans}
\end{figure}

\begin{theorem}
\label{thm:gvectorFanCA}
For any finite type exchange matrix~$\B_\circ$, the collection of cones
\[
\bigset{\R_{\ge0} \, \gvectorsFull{\B_\circ}{\seed}}{\seed \text{ seed of } \principalClusterAlgebra},
\]
together with all their faces, forms a complete simplicial fan~$\gvectorFan[\B_\circ]$, called the \defn{$\b{g}$-vector fan} or \defn{cluster fan} of~$\B_\circ$.
\end{theorem}

Moreover, this fan is known to be polytopal.
More precisely, consider a vector~$\b{h} \in \R^{\principalVariables}$ such that
\[
\b{h}_x + \b{h}_{x'} > \max \Big( \sum_{{y \in \cluster \cap \cluster', \; b_{xy} < 0}} -b_{xy} \, \b{h}_y \;,  \sum_{{y \in \cluster \cap \cluster', \; b_{xy} > 0}} b_{xy} \, \b{h}_y \Big).
\]
for any adjacent seeds~$(\B, \cluster)$ and~$(\B', \cluster')$ with~$\cluster \ssm \{x\} = \cluster' \ssm \{x'\}$ (with the usual convention that an empty sum is~$0$).
Such a vector~$\b{h}$ exists, see the discussion in~\cite[Prop.~28]{HohlwegPilaudStella}.

\begin{theorem}[{\cite[Thm.~26]{HohlwegPilaudStella}}]
\label{thm:generalizedAsso}
For any finite type exchange matrix~$\B_\circ$, the cluster fan~$\gvectorFan[\B_\circ]$ is the normal fan of the \defn{$\B_\circ$-associahedron}~$\Asso[\B_\circ, \b{h}]$ defined equivalently as
\begin{enumerate}[(i)]
\item the convex hull of the points~$\sum_{x \in \seed} \b{h}_x \, \cvectorFull{\B_\circ^\vee}{\seed^\vee}{x^\vee}$ for all seeds~$\seed$ of~$\principalClusterAlgebra$, or
\item the intersection of the halfspaces~$\set{\b{v} \in V^\vee}{\dotprod{\gvectorFull{\B_\circ}{x}}{\b{v}} \le \b{h}_x}$ for all cluster variables~$x$ of~$\principalClusterAlgebra$.
\end{enumerate}
\end{theorem}

\para{Mutations in the cluster fan}
We now discuss the linear dependences between $\b{g}$-vectors of adjacent seeds.
The following lemma was stated in~\cite[Lem.~19]{HohlwegPilaudStella}.

\begin{lemma}
\label{lem:linearDependencegvectorsCA}
For any finite type~exchange matrix~$\B_\circ$ and any adjacent seeds~${(\B, \cluster)}$ and~${(\B', \cluster')}$ in $\principalClusterAlgebra$ with~$\cluster \ssm \{x\} = \cluster' \ssm \{x'\}$, the $\b{g}$-vectors of~$\cluster \cup \cluster'$ with respect to~$\B_\circ$ satisfy precisely one of the following two linear dependences
\[
\gvectorFull{\B_\circ}{x} + \gvectorFull{\B_\circ}{x'} = \sum_{\substack{y \in \cluster \cap \cluster' \\ b_{xy} < 0}} -b_{xy} \, \gvectorFull{\B_\circ}{y}
\quad\text{or}\quad
\gvectorFull{\B_\circ}{x} + \gvectorFull{\B_\circ}{x'} = \sum_{\substack{y \in \cluster \cap \cluster' \\ b_{xy} > 0}} b_{xy} \, \gvectorFull{\B_\circ}{y}.
\]
\end{lemma}

Note that \cref{lem:linearDependencegvectorsCA} implies \cref{thm:generalizedAsso} since it ensures that the type cone~$\typeCone \big( \gvectorFan[\B_\circ] \big)$~of~the cluster fan contains the cone of height vectors~${\b{h} \in \R^{\principalVariables}}$ such that
\(
\b{h}_x + \b{h}_{x'} > \sum_{{y \in \cluster \cap \cluster', \; b_{xy} < 0}} -b_{xy}
\)
and
\(
\b{h}_x + \b{h}_{x'} > \sum_{{y \in \cluster \cap \cluster', \; b_{xy} > 0}} b_{xy} \, \b{h}_y
\)
for any seeds~$(\B, \cluster)$ and~$(\B', \cluster')$ with~${\cluster \ssm \{x\} = \cluster' \ssm \{x'\}}$.

Unfortunately, \cref{lem:linearDependencegvectorsCA} is less precise than \cref{prop:exchangeablePairsAsso}. Indeed, which of the two possible linear dependences is satisfied by the $\b{g}$-vectors of~$\cluster \cup \cluster'$ depends on the initial exchange matrix~$\B_\circ$.
However, for cluster algebras of finite type, this linear dependence is independent of the choice of the adjacent seeds containing~$x$ and~$x'$.
As explained in \cref{corollary:UERPforCAproof}, this statement follows from~\cite[Thm.~7.5]{BuanMarshReinekeReitenTodorov}.
We state it here for future reference.
Note that the corresponding statement for cluster variables, instead of~$\b{g}$-vectors, holds by~\cite[Thm.~1.11]{FominZelevinsky-ClusterAlgebrasII}.

\begin{proposition}
\label{prop:uniqueExchangePropertyCA}
For any finite type exchange matrix~$\B_\circ$, the cluster fan has the unique exchange relation property.
\end{proposition}


\subsubsection{Type cone of finite type cluster fans}

The linear dependences of \cref{lem:linearDependencegvectorsCA} provide a redundant description of the type cone of the cluster fan~$\gvectorFan[\B_\circ]$.
We denote by~$\b{n}(\B_\circ, x, x')$ the normal vector of the inequality of the type cone corresponding to two exchangeable cluster variables~$x$ to~$x'$ (it is well defined by \cref{prop:uniqueExchangePropertyCA}).
In other words, depending on which of the two linear dependences of \cref{lem:linearDependencegvectorsCA} holds, we have
\[
\b{n}(\B_\circ, x, x') \eqdef \b{f}_x + \b{f}_{x'} - \sum_{\substack{y \in \cluster \cap \cluster' \\ b_{xy} < 0}} -b_{xy} \, \b{f}_y
\qquad\text{or}\qquad
\b{n}(\B_\circ, x, x') \eqdef \b{f}_x + \b{f}_{x'} - \sum_{\substack{y \in \cluster \cap \cluster' \\ b_{xy} > 0}} b_{xy} \, \b{f}_y,
\]
where~$(\b{f}_x)_{x \in \principalVariables}$ denotes the canonical basis of~$\R^{\principalVariables}$.
We obtain the following statement.

\begin{corollary}
\label{coro:typeConeCA}
For any finite type exchange matrix~$\B_\circ$, the type cone of the cluster fan~$\gvectorFan[\B_\circ]$ is given by
\[
\typeCone \big( \gvectorFan[\B_\circ] \big) = \set{\b{h} \in \R^{\principalVariables}}{\dotprod{\b{n}(\B_\circ, x, x')}{\b{h}} > 0 \text{ for all exchangeable cluster variables } x, x'}.
\]
\end{corollary}

\begin{example}
\label{exm:typeConeCA}
Consider the cluster fans illustrated in \cref{fig:clusterFans}.
The type cone of the left fan of \cref{fig:clusterFans} lives in~$\R^9$ and has a lineality space of dimension~$3$.
It has $6$ facet-defining inequalities (given below), which correspond to the mesh mutations of \cref{thm:extremalExchangeablePairsCA} as illustrated in \cref{fig:labelFacetDefiningInequalititiesCA1}.

\[
\begin{array}{r|cccccccccc}
\text{variables} & \rotatebox{70}{$x_1$} & \rotatebox{70}{$x_2$} & \rotatebox{70}{$x_3$} & \rotatebox{70}{$\frac{x_2 + x_3}{x_1}$} & \rotatebox{70}{$\frac{x_1 + x_3}{x_2}$} & \rotatebox{70}{$\frac{x_1 + x_2}{x_3}$} & \rotatebox{70}{$\frac{x_1 + x_2 + x_3}{x_2 x_3}$} & \rotatebox{70}{$\frac{x_1 + x_2 + x_3}{x_1 x_3}$} & \rotatebox{70}{$\frac{x_1 + x_2 + x_3}{x_1 x_2}$} \\[.2cm]
\text{$\b{g}$-vectors} & \compactVectorT{1}{0}{0} & \compactVectorT{0}{1}{0} & \compactVectorT{0}{0}{1} & \compactVectorT{-1}{0}{1} & \compactVectorT{1}{-1}{0} & \compactVectorT{0}{1}{-1} & \compactVectorT{0}{0}{-1} & \compactVectorT{-1}{0}{0} & \compactVectorT{0}{-1}{0} \\[.6cm]
\text{facet} 		& 1 & 0 & 0 & 0 & -1 & 0 & 0 & 0 & 1 & \red \circled{A} \\
\text{defining}		& 0 & 0 & -1 & 1 & 1 & 0 & 0 & 0 & -1 & \red \circled{B} \\
\text{inequalities}	& 0 & 0 & 1 & -1 & 0 & 0 & 0 & 1 & 0 & \red \circled{C} \\
 					& 0 & -1 & 0 & 1 & 0 & 1 & 0 & -1 & 0 & \red \circled{D} \\
 					& 0 & 1 & 0 & 0 & 0 & -1 & 1 & 0 & 0 & \red \circled{E} \\
					& -1 & 0 & 0 & 0 & 1 & 1 & -1 & 0 & 0 & \red \circled{F} \\[.2cm]
\end{array}
\]




\begin{figure}[h]
	\capstart
    \begin{adjustbox}{center}
		\newcommand{\fitClusterVariable}[1]{\parbox[c][.7cm][c]{1.5cm}{\centering $#1$}}
        \begin{tikzpicture}
        	\matrix (m) [matrix of math nodes, row sep=.4cm, column sep=-.1cm, nodes={anchor=center, align=center, inner sep=0pt}]{
        		\fitClusterVariable{\phantom{1}} & \node (e1) {\red \circled{E}}; & \fitClusterVariable{\frac{x_1 + x_2 + x_3}{x_2 x_3}} && \fitClusterVariable{x_3} & \node (c) {\red \circled{C}}; & \fitClusterVariable{\frac{x_1 + x_2 + x_3}{x_1 x_3}} && \fitClusterVariable{x_1} & \node (a2) {\red \circled{A}}; & \fitClusterVariable{\phantom{1}} \\
        		& \fitClusterVariable{\frac{x_1 + x_2}{x_3}} & \node (f1) {\red \circled{F}}; & \fitClusterVariable{\frac{x_1 + x_3}{x_2}} & \node (b) {\red \circled{B}}; & \fitClusterVariable{\frac{x_2 + x_3}{x_1}} & \node (d) {\red \circled{D}}; & \fitClusterVariable{\frac{x_1 + x_2}{x_3}} & \node (f2) {\red \circled{F}}; & \fitClusterVariable{\frac{x_1 + x_3}{x_2}} \\
        		\fitClusterVariable{\phantom{1}} && \fitClusterVariable{x_1} & \node (a1) {\red \circled{A}}; & \fitClusterVariable{\frac{x_1 + x_2 + x_3}{x_1 x_2}} && \fitClusterVariable{x_2} & \node (e2) {\red \circled{E}}; & \fitClusterVariable{\frac{x_1 + x_2 + x_3}{x_2 x_3}} && \fitClusterVariable{\phantom{1}} \\};
        	\draw[densely dotted, thick] (m-1-1) -- (m-2-2);
        	\draw[densely dotted, thick] (m-3-1) -- (m-2-2);
        	\draw[->] (m-2-2) -- (m-1-3);
        	\draw[->] (m-2-2) -- (m-3-3);
        	\draw[->] (m-1-3) -- (m-2-4);
        	\draw[->] (m-3-3) -- (m-2-4);
        	\draw[->] (m-2-4) -- (m-1-5);
        	\draw[->] (m-2-4) -- (m-3-5);
        	\draw[->] (m-1-5) -- (m-2-6);
        	\draw[->] (m-3-5) -- (m-2-6);
        	\draw[->] (m-2-6) -- (m-1-7);
        	\draw[->] (m-2-6) -- (m-3-7);
        	\draw[->] (m-1-7) -- (m-2-8);
        	\draw[->] (m-3-7) -- (m-2-8);
        	\draw[->] (m-2-8) -- (m-1-9);
        	\draw[->] (m-2-8) -- (m-3-9);
        	\draw[->] (m-1-9) -- (m-2-10);
        	\draw[->] (m-3-9) -- (m-2-10);
        	\draw[densely dotted, thick] (m-2-10) -- (m-1-11);
        	\draw[densely dotted, thick] (m-2-10) -- (m-3-11);
        	\draw[red] (a1) -- (m-3-3);
        	\draw[red] (a1) -- (m-3-5);
        	\draw[red] (a1) -- (m-2-4);
        	\draw[red] (a2) -- (m-1-9);
        	\draw[red, densely dotted, thick] (a2) -- (m-1-11);
        	\draw[red] (a2) -- (m-2-10);
        	\draw[red] (b) -- (m-2-4);
        	\draw[red] (b) -- (m-2-6);
        	\draw[red] (b) -- (m-1-5);
        	\draw[red] (b) -- (m-3-5);
        	\draw[red] (c) -- (m-1-5);
        	\draw[red] (c) -- (m-1-7);
        	\draw[red] (c) -- (m-2-6);
        	\draw[red] (d) -- (m-2-6);
        	\draw[red] (d) -- (m-2-8);
        	\draw[red] (d) -- (m-1-7);
        	\draw[red] (d) -- (m-3-7);
        	\draw[red, densely dotted, thick] (e1) -- (m-1-1);
        	\draw[red] (e1) -- (m-1-3);
        	\draw[red] (e1) -- (m-2-2);
        	\draw[red] (e2) -- (m-3-7);
        	\draw[red] (e2) -- (m-3-9);
        	\draw[red] (e2) -- (m-2-8);
        	\draw[red] (f1) -- (m-2-2);
        	\draw[red] (f1) -- (m-2-4);
        	\draw[red] (f1) -- (m-1-3);
        	\draw[red] (f1) -- (m-3-3);
        	\draw[red] (f2) -- (m-2-8);
        	\draw[red] (f2) -- (m-2-10);
        	\draw[red] (f2) -- (m-1-9);
        	\draw[red] (f2) -- (m-3-9);
        \end{tikzpicture}
    \end{adjustbox}
	\caption{The facet-defining inequalities of the type cone~$\typeCone \big( \gvectorFan[\B_\circ] \big)$ of the cluster fan correspond to the mesh mutations described in \cref{thm:extremalExchangeablePairsCA}. See \cref{sec:clusterCategories} for a representation theoretic \mbox{interpretation}.}
	\label{fig:labelFacetDefiningInequalititiesCA1}
\end{figure}

\noindent
The type cone of the right fan of \cref{fig:clusterFans} lives in~$\R^{12}$ and has a lineality space of dimension~$3$.
It has $9$ facet-defining inequalities (given below), which correspond to the mesh mutations of \cref{thm:extremalExchangeablePairsCA} as illustrated in \cref{fig:labelFacetDefiningInequalititiesCA2}.

\centerline{$
\begin{array}{r|c@{\;\;}c@{\;\;}c@{\;\;}c@{\;\;}c@{\;\;}c@{\;}c@{\;}c@{\;}c@{}c@{}c@{}cc}
\text{variables} & \rotatebox{70}{$x_1$} & \rotatebox{70}{$x_2$} & \rotatebox{70}{$x_3$} & \rotatebox{70}{$\frac{x_2 + x_3}{x_1}$} & \rotatebox{70}{$\frac{x_1 + x_3}{x_2}$} & \rotatebox{70}{$\frac{x_1^2 + x_2^2}{x_3}$} & \rotatebox{70}{$\frac{x_1^2 + x_2^2 + x_1 x_3}{x_2 x_3}$} & \rotatebox{70}{$\frac{x_1^2 + x_2^2 + x_2 x_3}{x_1 x_3}$} & \rotatebox{70}{$\frac{x_1 + x_2 + x_3}{x_1 x_2}$} & \rotatebox{70}{$\frac{x_1^2 + x_2^2 + 2 x_1 x_3 + x_3^2}{x_2^2 x_3}$} & \rotatebox{70}{$\frac{x_1^2 + x_2^2 + 2 x_2 x_3 + x_3^2}{x_1^2 x_3}$} & \rotatebox{70}{$\frac{x_1^2 + x_2^2 + x_1 x_3 + x_2 x_3}{x_1 x_2 x_3}$} \\[.2cm]
\text{$\b{g}$-vectors} & \compactVectorT{1}{0}{0} & \compactVectorT{0}{1}{0} & \compactVectorT{0}{0}{1} & \compactVectorT{-1}{0}{1} & \compactVectorT{1}{-1}{0} & \compactVectorT{0}{2}{-1} & \compactVectorT{0}{1}{-1} & \compactVectorT{-1}{1}{0} & \compactVectorT{0}{-1}{0} & \compactVectorT{0}{0}{-1} & \compactVectorT{-2}{0}{1} & \compactVectorT{-1}{0}{0}  \\[.6cm]
\text{facet}		& 1 & 0 & 0 & 0 & -1 & 0 & 0 & 0 & 1 & 0 & 0 & 0 & \red \circled{A} \\
\text{defining}		& 0 & 0 & -1 & 1 & 1 & 0 & 0 & 0 & -1 & 0 & 0 & 0 & \red \circled{B} \\
\text{inequalities}	& 0 & 0 & 1 & -2 & 0 & 0 & 0 & 0 & 0 & 0 & 1 & 0 & \red \circled{C} \\
					& 0 & -1 & 0 & 1 & 0 & 0 & 0 & 1 & 0 & 0 & -1 & 0 & \red \circled{D} \\
					& 0 & 0 & 0 & 0 & 0 & 1 & 0 & -2 & 0 & 0 & 1 & 0 & \red \circled{E} \\
					& 0 & 1 & 0 & 0 & 0 & 0 & 0 & -1 & 0 & 0 & 0 & 1 & \red \circled{F} \\
					& 0 & 0 & 0 & 0 & 0 & -1 & 1 & 1 & 0 & 0 & 0 & -1 & \red \circled{G} \\
					& 0 & 0 & 0 & 0 & 0 & 1 & -2 & 0 & 0 & 1 & 0 & 0 & \red \circled{H} \\
					& -1 & 0 & 0 & 0 & 1 & 0 & 1 & 0 & 0 & -1 & 0 & 0 & \red \circled{I} \\[.2cm]
\end{array}
$}

\begin{figure}[h]
	\capstart
    \begin{adjustbox}{center}
		\newcommand{\fitClusterVariable}[1]{\parbox[c][.7cm][c]{2.3cm}{\centering $#1$}}
        \begin{tikzpicture}
        	\matrix (m) [matrix of math nodes, row sep=.4cm, column sep=-.7cm, nodes={anchor=center, align=center, inner sep=0pt}]{
        		& \fitClusterVariable{\frac{x_1^2 + x_2^2 + 2 x_1 x_3 + x_3^2}{x_2^2 x_3}} && \fitClusterVariable{x_3} & \node (c) {\red \circled{C}}; & \fitClusterVariable{\frac{x_1^2 + x_2^2 + 2 x_2 x_3 + x_3^2}{x_1^2 x_3}} & \node (e) {\red \circled{E}}; & \fitClusterVariable{\frac{x_1^2 + x_2^2}{x_3}} & \node (h) {\red \circled{H}}; & \fitClusterVariable{\frac{x_1^2 + x_2^2 + 2 x_1 x_3 + x_3^2}{x_2^2 x_3}} && \fitClusterVariable{\phantom{1}} \\
        		\fitClusterVariable{\phantom{1}} & \node (i1) {\red \circled{I}}; & \fitClusterVariable{\frac{x_1 + x_3}{x_2}} & \node (b) {\red \circled{B}}; & \fitClusterVariable{\frac{x_2 + x_3}{x_1}} & \node (d) {\red \circled{D}}; & \fitClusterVariable{\frac{x_1^2 + x_2^2 + x_2 x_3}{x_1 x_3}} & \node (g) {\red \circled{G}}; & \fitClusterVariable{\frac{x_1^2 + x_2^2 + x_1 x_3}{x_2 x_3}} & \node (i2) {\red \circled{I}}; & \fitClusterVariable{\frac{x_1 + x_3}{x_2}} & \\
        		& \fitClusterVariable{x_1} & \node (a1) {\red \circled{A}}; & \fitClusterVariable{\frac{x_1 + x_2 + x_3}{x_1 x_2}} && \fitClusterVariable{x_2} & \node (f) {\red \circled{F}}; & \fitClusterVariable{\frac{x_1^2 + x_2^2 + x_1 x_3 + x_2 x_3}{x_1 x_2 x_3}} && \fitClusterVariable{x_1} & \node (a2) {\red \circled{A}}; & \fitClusterVariable{\phantom{1}} \\};
        	\draw[densely dotted, thick] (m-2-1) -- (m-1-2);
        	\draw[densely dotted, thick] (m-2-1) -- (m-3-2);
        	\draw[->] (m-1-2) -- (m-2-3);
        	\draw[->] (m-3-2) -- (m-2-3);
        	\draw[->] (m-2-3) -- (m-1-4);
        	\draw[->] (m-2-3) -- (m-3-4);
        	\draw[->] (m-1-4) -- (m-2-5);
        	\draw[->] (m-3-4) -- (m-2-5);
        	\draw[->] (m-2-5) -- (m-1-6);
        	\draw[->] (m-2-5) -- (m-3-6);
        	\draw[->] (m-1-6) -- (m-2-7);
        	\draw[->] (m-3-6) -- (m-2-7);
        	\draw[->] (m-2-7) -- (m-1-8);
        	\draw[->] (m-2-7) -- (m-3-8);
        	\draw[->] (m-1-8) -- (m-2-9);
        	\draw[->] (m-3-8) -- (m-2-9);
        	\draw[->] (m-2-9) -- (m-1-10);
        	\draw[->] (m-2-9) -- (m-3-10);
        	\draw[->] (m-1-10) -- (m-2-11);
        	\draw[->] (m-3-10) -- (m-2-11);
        	\draw[densely dotted, thick] (m-2-11) -- (m-1-12);
        	\draw[densely dotted, thick] (m-2-11) -- (m-3-12);
        	\draw[red] (a1) -- (m-3-2);
        	\draw[red] (a1) -- (m-3-4);
        	\draw[red] (a1) -- (m-2-3);
        	\draw[red] (a2) -- (m-3-10);
        	\draw[red] (a2) -- (m-3-12);
        	\draw[red] (a2) -- (m-2-11);
        	\draw[red] (b) -- (m-2-3);
        	\draw[red] (b) -- (m-2-5);
        	\draw[red] (b) -- (m-1-4);
        	\draw[red] (b) -- (m-3-4);
        	\draw[red] (c) -- (m-1-4);
        	\draw[red] (c) -- (m-1-6);
        	\draw[red] (c) -- (m-2-5);
        	\draw[red] (d) -- (m-2-5);
        	\draw[red] (d) -- (m-2-7);
        	\draw[red] (d) -- (m-1-6);
        	\draw[red] (d) -- (m-3-6);
        	\draw[red] (e) -- (m-1-6);
        	\draw[red] (e) -- (m-1-8);
        	\draw[red] (e) -- (m-2-7);
        	\draw[red] (f) -- (m-3-6);
        	\draw[red] (f) -- (m-3-8);
        	\draw[red] (f) -- (m-2-7);
        	\draw[red] (g) -- (m-2-7);
        	\draw[red] (g) -- (m-2-9);
        	\draw[red] (g) -- (m-1-8);
        	\draw[red] (g) -- (m-3-8);
        	\draw[red] (h) -- (m-1-8);
        	\draw[red] (h) -- (m-1-10);
        	\draw[red] (h) -- (m-2-9);
        	\draw[red] (i1) -- (m-2-1);
        	\draw[red] (i1) -- (m-2-3);
        	\draw[red] (i1) -- (m-1-2);
        	\draw[red] (i1) -- (m-3-2);
        	\draw[red] (i2) -- (m-2-9);
        	\draw[red] (i2) -- (m-2-11);
        	\draw[red] (i2) -- (m-1-10);
        	\draw[red] (i2) -- (m-3-10);
        \end{tikzpicture}
    \end{adjustbox}
	\caption{The facet-defining inequalities of the type cone~$\typeCone \big( \gvectorFan[\B_\circ] \big)$ of the cluster fan correspond to the mesh mutations described in \cref{thm:extremalExchangeablePairsCA}. See \cref{sec:clusterCategories} for a representation theoretic \mbox{interpretation}.}
	\label{fig:labelFacetDefiningInequalititiesCA2}
\end{figure}
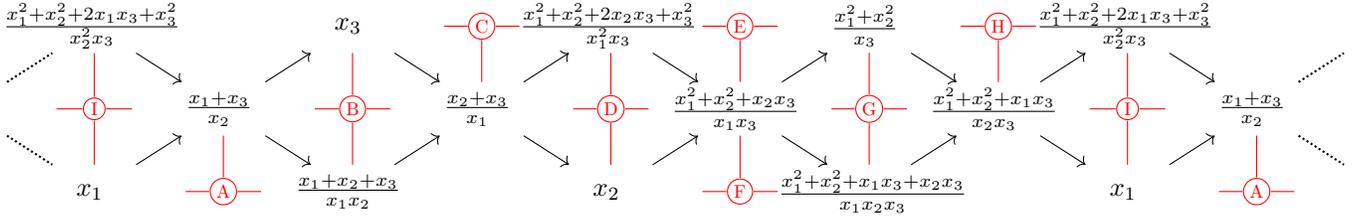
\end{example}

In order to describe the facets of this type cone, we need the following special mutations.

\begin{definition}
\label{def:meshMutation}
The mutation of a seed~$\seed = (\B, \cluster)$ in the direction of a cluster variable~$x \in \cluster$ is a \defn{mesh mutation} that \defn{starts} (resp.~\defn{ends}) at~$x$ if the entries~$b_{xy}$ for~$y \in \cluster$ are all non-negative (resp.~all non-positive).
A mesh mutation is \defn{initial} if it ends at a cluster variable of an initial seed.
We denote by~$\meshes$ the set of all pairs~$\{x,x'\}$ where~$x$ and~$x'$ are two cluster variables of~$\principalClusterAlgebra$ which are exchangeable via a non-initial mesh mutation.
\end{definition}

The following statement is proved in \cref{coro:number of positive meshes}.

\begin{lemma}
\label{lem:bijectionMeshMutations}
We have that~$|\principalVariables| = |\meshes| + n$.
\end{lemma}

The following statement describes the linear dependence in the mesh mutations. It follows from \cref{lem:linearDependencegvectorsCA,def:meshMutation}.

\begin{lemma}
\label{lem:dependenceMesh}
Consider two adjacent seeds~${(\B, \cluster)}$ and~${(\B', \cluster')}$ with~$\cluster \ssm \{x\} = \cluster' \ssm \{x'\}$ connected by a mesh mutation.
If the mesh mutation is initial, then
\[
\gvectorFull{\B_\circ}{x} + \gvectorFull{\B_\circ}{x'} = 0.
\]
Otherwise, the $\b{g}$-vectors of~$\cluster \cup \cluster'$ with respect to~$\B_\circ$ satisfy the linear dependence
\[
\gvectorFull{\B_\circ}{x} + \gvectorFull{\B_\circ}{x'} = \sum_{y \in \cluster \cap \cluster'} |b_{xy}| \, \gvectorFull{\B_\circ}{y}.
\]
\end{lemma}

For~$\{x,x'\} \in \meshes$ and~$y \in \principalVariables$, we denote by~$\coefficient[y][x][x']$ the coefficient of~$\gvectorFull{\B_\circ}{y}$ in the linear dependence between the $\b{g}$-vectors~$\gvectorFull{\B_\circ}{x}$ and~$\gvectorFull{\B_\circ}{x'}$. In other words, according to \cref{lem:dependenceMesh}, if~${(\B, \cluster)}$ and~${(\B', \cluster')}$ are two adjacent seeds with~$\cluster \ssm \{x\} = \cluster' \ssm \{x'\}$, we have~$\coefficient[y][x][x'] = |b_{xy}|$ for~$y \in \cluster \cap \cluster'$ and~$\coefficient[y][x][x'] = 0$ otherwise.
The following statement will be shown in \cref{corollary:proof of prop meshMutations}.

\begin{proposition}
\label{prop:meshMutations}
For any finite type exchange matrix~$\B_\circ$ (acyclic or not, simply-laced or not), the linear dependence between the $\b{g}$-vectors of any mutation can be decomposed into positive combinations of linear dependences between $\b{g}$-vectors of non-initial mesh mutations.
\end{proposition}

We derive from \cref{prop:meshMutations} the following description of the facets of the type cone of the cluster fan.

\begin{theorem}
\label{thm:extremalExchangeablePairsCA}
For any finite type exchange matrix~$\B_\circ$ (acyclic or not, simply-laced or not), the non-initial mesh mutations precisely correspond to the extremal exchangeable pairs of the cluster fan~$\gvectorFan[\B_\circ]$.
\end{theorem}

\begin{proof}
Since~$\gvectorFan[\B_\circ]$ has dimension~$n$ and~$|\principalVariables|$ rays (corresponding to the cluster variables), we know from \cref{rem:dimTypeCone} that there are at least~$|\principalVariables| - n$ extremal exchangeable pairs.
Moreover, \cref{prop:meshMutations,lem:bijectionMeshMutations} ensure that there are at most $|\meshes| = |\principalVariables| - n$ extremal exchangeable pairs (corresponding to the non-initial mesh mutations).
We conclude that all non-initial mesh mutations are extremal exchangeable pairs.
\end{proof}

Our next statement follows from the end of the previous proof.

\begin{corollary}
\label{coro:simplicialTypeConeCA}
The type cone~$\typeCone \big( \gvectorFan[\B_\circ] \big)$ is simplicial.
\end{corollary}

As a follow-up to \cref{thm:extremalExchangeablePairsCA}, we conjecture the following surprising property.

\begin{conjecture}
Any positive dual $\b{c}$-vector supports exactly one extremal exchangeable pair.
\end{conjecture}

Combining \cref{coro:simplicialTypeCone,coro:simplicialTypeConeCA,thm:extremalExchangeablePairsCA}, we derive the following description of all polytopal realizations of the cluster fan~$\gvectorFan[\B_\circ]$.
This result was stated in~\cite{BazierMatteDouvilleMousavandThomasYildirim} in the special situation of acyclic seeds in simply-laced types.
Note that our proof is quite different from~\cite{BazierMatteDouvilleMousavandThomasYildirim} as it relies on our type cone approach.

\begin{theorem}
\label{thm:allPolytopalRealizationsCA}
For any finite type exchange matrix~$\B_\circ$ (acyclic or not, simply-laced or not), and for any~$\b{\ell} \in \R_{>0}^{\meshes}$, the polytope
\[
R_{\b{\ell}}(\B_\circ) \eqdef \Bigset{\b{z} \in \R^{\principalVariables}}{\b{z} \ge 0 \text{ and } \b{z}_x + \b{z}_{x'} - \sum_{y \in \principalVariables} \coefficient[y][x][x'] \, \b{z}_{y} = \b{\ell}_{\{x,x'\}} \text{ for all } \{x,x'\} \in \meshes}
\]
is a generalized associahedron, whose normal fan is the cluster fan~$\gvectorFan[\B_\circ]$.
Moreover, the polytopes~$R_\b{\ell}(\B_\circ)$ for~$\b{\ell} \in \R_{>0}^{\meshes}$ describe all polytopal realizations of~$\gvectorFan[\B_\circ]$.
\end{theorem}

\begin{example}
To complement \cref{subsec:typeConeAsso}, let us translate to diagonals and triangulations the result of \cref{thm:allPolytopalRealizationsCA} in type~$A$.
Let~$n \ge 3$.
Consider a convex $(n+3)$-gon whose vertices are labeled modulo~$n+3$ and whose internal and boundary diagonals are denoted by~$\Delta(n) \eqdef \binom{\Z/(n+3)\Z}{2}$.
Consider a triangulation~$T$ of this polygon, formed by the $n+3$ boundary edges and $n$ internal diagonals of the $(n+3)$-gon.
Then for any~$\b{\ell} \in \R_{>0}^{\Delta(n) \ssm T}$, the polytope
\[
R_\b{\ell}(T) \eqdef \Bigset{\b{z} \in \R^{\Delta(n)}}{\begin{array}{l} \b{z} \ge 0 \quad\text{ and }\quad \b{z}_{(a,a+1)} = 0 \text{ for all } a \in \Z/(n+3)\Z \quad\text{ and }\quad \\ \b{z}_{(a,b)} + \b{z}_{(a-1,b-1)} - \b{z}_{(a,b-1)} - \b{z}_{(a-1,b)} = \b{\ell}_{(a,b)} \text{ for all } (a,b) \notin T \end{array}}
\]
is an associahedron whose normal fan is the $\b{g}$-vector fan~$\gvectorFan[T]$ with respect to the triangulation~$T$.
Moreover, the polytopes~$R_\b{\ell}(T)$ for~$\b{\ell} \in \R_{>0}^{\Delta(n) \ssm T}$ describe all polytopal realizations of~$\gvectorFan[T]$.
\end{example}



\subsection{Non-kissing complexes and gentle associahedra}
\label{subsec:typeConeNKC}

Gentle associahedra were constructed by Y.~Palu, V.~Pilaud and P.-G.~Plamondon~\cite{PaluPilaudPlamondon-nonkissing} in the context of support $\tau$-tilting for gentle algebras.
For a given $\tau$-tilting finite gentle quiver~$\quiver$ (defined in the next section), the $\quiver$-associahedron~$\Asso[\quiver]$ is a simple polytope which encodes certain representations of~$\quiver$ and their $\tau$-tilting relations.
Combinatorially, the $\quiver$-associahedron is a polytopal realization of the non-kissing complex of~$\quiver$, defined as the simplicial complex of all collections of walks on the blossoming quiver~$\quiver\blossom$ which are pairwise non-kissing.
The non-kissing complex encompasses two families of simplicial complexes studied independently in the literature: on the one hand the grid associahedra introduced by T.~K.~Petersen, P.~Pylyavskyy and D.~Speyer in~\cite{PetersenPylyavskyySpeyer} for a staircase shape, studied by F.~Santos, C.~Stump and V.~Welker~\cite{SantosStumpWelker} for rectangular shapes, and extended by T.~McConville in~\cite{McConville} for arbitrary grid shapes; and on the other hand the Stokes polytopes and accordion associahedra studied by Y.~Baryshnikov~\cite{Baryshnikov}, F.~Chapoton~\cite{Chapoton-quadrangulations}, A.~Garver and T.~McConville~\cite{GarverMcConville} and T.~Manneville and V.~Pilaud~\cite{MannevillePilaud-accordion}.
These two families naturally extend the classical associahedron, obtained from a line quiver.
Non-kissing complexes are geometrically realized by polytopes called gentle associahedra, whose normal fan is called the non-kissing fan: its rays correspond to walks in the quiver and its cones are generated by the non-kissing walks.
In this section, we describe the type cone of the non-kissing fan of a quiver~$\quiver$ with no self-kissing walks.


\subsubsection{Non-kissing complex and non-kissing fan of a gentle quiver}
\label{subsubsec:nonkissingComplex}

We present the definitions and properties of the non-kissing complex of a gentle quiver, following the presentation of~\cite{PaluPilaudPlamondon-nonkissing}.
See also~\cite{BrustleDouvilleMousavandThomasYildirim} for an alternative presentation.

\para{Gentle quivers}
Consider a \defn{bound quiver}~${\quiver = (Q,I)}$, formed by a finite quiver~$Q = (Q_0, Q_1, s, t)$ and an ideal~$I$ of the path algebra~$kQ$ (the~$k$-vector space generated by all paths in~$Q$, including vertices as paths of length zero, with multiplication induced by concatenation of paths) such that~$I$ is generated by linear combinations of paths of length at least two, and $I$ contains all sufficiently large paths. See~\cite{AssemSimsonSkowronski} for background.

Following M.~Butler and C.~Ringel~\cite{ButlerRingel}, we say that~$\quiver$ is a \defn{gentle bound quiver} when
\begin{enumerate}[(i)]
\item each vertex~$a \in Q_0$ has at most two incoming and two outgoing arrows,
\item the ideal~$I$ is generated by paths of length exactly two,
\item for any arrow~$\beta \in Q_1$, there is at most one arrow~$\alpha \in Q_1$ such that~$t(\alpha) = s(\beta)$ and~${\alpha\beta\notin I}$ (resp.~$\alpha\beta \in I$) and at most one arrow~$\gamma \in Q_1$ such that~$t(\beta) = s(\gamma)$~and~${\beta\gamma\notin I}$~(resp.~${\beta\gamma \in I}$).
\end{enumerate}
The algebra~$kQ/I$ is called a \defn{gentle algebra}.

The \defn{blossoming quiver}~$\quiver\blossom$ of a gentle quiver~$\quiver$ is the gentle quiver obtained by completing all vertices of~$\quiver$ with additional incoming or outgoing \defn{blossoms} such that all vertices of~$\quiver$ become $4$-valent.
For instance, \cref{fig:exmNonkissing},(left) shows a blossoming quiver: the vertices of~$\quiver$ appear in black, while the blossom vertices of~$\quiver\blossom$ appear in white.
We now always assume that~$\quiver$ is a gentle quiver with blossoming quiver~$\quiver\blossom$.

\para{Strings and walks}
A \defn{string} in~$\quiver = (Q,I)$ is a word of the form
\(
\rho = \alpha_1^{\varepsilon_1}\alpha_2^{\varepsilon_2}\cdots \alpha_\ell^{\varepsilon_\ell},
\)
where
	\begin{enumerate}[(i)]
	\item $\alpha_i \in Q_1$ and~$\varepsilon_i \in \{-1,1\}$ for all~$i \in [\ell]$,
	\item $t(\alpha_i^{\varepsilon_i}) = s(\alpha_{i+1}^{\varepsilon_{i+1}})$ for all~$i \in [\ell-1]$,
	\item there is no path~$\pi \in I$ such that~$\pi$ or~$\pi^{-1}$ appears as a factor of~$\rho$, and
	\item $\rho$ is reduced, in the sense that no factor~$\alpha\alpha^{-1}$ or~$\alpha^{-1}\alpha$ appears in~$\rho$, for~$\alpha \in Q_1$.
	\end{enumerate}
The integer~$\ell$ is called the \defn{length} of the string~$\rho$.
For each vertex~$a \in Q_0$, there is also a \defn{string of length zero}, denoted by~$\varepsilon_a$, that starts and ends at~$a$.
We implicitly identify the two inverse strings~$\rho$ and~$\rho^{-1}$, and call it an \defn{undirected string} of~$\quiver$.
Let~$\strings$ denote the set of strings of~$\quiver$.

A \defn{walk} of~$\quiver$ is a maximal string of its blossoming quiver~$\quiver\blossom$ (meaning that each endpoint is a blossom).
As for strings, we implicitly identify the two inverse walks~$\omega$ and~$\omega^{-1}$, and call it an \defn{undirected walk} of~$\quiver$.
Let~$\walks$ denote the set of walks of~$\quiver$.

A \defn{substring} of a walk~$\omega = \alpha_1^{\varepsilon_1} \cdots \alpha_\ell^{\varepsilon_\ell}$ of~$\quiver$ is a string~$\sigma = \alpha_{i+1}^{\varepsilon_{i+1}} \cdots \alpha_{j-1}^{\varepsilon_{j-1}}$ of~$\quiver$ for some indices~$1 \le i < j \le \ell$. Note that by definition,
\begin{itemize}
\item the endpoints of~$\sigma$ are not allowed to be blossom endpoints of~$\omega$,
\item the position of~$\sigma$ as a factor of~$\omega$ matters (the same string at a different position is considered a different substring).
\item the string~$\varepsilon_a$ is a substring of~$\omega$ for each occurence of~$a$ as a vertex of~$\omega$ (take~$j = i+1$).
\end{itemize}
We denote by~$\Sigma(\omega)$ the set of substrings of~$\omega$.
We say that the substring~$\sigma = \alpha_{i+1}^{\varepsilon_{i+1}} \cdots \alpha_{j-1}^{\varepsilon_{j-1}}$ is \defn{at the bottom} (resp.~\defn{on top}) of the walk~$\omega = \alpha_1^{\varepsilon_1} \cdots \alpha_\ell^{\varepsilon_\ell}$ if~$\varepsilon_i = 1$ and~$\varepsilon_j = -1$ (resp.~if~$\varepsilon_i = -1$ and~$\varepsilon_j = 1$).
In other words the two arrows of~$\omega$ incident to the endpoints of~$\sigma$ point towards~$\sigma$ (resp.~outwards from~$\sigma$).
We denote by~$\Sigma_\bottom(\omega)$ and~$\Sigma_\top(\omega)$ the sets of bottom and top substrings of~$\omega$ respectively.
We use the same notation for undirected walks (of course, substrings of an undirected walk are undirected).

A \defn{peak} (resp.~\defn{deep}) of a walk~$\omega$ is a substring of~$\omega$ of length zero which is on top (resp.~at the bottom) of~$\omega$.
A walk~$\omega$ is \defn{straight} if it has no peak or deep (\ie if~$\omega$ or~$\omega^{-1}$ is a path in~$\quiver\blossom$), and \defn{bending} otherwise.
We denote by~$\peaks{\omega}$ (resp.~$\deeps{\omega}$) the multisets of vertices of peaks (resp.~deeps) of~$\omega$.

\begin{figure}[t]
	\capstart
	\centerline{\includegraphics[width=\textwidth]{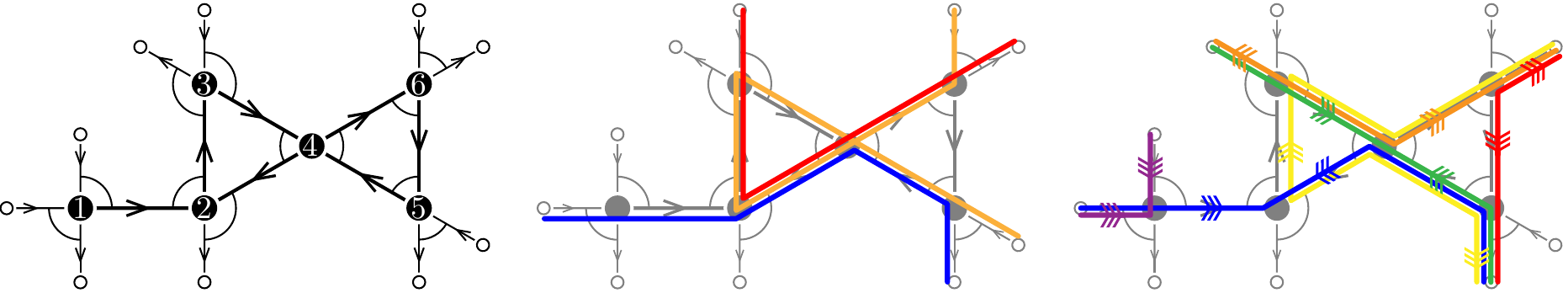}}
	\caption{A blossoming quiver (left), three pairwise kissing walks (middle), and a non-kissing facet (right). Illustration from~\cite{PaluPilaudPlamondon-nonkissing}.}
	\label{fig:exmNonkissing}
\end{figure}

\para{Non-kissing complex}
Let~$\omega$ and~$\omega'$ be two undirected walks on~$\quiver$.
We say that~$\omega$ \defn{kisses}~$\omega'$ if ${\Sigma_\top(\omega) \cap \Sigma_\bottom(\omega')} \ne \varnothing$.
In other words, $\omega$ and~$\omega'$ share a common substring~$\sigma$, and both arrows of~$\omega$ (resp.~of~$\omega'$) incident to the endpoints of~$\sigma$ but not in~$\sigma$ are outgoing (resp.~incoming) at the endpoints of~$\sigma$.
See \cref{fig:kissingCrossing} for a schematic representation and \cref{fig:exmNonkissing}\,(middle) where the three walks are pairwise kissing.
We say that~$\omega$ and~$\omega'$ are \defn{kissing} if~$\omega$ kisses~$\omega'$ or~$\omega'$ kisses~$\omega$ (or both).
Note that we authorize the situation where the common finite substring is reduced to a vertex~$a$, that~$\omega$ can kiss~$\omega'$ several times, that~$\omega$ and~$\omega'$ can mutually kiss, and that~$\omega$ can kiss itself.
For example, the orange walk in \cref{fig:exmNonkissing}\,(middle) is self-kissing (at its self-intersection).
We say that a walk is \defn{proper} if it is not straight nor self-kissing, and \defn{improper} otherwise.
We denote by~$\properWalks$ the set of all proper walks of~$\quiver$.

\begin{figure}[h]
	\capstart
	\centerline{\includegraphics[scale=1]{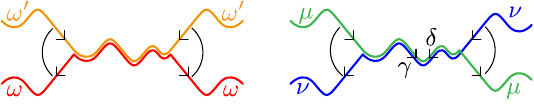}}
	\caption{A schematic representation of kissing and non-kissing walks. The walks $\omega$ and~$\omega'$ kiss (left) while the walks~$\mu$ and~$\nu$ share a common substring but do not kiss (right). Illustration from~\cite{PaluPilaudPlamondon-nonkissing}.}
	\label{fig:kissingCrossing}
\end{figure}

The (reduced) \defn{non-kissing complex} of~$\quiver$ is the simplicial complex~$\NKC$ whose faces are the collections of pairwise non-kissing proper walks of~$\quiver$.
For example, \cref{fig:exmNonkissing}\,(right) represents a non-kissing facet.
As shown in~\cite[Thm.~2.46]{PaluPilaudPlamondon-nonkissing}, the non-kissing complex is a combinatorial model for the support $\tau$-tilting complex on $\tau$-rigid modules over~$kQ/I$.
The quiver~$\quiver$ is called \defn{$\tau$-tilting finite} or \defn{non-kissing finite} when this complex is finite (in other words, $\quiver$ has finitely may non-kissing walks).

\para{$\b{g}$- and $\b{c}$-vectors}
We now define two families of vectors associated to walks in the non-kissing complex.
Let~$(\b{e}_v)_{v \in Q_0}$ denote the canonical basis of~$\R^{Q_0}$.
For a multiset~$V = \{\!\{v_1, \dots, v_k\}\!\}$ of~$Q_0$, we denote by $\multiplicityVector_V \eqdef \sum_{i \in [k]} \b{e}_{v_i}$.

\begin{definition}[{\cite[Def.~{4.8}]{PaluPilaudPlamondon-nonkissing}}]
The \defn{$\b{g}$-vector} of a walk~$\omega$ is~$\gvector{\omega} \eqdef \multiplicityVector_{\peaks{\omega}} - \multiplicityVector_{\deeps{\omega}}$.
We set~$\gvectors{F} \eqdef \set{\gvector{\omega}}{\omega \in F}$ for a non-kissing facet~$F \in \NKC$.
\end{definition}

For example, the $\b{g}$-vectors of the blue, green, and yellow walks of \cref{fig:exmNonkissing}\,(right) are respectively~$(0, -1, 0, 0, 1, 0)$, $(0, 0, 1, -1, 1, 0)$ and $(0, 0, 0, 0, 1, 0)$.
Note that by definition, the $\gvector{\omega} = 0$ for a straight walk~$\omega$.

To define the other family of vectors, we need to recall the notion of distinguished substring of a walk defined in~\cite[Def.~2.25]{PaluPilaudPlamondon-nonkissing}.
Consider an arrow~$\alpha \in Q_1$ contained in two distinct walks~$\omega, \omega'$ of a non-kissing facet~$F \in \NKC$, and let~$\sigma$ denote the common substring of~$\omega$ and~$\omega'$.
We write~$\omega \prec_\alpha \omega'$ if $\omega$ enters and/or exits~$\sigma$ with arrows in the direction pointed by~$\alpha$, while $\omega'$ enters and/or exits~$\sigma$ with arrows in the direction opposite to~$\alpha$.
For example, in \cref{fig:kissingCrossing}, we have~$\mu \prec_\gamma \nu$ but~$\nu \prec_\delta \mu$.
The \defn{distinguished walk} of the non-kissing facet~$F \in \NKC$ at the arrow~$\alpha$ is the maximum walk of~$F$ for~$\prec_\alpha$.
The \defn{distinguished arrows} of the walk~$\omega$ in the non-kissing facet~$F \in \NKC$ are the arrows where~$\omega$ is the distinguished walk.
It is shown in~\cite[Prop.~2.28]{PaluPilaudPlamondon-nonkissing} that each bending walk in a non-kissing facet has precisely two distinguished arrows, pointing in opposite directions.
For example, we have marked in \cref{fig:exmNonkissing}\,(right) the two distinguished arrows on each walk by triple arrows.
The \defn{distinguished substring}~$\distinguishedString{\omega}{F}$ of the walk~$\omega$ in the non-kissing facet~$F \in \NKC$ is the substring located between its two distinguished arrows.

\begin{definition}[{\cite[Def.~{4.11}]{PaluPilaudPlamondon-nonkissing}}]
The \defn{$\b{c}$-vector} of a walk~$\omega$ in a non-kissing facet~${F \in \NKC}$ is~${\cvector{F}{\omega} \eqdef \distinguishedSign{\omega}{F} \, \multiplicityVector_{\distinguishedString{\omega}{F}}}$, where~$\distinguishedSign{\omega}{F} \in \{-1,1\}$ is positive if~${\distinguishedString{\omega}{F} \in \Sigma_\top(\omega)}$ and negative if~${\distinguishedString{\omega}{F} \in \Sigma_\bottom(\omega)}$.
We set~$\cvectors{F} \eqdef \set{\cvector{F}{\omega}}{\omega \in F}$ for a non-kissing facet~$F \in \NKC$.
\end{definition}

For example, the $\b{c}$-vectors of the blue, green, and yellow walks in the non-kissing facet of \cref{fig:exmNonkissing}\,(right) are respectively~$(0, -1, 0, 0, 0, 0)$, $(0, 0, 0, -1, 0, 0)$ and $(0, 1, 0, 1, 1, 0)$.
As the reader can observe on the example of \cref{fig:exmNonkissing}\,(right), the $\b{g}$- and $\b{c}$-vectors also form dual bases.

\begin{proposition}
\label{prop:gvectorscvectorsDualBasesGentle}
For any non-kissing facet~$F \in \NKC$, the set of $\b{g}$-vectors~$\gvectors{F}$ and the set of $\b{c}$-vectors~$\cvectors{F}$ form dual bases, that is
\(
{\bigdotprod{\gvector{\omega}}{\cvector{F}{\omega'}} = \delta_{\omega=\omega'}}
\)
for any two walks~$\omega, \omega' \in F$.
\end{proposition}

The following result will also be essential in our discussion.
We say that a string~$\sigma$ is \defn{distinguishable} if it is the distinguished substring of a walk in a non-kissing facet, and we denote by~$\distinguishableStrings$ the set of distinguishable strings of~$\quiver$.

\begin{proposition}[{\cite[Prop.~3.68]{PaluPilaudPlamondon-nonkissing}}]
\label{prop:bijectionStringsWalks}
The number of distinguishable strings~$\distinguishableStrings$ and proper walks~$\properWalks$ of the quiver~$\quiver$ are related by
\[
|\distinguishableStrings| + |Q_0| = |\properWalks|.
\]
\end{proposition}

\para{Non-kissing fan and gentle associahedron}
The $\b{g}$-vectors support a complete simplicial fan realization of the non-kissing complex~$\NKC$.
Examples are illustrated in \cref{fig:nonkissingFans}.

\begin{theorem}[{\cite[Thm.~4.17]{PaluPilaudPlamondon-nonkissing}}]
\label{thm:nonkissingFan}
For any non-kissing finite gentle quiver~$\quiver$, the set of cones
\[
\gvectorFan[\quiver] \eqdef \set{\R_{\ge 0} \, \gvectors{F}}{F \text{ non-kissing face of } \NKC}
\]
is a complete simplicial fan of~$\R^{Q_0}$, called \defn{non-kissing fan} of~$\quiver$, which realizes the non-kissing complex~$\NKC$.
\end{theorem}

\begin{figure}[h]
	\capstart
	\centerline{\includegraphics[scale=.45]{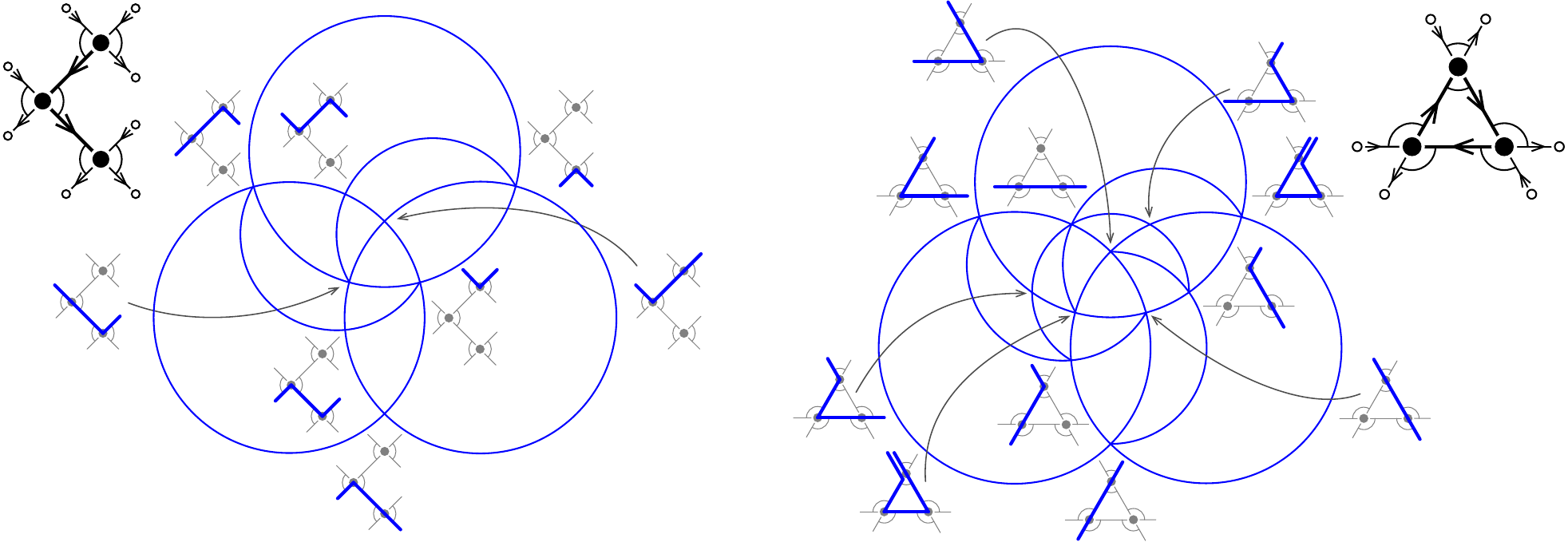}}
	\caption{Two non-kissing fans. As the fans are $3$-dimensional, we intersect them with the sphere and stereographically project them from the direction~$(-1,-1,-1)$. Illustration from~\cite{PaluPilaudPlamondon-nonkissing}.}
	\label{fig:nonkissingFans}
\end{figure}

\enlargethispage{.1cm}
It is proved in~\cite[Thm.~4.27]{PaluPilaudPlamondon-nonkissing} that the non-kissing fan comes from a polytope.
For a walk~$\omega$, denote by~$\KN(\omega)$ the sum over all other walks~$\omega'$ of the number of kisses between~$\omega$~and~$\omega'$.

\begin{theorem}[{\cite[Thm.~4.27]{PaluPilaudPlamondon-nonkissing}}]
\label{thm:nonkissingAsso}
For any non-kissing finite gentle quiver~$\quiver$, the non-kissing fan~$\gvectorFan[\quiver]$ is the normal fan of the \defn{gentle associahedron}~$\Asso[\quiver]$ defined equivalently as:
\begin{enumerate}[(i)]
\item the convex hull of the points~$\sum_{\omega \in F} \KN(\omega) \, \cvector{F}{\omega}$ for all facets~${F \in \NKC}$,~or
\item the intersection of the halfspaces~$\bigset{\b{x} \in \R^{Q_0}}{\dotprod{\gvector{\omega}}{\b{x}} \le \KN(\omega)}$ for all walks~$\omega$ on~$\bar Q$.
\end{enumerate}
\end{theorem}

\para{Flips in the non-kissing fan}
Although we lack a characterization of the exchangeable pairs of the non-kissing complex (see \cref{rem:exchangeablePairsNKC}), we can still describe the linear dependence among the $\b{g}$-vectors involved in a flip.
The following statement is partially proved in~\cite[Thm.~4.17]{PaluPilaudPlamondon-nonkissing}.
The notations are illustrated in \cref{fig:flip}.

\begin{proposition}
\label{prop:exchangeablePairsNKC}
Let~$\omega, \omega'$ be two exchangeable walks on~$\quiver$. Then:
\begin{enumerate}[(i)]

\item For any non-kissing facets~$F, F'$ of~$\NKC$ with~$F \ssm \{\omega\} = F' \ssm \{\omega'\}$, the distinguished substrings of~$\omega$ in~$F$ and of~$\omega'$ in~$F'$ coincide to a string~$\sigma$. Moreover, if we decompose~${\omega = \rho \sigma \tau}$ and~${\omega' = \rho' \sigma \tau'}$, then the facets~$F$ and~$F'$ both contain the walks~$\mu \eqdef \rho' \sigma \tau$ and~$\nu \eqdef \rho \sigma \tau'$.
\item The substring~$\sigma$ and thus the walks~$\mu$ and~$\nu$ only depend on the exchangeable walks~$\omega$ and~$\omega'$, and not on the adjacent non-kissing facet~$F $ and~$F'$.
\item The linear dependence between the $\b{g}$-vectors of~$F \cup F'$ is given by
\[
\gvector{\omega} + \gvector{\omega'} = \gvector{\mu} + \gvector{\nu}.
\]
\item The non-kissing fan of~$\quiver$ has the unique exchange relation property.
\item The $\b{c}$-vector orthogonal to all $\b{g}$-vectors~$\gvector{\lambda}$ for~$\lambda \in F \cap F'$ is the multiplicity vector~$\multiplicityVector_{\sigma}$ of the vertices of the substring~$\sigma$ of~$\omega$ and~$\omega'$.
\end{enumerate}

\begin{figure}[b]
	\capstart
	\centerline{\includegraphics[scale=1]{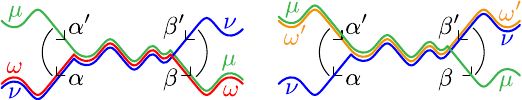}}
	\caption{A schematic representation of a flip. Illustration from~\cite{PaluPilaudPlamondon-nonkissing}.}
	\label{fig:flip}
\end{figure}
\end{proposition}

\begin{proof}
Points~(i), (iii) and~(v) were shown in~\cite[Prop.~2.33, Thm.~4.17 \& Prop~4.16]{PaluPilaudPlamondon-nonkissing}. We postpone the proof of point~(ii) until \cref{coro:exchangeablePairsNKC}. Finally, point~(iv) follows directly from points~(ii) and~(iii).
\end{proof}

\begin{remark}
\label{rem:exchangeablePairsNKC}
\enlargethispage{.7cm}
In view of \cref{prop:exchangeablePairsNKC}, it is tempting to look for a characterization of the exchangeable pairs~$\omega, \omega'$ using the kisses between~$\omega$ and~$\omega'$.
This question was discussed in~\cite[Sect.~9]{BrustleDouvilleMousavandThomasYildirim}.
However, as illustrated for instance in the non-kissing complex of \cref{fig:nonkissingFans}\,(right),
\begin{itemize}
\item two exchangeable walks may kiss along more than one string (only one is distinguished),
\item two non-exchangeable walks can kiss along more than one distinguishable string,
\item two walks that kiss along a single distinguishable string are not always exchangeable,
\item not all strings are distinguishable.
\end{itemize}
In \cref{subsubsec:simplicialTypeConeNKC}, we will restrict to a situation that avoids all these patologies.
\end{remark}

\para{Two families of examples: gentle grid and dissection quivers}
The initial motivation for non-kissing complexes came from two families of examples, illustrated in \cref{fig:dissectionGridQuivers}:

\medskip
\begin{enumerate}[(1)]
\item \textbf{Grid quivers}: Consider the infinite grid quiver~$\quiver_{Z^2}$, whose vertices are all integer points of~$\Z^2$, whose arrows are~$(i,j) \longrightarrow (i,j+1)$ and~$(i,j) \longrightarrow (i+1,j)$ for any~$(i,j) \in \Z^2$, and whose relations are~$(i-1,j) \longrightarrow (i,j) \longrightarrow (i,j+1)$ and~$(i,j-1) \longrightarrow (i,j) \longrightarrow (i+1,j)$ for any~$(i,j) \in \Z^2$. A \defn{gentle grid quiver} is any subquiver~$\quiver_A$ of~$\quiver_{Z^2}$ induced by a finite subset~$A \subset \Z^2$ of the integer grid. See \cref{fig:dissectionGridQuivers}. The non-kissing complex of gentle grid quivers were introduced in~\cite{McConville} with motivation coming from~\cite{PetersenPylyavskyySpeyer} and~\cite{SantosStumpWelker}. 

\medskip
\item \textbf{Dissection quivers}: Consider a dissection~$D$ of a convex polygon~$P$ (that is a crossing-free set of diagonals of~$P$) and its \defn{gentle dissection quiver}~$\quiver_D$, whose vertices are the internal diagonals of~$D$, whose arrows connect pairs of consecutive internal diagonals along the boundary of a face of~$D$, and whose relations correspond to triples of consecutive internal diagonals along the boundary of a face of~$D$. See \cref{fig:dissectionGridQuivers}. The non-kissing complex of~$\quiver_D$ then corresponds to non-crossing sets of accordions of~$D$, where an accordion is a segment connecting the middles of two boundary edges of~$P$ and crossing a connected set of diagonals of~$D$. This accordion complex was studied in~\cite{GarverMcConville} and~\cite{MannevillePilaud-accordion} with motivation coming from~\cite{Baryshnikov}, and~\cite{Chapoton-quadrangulations}.
\end{enumerate}

\begin{figure}[h]
	\capstart
	\centerline{\includegraphics[scale=.35]{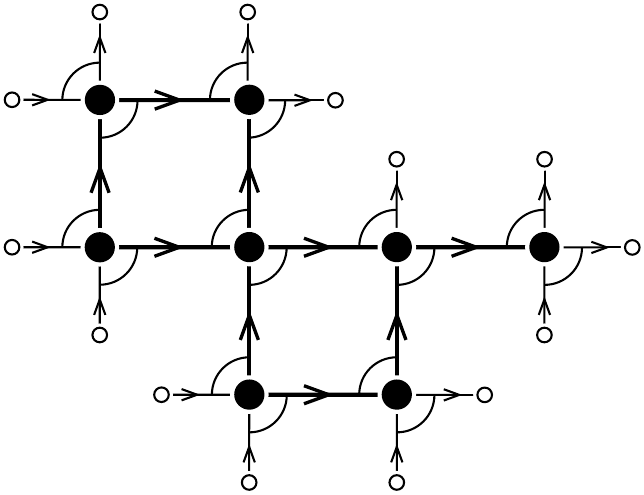} \quad \includegraphics[scale=.8]{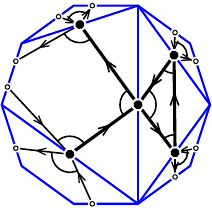} \quad \includegraphics[scale=1.2]{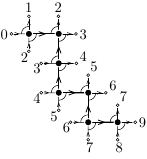} \quad \includegraphics[scale=1.2]{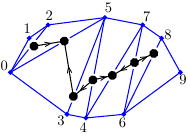}}
	\caption{A gentle grid quiver (left), a gentle dissection quiver (middle left), and a path quiver which is both the gentle grid quiver of a ribbon (middle right) and the dissection quiver of a triangulation (right). Illustration from~\cite{PaluPilaudPlamondon-nonkissing}.}
	\label{fig:dissectionGridQuivers}
\end{figure}

These two families of non-kissing complexes are well-behaved as they avoid all pathologies of \cref{rem:exchangeablePairsNKC}.
However, they still provide good examples of non-kissing complexes.
In particular, both families contain the classical associahedron.
Namely, the gentle associahedron~$\Asso[\quiver]$ is the classical associahedron of~\cite{ShniderSternberg, Loday} presented in \cref{thm:associahedronLoday} when~${\quiver = \quiver_A}$ for the path~${A = \set{(0,j)}{j \in [n]}}$ or equivalently~$\quiver = \quiver_D$ for the fan triangulation~$D$ (where all internal diagonals are incident to the same point).
More generally, $\Asso[\quiver]$ is an associahedron of~\cite{HohlwegLange} when~$\quiver = \quiver_A$ for a ribbon~$A$ or equivalently~$\quiver = \quiver_D$ for a triangulation~$D$.
See \cref{fig:dissectionGridQuivers}.

Note that, it was shown in~\cite{PaluPilaudPlamondon-surfaces} that the accordion complexes can be extended to dissections of arbitrary orientable surfaces with marked points and then provide a geometric model for all non-kissing complexes of gentle quivers.


\subsubsection{Type cones of non-kissing fans}

We now discuss the type cones of the non-kissing fans defined in \cref{thm:nonkissingFan}. We obtain from \cref{prop:exchangeablePairsNKC} the following redundant description.

\begin{corollary}
For any non-kissing finite gentle quiver~$\quiver$, the type cone of the non-kissing fan~$\gvectorFan[\quiver]$ is given by
\[
\typeCone \big( \gvectorFan[\quiver] \big) = \set{\b{h} \in \R^{\walks}}{\begin{array}{l} \b{h}_\omega = 0 \text{ for any improper walk } \omega \\ \b{h}_\omega + \b{h}_{\omega'} > \b{h}_\mu + \b{h}_\nu \text{ for any exchangeable walks } \omega, \omega' \end{array}},
\]
where the walks~$\mu$ and~$\nu$ for two exchangeable walks~$\omega, \omega'$ are defined in \cref{prop:exchangeablePairsNKC}.
\end{corollary}

\pagebreak

\begin{example}
\label{exm:typeConeNKC}
\enlargethispage{.4cm}
Consider the non-kissing fans illustrated in \cref{fig:nonkissingFans}.
The type cone of the left fan of \cref{fig:nonkissingFans} lives in~$\R^8$ and has a lineality space of dimension~$3$.
It has $5$ facet-defining inequalities (given below), which correspond to the flips described in \cref{prop:interestingExchangeablePairsNKC,prop:extremalExchangeablePairsNKC} and illustrated in \cref{fig:labelFacetDefiningInequalititiesNKC}\,(left).

\[
\begin{array}{r|ccccccccc}
\text{walks} & \raisebox{-.4cm}{\includegraphics[scale=.5]{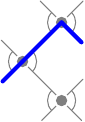}} & \raisebox{-.4cm}{\includegraphics[scale=.5]{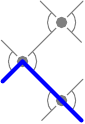}} & \raisebox{-.4cm}{\includegraphics[scale=.5]{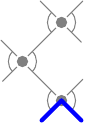}} & \raisebox{-.4cm}{\includegraphics[scale=.5]{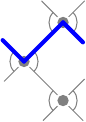}} & \raisebox{-.4cm}{\includegraphics[scale=.5]{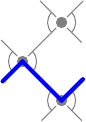}} & \raisebox{-.4cm}{\includegraphics[scale=.5]{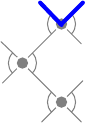}} & \raisebox{-.4cm}{\includegraphics[scale=.5]{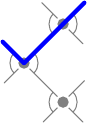}} & \raisebox{-.4cm}{\includegraphics[scale=.5]{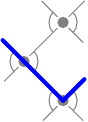}} \\[.6cm]
\text{$\b{g}$-vectors} & \compactVectorT{1}{0}{0} & \compactVectorT{0}{1}{0} & \compactVectorT{0}{0}{1} & \compactVectorT{1}{-1}{0} & \compactVectorT{0}{1}{-1} & \compactVectorT{-1}{0}{0} & \compactVectorT{0}{-1}{0} & \compactVectorT{0}{0}{-1} \\[.6cm]
\text{facet}		& 0 & -1 & 1 & 0 & 1 & 0 & 0 & 0 & \red \circled{A} \\
\text{defining}		& 0 & 1 & 0 & 0 & -1 & 0 & 0 & 1 & \red \circled{B} \\
\text{inequalities}	& -1 & 0 & 0 & 1 & 1 & 0 & 0 & -1 & \red \circled{C} \\
					& 1 & 0 & 0 & -1 & 0 & 0 & 1 & 0 & \red \circled{D} \\
					& 0 & 0 & 0 & 1 & 0 & 1 & -1 & 0 & \red \circled{E} \\[.2cm]
\end{array}
\]




\medskip
\noindent
The type cone of the right fan of \cref{fig:nonkissingFans} lives in~$\R^{11}$ and has a lineality space of dimension~$3$.
It has $9$ facet-defining inequalities (given below), which correspond to the flips illustrated in \cref{fig:labelFacetDefiningInequalititiesNKC}\,(right).
In particular, it is not simplicial.

\[
\begin{array}{r|c@{\;}c@{\;}c@{\;}c@{\;}c@{\;}c@{\;}c@{\;}c@{\;}c@{\;}c@{\;}cc}
\text{walks} & \raisebox{-.4cm}{\includegraphics[scale=.5]{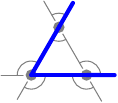}} & \raisebox{-.4cm}{\includegraphics[scale=.5]{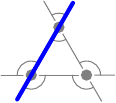}} & \raisebox{-.4cm}{\includegraphics[scale=.5]{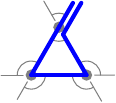}} & \raisebox{-.4cm}{\includegraphics[scale=.5]{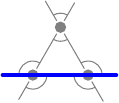}} & \raisebox{-.4cm}{\includegraphics[scale=.5]{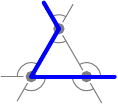}} & \raisebox{-.4cm}{\includegraphics[scale=.5]{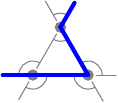}} & \raisebox{-.4cm}{\includegraphics[scale=.5]{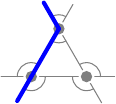}} & \raisebox{-.4cm}{\includegraphics[scale=.5]{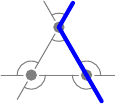}} & \raisebox{-.4cm}{\includegraphics[scale=.5]{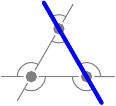}} & \raisebox{-.4cm}{\includegraphics[scale=.5]{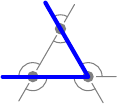}} & \raisebox{-.4cm}{\includegraphics[scale=.5]{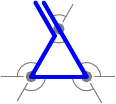}} \\[.6cm]
\text{$\b{g}$-vectors} & \compactVectorT{1}{0}{0} & \compactVectorT{0}{1}{0} & \compactVectorT{0}{0}{1} & \compactVectorT{1}{-1}{0} & \compactVectorT{1}{0}{-1} & \compactVectorT{0}{-1}{1} & \compactVectorT{0}{1}{-1} & \compactVectorT{-1}{0}{1} & \compactVectorT{-1}{0}{0} & \compactVectorT{0}{-1}{0} & \compactVectorT{0}{0}{-1} \\[.6cm]
\text{facet}		& -1 & 1 & 0 & 0 & 1 & 0 & -1 & 0 & 0 & 0 & 0 & \red \circled{A} \\
\text{defining}		& 1 & 0 & 0 & 0 & -1 & 0 & 0 & 0 & 0 & 0 & 1 & \red \circled{B} \\
\text{inequalities}	& 0 & 0 & 0 & 1 & -1 & 0 & 1 & 0 & 0 & 0 & 0 & \red \circled{C} \\
					& 1 & 0 & -1 & -1 & 0 & 1 & 0 & 0 & 0 & 0 & 0 & \red \circled{D} \\
					& 0 & 0 & 0 & -1 & 1 & 0 & 0 & 0 & 0 & 1 & -1 & \red \circled{E} \\
					& 0 & 0 & 1 & 0 & 0 & -1 & 0 & 0 & 0 & 1 & 0 & \red \circled{F}\\
					& 0 & 0 & 0 & 1 & 0 & -1 & 0 & 1 & 0 & 0 & 0 & \red \circled{G} \\
					& 0 & 0 & 0 & 0 & 0 & 1 & 0 & -1 & 1 & -1 & 0 & \red \circled{H} \\
					& 0 & -1 & 0 & 0 & 0 & 0 & 1 & 1 & -1 & 0 & 0 & \red \circled{K} \\[.2cm]
\end{array}
\]

\begin{figure}[h]
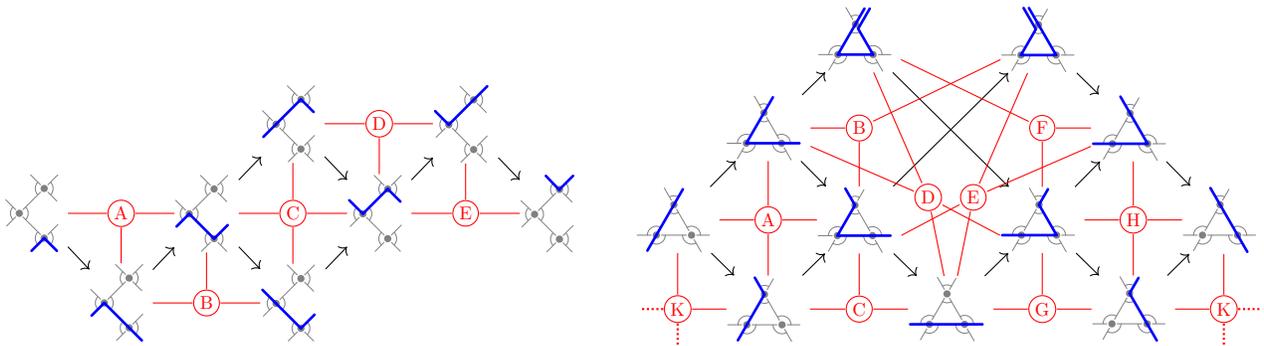

	\capstart
    \begin{adjustbox}{center}
        \begin{tikzpicture}
        	\matrix (m) [matrix of math nodes, row sep=.15cm, column sep=.3cm, nodes={anchor=center, align=center, inner sep=0pt}]{
        		&&& \includegraphics[scale=.5]{walkA1} & \node (d) {\red \circled{D}}; & \includegraphics[scale=.5]{walkA7} & \\
        		\includegraphics[scale=.5]{walkA3} & \node (a) {\red \circled{A}}; & \includegraphics[scale=.5]{walkA5} & \node (c) {\red \circled{C}}; & \includegraphics[scale=.5]{walkA4} & \node (e) {\red \circled{E}}; & \includegraphics[scale=.5]{walkA6} \\
        		& \includegraphics[scale=.5]{walkA2} & \node (b) {\red \circled{B}}; & \includegraphics[scale=.5]{walkA8} &&& \\};
        	\draw[->] (m-2-1) -- (m-3-2);
        	\draw[->] (m-3-2) -- (m-2-3);
        	\draw[->] (m-2-3) -- (m-1-4);
        	\draw[->] (m-2-3) -- (m-3-4);
        	\draw[->] (m-1-4) -- (m-2-5);
        	\draw[->] (m-3-4) -- (m-2-5);
        	\draw[->] (m-2-5) -- (m-1-6);
        	\draw[->] (m-1-6) -- (m-2-7);
        	\draw[red] (a) -- (m-2-1);
        	\draw[red] (a) -- (m-2-3);
        	\draw[red] (a) -- (m-3-2);
        	\draw[red] (b) -- (m-3-2);
        	\draw[red] (b) -- (m-3-4);
        	\draw[red] (b) -- (m-2-3);
        	\draw[red] (c) -- (m-2-3);
        	\draw[red] (c) -- (m-2-5);
        	\draw[red] (c) -- (m-1-4);
        	\draw[red] (c) -- (m-3-4);
        	\draw[red] (d) -- (m-1-4);
        	\draw[red] (d) -- (m-1-6);
        	\draw[red] (d) -- (m-2-5);
        	\draw[red] (e) -- (m-2-5);
        	\draw[red] (e) -- (m-2-7);
        	\draw[red] (e) -- (m-1-6);
        \end{tikzpicture}
        \quad
        \begin{tikzpicture}
        	\matrix (m) [matrix of math nodes, row sep=.3cm, column sep=.1cm, nodes={anchor=center, align=center, inner sep=0pt}]{
        		&& \includegraphics[scale=.5]{walkB3} && \includegraphics[scale=.5]{walkB11} && \\
        		& \includegraphics[scale=.5]{walkB1} & \node (b) {\red \circled{B}}; && \node (f) {\red \circled{F}}; & \includegraphics[scale=.5]{walkB10} & \\
        		\includegraphics[scale=.5]{walkB2} & \node (a) {\red \circled{A}}; & \includegraphics[scale=.5]{walkB5} & \node (d) at (-.3,.3) {\red \circled{D}}; \node (e) at (.3,.3) {\red \circled{E}}; & \includegraphics[scale=.5]{walkB6} & \node (h) {\red \circled{H}}; & \includegraphics[scale=.5]{walkB9} \\
        		\node (k1) {\red \circled{K}}; & \includegraphics[scale=.5]{walkB7} & \node (c) {\red \circled{C}}; & \includegraphics[scale=.5]{walkB4} & \node (g) {\red \circled{G}}; & \includegraphics[scale=.5]{walkB8} & \node (k2) {\red \circled{K}}; \\};
        	\draw[->] (m-3-1) -- (m-2-2);
        	\draw[->] (m-3-1) -- (m-4-2);
        	\draw[->] (m-2-2) -- (m-1-3);
        	\draw[->] (m-2-2) -- (m-3-3);
        	\draw[->] (m-4-2) -- (m-3-3);
        	\draw[->] (m-1-3) -- (m-3-5);
        	\draw[->] (m-3-3) -- (m-1-5);
        	\draw[->] (m-3-3) -- (m-4-4);
        	\draw[->] (m-4-4) -- (m-3-5);
        	\draw[->] (m-1-5) -- (m-2-6);
        	\draw[->] (m-3-5) -- (m-2-6);
        	\draw[->] (m-3-5) -- (m-4-6);
        	\draw[->] (m-2-6) -- (m-3-7);
        	\draw[->] (m-4-6) -- (m-3-7);
        	\draw[red] (a) -- (m-3-1);
        	\draw[red] (a) -- (m-3-3);
        	\draw[red] (a) -- (m-2-2);
        	\draw[red] (a) -- (m-4-2);
        	\draw[red] (b) -- (m-2-2);
        	\draw[red] (b) -- (m-1-5);
        	\draw[red] (b) -- (m-3-3);
        	\draw[red] (c) -- (m-4-2);
        	\draw[red] (c) -- (m-4-4);
        	\draw[red] (c) -- (m-3-3);
        	\draw[red] (d) -- (m-2-2);
        	\draw[red] (d) -- (m-3-5.200);
        	\draw[red] (d) -- (m-1-3);
        	\draw[red] (d) -- (m-4-4);
        	\draw[red] (e) -- (m-3-3.340);
        	\draw[red] (e) -- (m-2-6);
        	\draw[red] (e) -- (m-4-4);
        	\draw[red] (e) -- (m-1-5);
        	\draw[red] (f) -- (m-1-3);
        	\draw[red] (f) -- (m-2-6);
        	\draw[red] (f) -- (m-3-5);
        	\draw[red] (g) -- (m-4-4);
        	\draw[red] (g) -- (m-4-6);
        	\draw[red] (g) -- (m-3-5);
        	\draw[red] (h) -- (m-3-5);
        	\draw[red] (h) -- (m-3-7);
        	\draw[red] (h) -- (m-2-6);
        	\draw[red] (h) -- (m-4-6);
        	\draw[red] (k1) -- (m-3-1);
        	\draw[red] (k1) -- (m-4-2);
        	\draw[red, densely dotted, thick] (k1.west) -- ([xshift=-.3cm]k1.west);
        	\draw[red, densely dotted, thick] (k1.south) -- ([yshift=-.3cm]k1.south);
        	\draw[red] (k2) -- (m-4-6);
        	\draw[red] (k2) -- (m-3-7);
        	\draw[red, densely dotted, thick] (k2.east) -- ([xshift=.3cm]k2.east);
        	\draw[red, densely dotted, thick] (k2.south) -- ([yshift=-.3cm]k2.south);
        \end{tikzpicture}
    \end{adjustbox}
	\caption{The facet-defining inequalities of the type cone~$\typeCone \big( \gvectorFan[\quiver] \big)$ of the non-kissing fan, represented on the Auslander Reiten quiver of the gentle algebra of~$\quiver$. While they correspond to meshes on the left, they are not as clear in the general case as on the right. See also \cref{sec:extricats}.}
	\label{fig:labelFacetDefiningInequalititiesNKC}
\end{figure}
\end{example}

As illustrated in \cref{exm:typeConeNKC}, the type cone of the non-kissing fan is not always simplicial and we do not always understand its extremal exchangeable pairs.
In the next section section, we will explore a special family of gentle quivers for which we can completely describe the type cone.
The combination of this special family with computer experiments in the general case supports the following conjecture.

\begin{conjecture}
Consider a distinguishable string~$\sigma \in \distinguishableStrings$, and let~$\omega$ (resp.~$\omega'$) be the walk obtained from~$\sigma$ by adding two hooks (resp.~two cohooks) at the endpoints of~$\sigma$. Then
\begin{enumerate}
\item The $\b{c}$-vector~$\multiplicityVector_\sigma$ is the direction of at least one extremal exchangeable pair.
\item If the walks~$\omega$ and~$\omega'$ are non-self-kissing and exchangeable, then they form the unique extremal exchangeable pair directed by~$\sigma$. These extremal exchangeable pairs correspond to meshes of the Auslander-Reiten quiver.
\end{enumerate}
\end{conjecture}


\subsubsection{Simplicial type cones for brick and $2$-acyclic gentle bound quivers}
\label{subsubsec:simplicialTypeConeNKC}

The following family of gentle quivers was considered in~\cite[Sect.~4]{GarverMcConvilleMousavand}.

\begin{proposition}
\label{prop:noSelfKissing}
The following conditions are equivalent for a gentle quiver~$\quiver$.
\begin{enumerate}[(i)]
\item any (non necessarily oriented) cycle of~$\quiver$ contains at least two relations in~$I$,
\item any string of~$\quiver$ is distinguishable,
\item no walk on~$\quiver$ is self-kissing.
\end{enumerate}
\end{proposition}

In this section, we will restrict our attention to the following family of quivers.

\begin{definition}
\label{def:brick2acyclic}
A gentle quiver~$\quiver$ is called:
\begin{itemize}
\item \defn{brick} if it satisfies the three equivalent conditions of \cref{prop:noSelfKissing},
\item \defn{$2$-acyclic} if it contains no cycle of length~$2$.
\end{itemize}
\end{definition}

Note that the family of brick and $2$-acyclic gentle quivers already contains a lot of relevant examples, including the gentle grid and dissection quivers discussed in \cref{subsubsec:nonkissingComplex}.
In particular, the classical associahedron is the gentle associahedron of a brick and $2$-acyclic gentle quiver.
We will see in \cref{coro:simplicialTypeConeNKC} that the type cone of the non-kissing fan~$\gvectorFan[\quiver]$ of a brick and $2$-acyclic gentle quiver~$\quiver$ happens to be simplicial, and we will derive in \cref{thm:allPolytopalRealizationsNKC} a simple description of all polytopal realizations of~$\gvectorFan[\quiver]$.

For a string~$\sigma$ of~$\quiver$, we denote by $\sigma\hR$ (resp.~$\sigma\cR$) the unique string of the blossoming quiver~$\quiver\blossom$ of the form $\sigma\hR = \sigma \alpha_1^{-1} \alpha_2 \dots \alpha_\ell$ (resp.~$\sigma\cR = \sigma \alpha_1 \alpha_2^{-1} \dots \alpha_\ell^{-1}$) with~$\ell \ge 1$ and~${\alpha_1, \dots, \alpha_\ell \in Q_1}$ and such that~$t(\alpha_\ell)$ (resp.~$s(\alpha_\ell)$) is a blossom of~$\quiver\blossom$.
These notations are motivated by the representation of strings used in~\cite{ButlerRingel, PaluPilaudPlamondon-nonkissing}, and the terminology usually says that~$\sigma\hR$ (resp.~$\sigma\cR$) is obtained by adding a \defn{hook} (resp.~a \defn{cohook}) to~$\sigma$.
We define similarly~$\hL\sigma$ (resp.~$\cL\sigma$).
The walk~$\hL(\sigma\hR) = (\hL\sigma)\hR$ of~$\quiver$ is simply denoted by~$\hh{\sigma}$, and we define similarly~$\cc{\sigma}$, $\hc{\sigma}$ and~$\ch{\sigma}$.

\begin{proposition}
\label{prop:interestingExchangeablePairsNKC}
For any brick and $2$-acyclic gentle quiver~$\quiver$ and any string~$\sigma \in \strings$, the walks~$\cc{\sigma}$ and~$\hh{\sigma}$ are exchangeable with distinguished substring~$\sigma$.
\end{proposition}

\begin{proof}
We fix some notations similar to that of \cref{fig:flip}.
Let~$\omega \eqdef \cc{\sigma}$, $\omega' \eqdef \hh{\sigma}$, $\mu \eqdef \hc{\sigma}$ and~$\nu \eqdef \ch{\sigma}$.
See \cref{fig:flip}.

Since~$\quiver$ is brick, these walks are not self-kissing by \cref{prop:noSelfKissing} and since they are bending, they are proper walks.
We want to show that there are non-kissing facets~$F, F'$ of the non-kissing complex~$\NKC$ containing both~$\mu$ and~$\nu$ and such that~$F \ssm \{\omega\} = F' \ssm \{\omega'\}$.

We first show that~$\mu$ and~$\nu$ are compatible.
Let~$\tau$ be a maximal common substring of~$\mu$ and~$\nu$.
If~$\sigma$ and~$\tau$ are not disjoint, then~$\sigma = \tau$ and $\mu$ and~$\nu$ are not kissing at~$\tau$.
We can therefore assume that $\tau$ appears completely before or completely after~$\sigma$ in both~$\mu$ and~$\nu$.
We distinguish two different cases:
\begin{itemize}
\item If~$\tau$ appears before~$\sigma$ in both~$\mu$ and~$\nu$, then~$\alpha \tau \alpha'$ forms a cycle in~$\quiver$. If~$\tau$ is reduced to a vertex, then we have a $2$-cycle. Otherwise, since~$\alpha\tau$ is a substring of~$\nu$ and~$\alpha'\tau$ is a substring of~$\nu$, the cycle~$\alpha \tau \alpha'$ contains a unique relation~$\alpha'\alpha$. This rules out this case under the assumption that~$\quiver$ is brick and $2$-acyclic. The case when~$\tau$ appears after~$\sigma$ in both~$\mu$ and~$\nu$ is symmetric.
\item If~$\tau$ appears before~$\sigma$ in~$\mu$ and after~$\sigma$ in~$\nu$, then~$\omega$, $\mu$ and~$\nu$ finish at the same blossom so that~$\mu$ and~$\nu$ are not kissing at~$\tau$. The case when~$\tau$ appears after~$\sigma$ in~$\mu$ and before~$\sigma$ in~$\nu$ is symmetric.
\end{itemize}
We conclude that~$\mu$ and~$\nu$ are non-kissing, which also implies that~$\{\mu, \nu, \omega\}$ and~$\{\mu, \nu, \omega'\}$ are non-kissing faces of~$\NKC$.
We can therefore consider a non-kissing facet~$F$ containing~$\{\mu, \nu, \omega\}$.

We claim that~$\omega$ is distinguished at~$\alpha$ and~$\beta$ in~$F$, that~$\mu$ is distinguished at~$\alpha'$ in~$F$ and that~$\nu$ is distinguished at~$\beta'$ in~$F$.
Let us just prove that~$\mu$ is distinguished at~$\alpha'$, the other statements being similar.
Consider any walk~$\lambda$ of~$F$ containing~$\alpha'$, and let~$\tau$ be the maximal common substring of~$\omega$ and~$\lambda$.
Note that~$\omega$ has the outgoing arrow~$\alpha$ while~$\lambda$ has the incoming arrow~$\alpha'$ at one end of~$\tau$.
Since~$\omega$ and~$\lambda$ are compatible, it ensures that either~$\tau$ ends at a blossom, or $\lambda$ has an outgoing arrow at the other end of~$\tau$.
This shows that~$\lambda \prec_{\alpha'} \mu$ since $\lambda$ separates from~$\mu$ with an arrow in the same direction as~$\alpha'$.

Since this precisely coincides with the description of the flip in the non-kissing complex (see \cref{prop:exchangeablePairsNKC}\,(i) or~\cite[Prop.~2.33]{PaluPilaudPlamondon-nonkissing} for an alternative presentation), we conclude that the flip of~$\omega$ in~$F$ creates a facet~$F'$ containing~$\omega'$.
Therefore, $\omega$ and~$\omega'$ are exchangeable with distinguished substring~$\sigma$.
\end{proof}

The following statement describes the type cone of the non-kissing fan of a brick and $2$-acyclic gentle quiver.
We provide here an elementary combinatorial proof, a more general representation theoretic perspective is discussed in \cref{thm:brick-algebra-condition}.

\begin{proposition}
\label{prop:extremalExchangeablePairsNKC}
For any brick and $2$-acyclic gentle quiver~$\quiver$, the extremal exchangeable pairs for the non-kissing fan of~$\quiver$ are precisely the pairs~$\{\cc{\sigma}, \hh{\sigma}\}$ for all strings~$\sigma \in \strings$.
\end{proposition}

\begin{proof}
Let~$(\b{f}_\omega)_{\omega \in \walks}$ be the canonical basis of~$\R^{\walks}$.
Consider two exchangeable walks~$\omega$ and~$\omega'$ with distinguished substring~$\sigma \in \Sigma_\top(\omega) \cap \Sigma_\bottom(\omega')$.
Decompose~$\omega = \rho \sigma \tau$ and~$\omega' = \rho' \sigma \tau'$ and define~$\mu \eqdef \rho' \sigma \tau$ and~$\nu \eqdef \rho \sigma \tau'$ as in \cref{prop:exchangeablePairsNKC}\,(i). \cref{prop:exchangeablePairsNKC}\,(iii) ensures that the linear dependence between the corresponding $\b{g}$-vectors is given by
\[
\gvector{\omega} + \gvector{\omega'} = \gvector{\mu} + \gvector{\nu}.
\]
Therefore, the inner normal vector of the corresponding inequality of the type cone~$\typeCone \big( \gvectorFan[\quiver] \big)$ is
\[
\b{n}(\omega, \omega') \eqdef \b{f}_\omega + \b{f}_{\omega'} - \b{f}_\mu - \b{f}_\nu.
\]
We claim that this normal vector is always a positive linear combination of the normal vectors~$\b{m}(\sigma) \eqdef \b{n}(\cc{\sigma}, \hh{\sigma}) = \b{f}_{\cc{\sigma}} + \b{f}_{\hh{\sigma}} - \b{f}_{\hc{\sigma}} - \b{f}_{\ch{\sigma}}$ for all strings~$\sigma \in \strings$.
Our proof works by descending induction on the length~$\lambda(\omega,\omega') \eqdef \ell(\sigma)$ of the common substring of~$\omega$ and~$\omega'$.
If~$\lambda(\omega, \omega')$ is big enough, then the walk~$\omega$ (resp.~$\omega'$) is just obtained by adding two outgoing (resp.~incoming) blossoms at the end of~$\sigma$, thus $\omega = \cc{\sigma}$ (resp.~$\omega' = \hh{\sigma}$), and there is nothing to prove.
Assume now that~$\omega \ne \hh{\sigma}$ (the situation where~$\omega' \ne \cc{\sigma}$ is symmetric).
If~$\rho \ne \cL$, observe that
\begin{itemize}
\item $\omega$ and~$\cL \sigma \tau'$ are exchangeable with~$\b{n}(\omega, \cL \sigma \tau') = \b{f}_\omega + \b{f}_{\cL \sigma \tau'} - \b{f}_{\cL \sigma \tau} - \b{f}_\nu$, and
\item $\cL \sigma \tau$ and~$\omega'$ are exchangeable with~$\b{n}(\cL \sigma \tau, \omega') = \b{f}_{\cL \sigma \tau} + \b{f}_{\omega'} - \b{f}_{\cL \sigma \tau'} - \b{f}_\mu$.
\end{itemize}
We derive that
\[
\b{n}(\omega, \omega') = \b{n}(\omega, \cL \sigma \tau') + \b{n}(\cL \sigma \tau, \omega').
\]
Observe moreover that since~$\omega$ has outgoing arrows at the endpoints of~$\sigma$, the common substring of~$\omega$ and~$\cL \sigma \tau'$ strictly contains~$\sigma$ so that $\lambda(\omega, \cL \sigma \tau') > \lambda(\omega, \omega')$.
By induction, $\b{n}(\omega, \cL \sigma \tau')$ is thus a positive linear combination of~$\b{m}(\sigma)$ for~$\sigma \in \strings$.
By symmetry, we obtain the four equalities
\[
\b{n}(\omega, \omega') = 
\begin{cases}
\b{n}(\omega, \cL \sigma \tau') + \b{n}(\cL \sigma \tau, \omega') & \text{if } \rho \ne \cL, \\
\b{n}(\omega, \rho' \sigma \cR) + \b{n}(\rho \sigma \cR, \omega') & \text{if } \tau \ne \cR, \\
\b{n}(\omega, \hL \sigma \tau') + \b{n}(\hL \sigma \tau, \omega') & \text{if } \rho' \ne \hL, \\
\b{n}(\omega, \rho' \sigma \hR) + \b{n}(\rho \sigma \hR, \omega') & \text{if } \tau' \ne \hR.
\end{cases}
\]
These four equalities are illustrated on \cref{fig:fourEqualities}.
\begin{figure}[t]
	\capstart
	\begin{adjustbox}{center}
    	\begin{tabular}{c@{\qquad}c}
    		\includegraphics[scale=1.2]{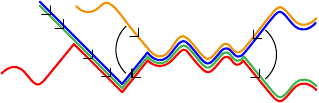} & \includegraphics[scale=1.2]{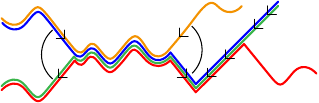} \\
    		$\b{n}({\red \omega}, {\orange \omega'}) = \b{n}({\red \omega}, {\blue \cL \sigma \tau'}) + \b{n}({\green \cL \sigma \tau}, {\orange \omega'})$ & $\b{n}({\red \omega}, {\orange \omega'}) = \b{n}({\red \omega}, {\blue \rho' \sigma \cR}) + \b{n}({\green \rho \sigma \cR}, {\orange \omega'})$ \\[.4cm]
    		\includegraphics[scale=1.2]{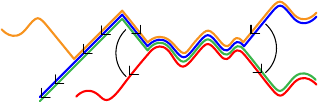} & \includegraphics[scale=1.2]{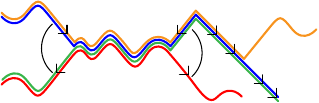} \\
    		$\b{n}({\red \omega}, {\orange \omega'}) = \b{n}({\red \omega}, {\blue \hL \sigma \tau'}) + \b{n}({\green \hL \sigma \tau}, {\orange \omega'})$ & $\b{n}({\red \omega}, {\orange \omega'}) = \b{n}({\red \omega}, {\blue \rho' \sigma \hR}) + \b{n}({\green \rho \sigma \hR}, {\orange \omega'})$ \\
    	\end{tabular}
	\end{adjustbox}
	\caption{Schematic representation of the four equalities in the proof of \cref{prop:extremalExchangeablePairsNKC}.}
	\label{fig:fourEqualities}
\end{figure}
Moreover~$\b{n}(\omega, \cL \sigma \tau')$, $\b{n}(\cL \sigma \tau, \hL \sigma \cR)$, $\b{n}(\cL \sigma \hR, \hL \sigma \tau')$ and~$\b{n}(\hL \sigma \tau, \omega')$ are all positive combinations of~$\b{m}(\sigma)$ for~$\sigma \in \strings$ by induction hypothesis.
Applying these equalities one after the other, we obtain
\[
\renewcommand{\arraystretch}{1.3}
\begin{array}[t]{l@{\;\;}l@{\;}l@{\;}l@{\;}l@{\;}l@{\;}l@{\;}l@{\;}l@{\;}l}
  & \b{n}(\omega, \omega') & & \hfill\raisebox{-.1cm}{\text{\tiny{Equality 1}}} & & & & & & \\
\cline{2-4}
= & \b{n}(\omega, \cL \sigma \tau') & + & \b{n}(\cL \sigma \tau, \omega') & & & & & & \hfill\raisebox{-.1cm}{\text{\tiny{Equality 2}}} \\
\cline{4-10}
= & \b{n}(\omega, \cL \sigma \tau') & + & \b{n}(\cL \sigma \tau, \hL \sigma \tau') & & \hfill\raisebox{-.1cm}{\text{\tiny{Equality 3}}} & & & + & \b{n}(\hL \sigma \tau, \omega') \\
\cline{4-6}
= & \b{n}(\omega, \cL \sigma \tau') & + & \b{n}(\cL \sigma \tau, \hL \sigma \cR) & + & \b{n}(\cL \sigma \cR, \hL \sigma \tau') & & \hfill\raisebox{-.1cm}{\text{\tiny{Equality 4}}} & + & \b{n}(\hL \sigma \tau, \omega') \\
\cline{6-8}
= & \b{n}(\omega, \cL \sigma \tau') & + & \b{n}(\cL \sigma \tau, \hL \sigma \cR) & + & \b{n}(\cL \sigma \cR, \hL \sigma \hR) & + & \b{n}(\cL \sigma \hR, \hL \sigma \tau') & + & \b{n}(\hL \sigma \tau, \omega') \\
= & \b{n}(\omega, \cL \sigma \tau') & + & \b{n}(\cL \sigma \tau, \hL \sigma \cR) & + & \multicolumn{1}{c}{\b{m}(\sigma)} & + & \b{n}(\cL \sigma \hR, \hL \sigma \tau') & + & \b{n}(\hL \sigma \tau, \omega')
\end{array}
\renewcommand{\arraystretch}{1}
\]
where we fix the convention~$\b{n}(\lambda, \lambda) = 0$ in case~$\rho = \hL$, $\rho' = \cL$, $\tau = \hR$ or~$\tau' = \cR$.
We conclude that~$\b{n}(\omega, \omega')$ is a positive combination of~$\b{m}(\sigma)$ for~$\sigma \in \strings$, since~$\b{n}(\omega, \cL \sigma \tau')$, $\b{n}(\cL \sigma \tau, \hL \sigma \cR)$, $\b{n}(\cL \sigma \hR, \hL \sigma \tau')$ and~$\b{n}(\hL \sigma \tau, \omega')$ are.
This shows that all extremal exchangeable pairs are of the form~$\{\cc{\sigma}, \hh{\sigma}\}$ for~$\sigma \in \strings$.

Conversely, we know from \cref{rem:dimTypeCone} that there are at least~$|\properWalks|-|Q_0|$ extremal exchangeable pairs. Since~$|\strings| = |\distinguishableStrings| = |\properWalks| - |Q_0|$ by \cref{prop:bijectionStringsWalks}, we conclude that all exchangeable pairs~$\{\cc{\sigma}, \hh{\sigma}\}$ for~$\sigma \in \strings$ are extremal.
\end{proof}

Our next statement follows from the end of the previous proof.
See also \cref{coro:simplicialTypeConeNKCalgebraic}.

\begin{corollary}
\label{coro:simplicialTypeConeNKC}
For any brick and $2$-acyclic gentle quiver~$\quiver$, the type cone~$\typeCone \big( \gvectorFan[\quiver] \big)$ of the non-kissing fan~$\gvectorFan[\quiver]$ is simplicial.
\end{corollary}

Combining \cref{coro:simplicialTypeCone,coro:simplicialTypeConeNKC,prop:extremalExchangeablePairsNKC}, we derive the following description of all polytopal realizations of the non-kissing fan~$\gvectorFan[\quiver]$ of a brick and $2$-acyclic quiver~$\quiver$.

\begin{theorem}
\label{thm:allPolytopalRealizationsNKC}
For any brick and $2$-acyclic gentle quiver~$\quiver$ and any~$\b{\ell} \in \R_{>0}^{\strings}$, the polytope
\[
R_\b{\ell}(\quiver) \eqdef \set{\b{z} \in \R^{\walks}}{\begin{array}{l} \b{z} \ge 0 \qquad\text{and}\qquad \b{z}_\omega = 0 \text{ for any improper walk } \omega \\ \b{z}_{\cc{\sigma}} + \b{z}_{\hh{\sigma}} - \b{z}_{\hc{\sigma}} - \b{z}_{\ch{\sigma}} = \b{\ell}_\sigma \text{ for all } \sigma \in \strings\end{array}}
\]
is a realization of the non-kissing fan~$\gvectorFan[\quiver]$.
Moreover, the polytopes~$R_\b{\ell}(\quiver)$ for~$\b{\ell} \in \R_{>0}^{\strings}$ describe all polytopal realizations of~$\gvectorFan[\quiver]$.
\end{theorem}

We also obtain from \cref{prop:extremalExchangeablePairsNKC} the following surprising property.

\begin{corollary}
Any $\b{c}$-vector supports exactly one extremal exchangeable pair.
\end{corollary}

\begin{remark}
\label{rem:meshARquiver}
Although not needed in the proof of \cref{prop:extremalExchangeablePairsNKC}, we note that the extremal exchangeable pairs~$\{\cc{\sigma}, \hh{\sigma}\}$ and their linear dependencies~${\gvector{\cc{\sigma}} + \gvector{\hh{\sigma}} - \gvector{\hc{\sigma}} - \gvector{\ch{\sigma}}}$ precisely correspond to the meshes of the Auslander-Reiten quiver of~$\quiver$.
\end{remark}


\begin{remark}
\label{exm:nonSimplicialTypeConeNKC}
Note that \cref{prop:interestingExchangeablePairsNKC,prop:extremalExchangeablePairsNKC} and therefore \cref{coro:simplicialTypeConeNKC,thm:allPolytopalRealizationsNKC} fail when the quiver is not brick or $2$-cyclic.
The smallest exemples are the $2$-cycles whose non-kissing fans are represented in \cref{fig:nonkissingFans2}.
The left one is not brick: it has a non-distinguishable string and a self-kissing walk.
The right one is brick but is $2$-cyclic.

\begin{figure}[h]
	\capstart
	\centerline{\includegraphics[scale=.45]{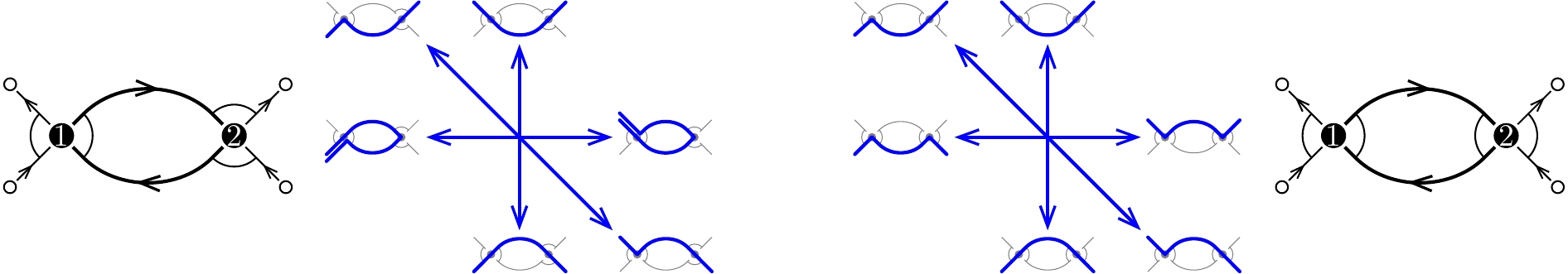}}
	\caption{The non-kissing fans of two $2$-cyclic gentle quivers. Illustration from~\cite{PaluPilaudPlamondon-nonkissing}.}
	\label{fig:nonkissingFans2}
\end{figure}

In both cases, \cref{prop:interestingExchangeablePairsNKC} fails for the string~$\sigma$ of length~$0$ at vertex~$1$ (resp.~$2$) as the resulting walks~$\cc{\sigma}$ and~$\hh{\sigma}$ have $\b{g}$-vectors~$\gvector{\cc{\sigma}} = (1,-1)$ and~$\gvector{\hh{\sigma}} = (-1,1)$ (resp.~$\gvector{\cc{\sigma}} = (-1,1)$ and~$\gvector{\hh{\sigma}} = (1,-1)$), and are therefore not exchangeable.
\cref{prop:extremalExchangeablePairsNKC} also fails as the type cone has dimension~$4$ and the following~$5$ inequalities:

\[
\begin{array}{r|cccccc}
\text{walks} & \raisebox{-.25cm}{\includegraphics[scale=.5]{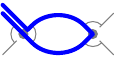}} & \raisebox{-.25cm}{\includegraphics[scale=.5]{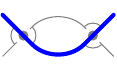}} & \raisebox{-.25cm}{\includegraphics[scale=.5]{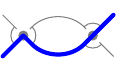}} & \raisebox{-.25cm}{\includegraphics[scale=.5]{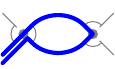}} & \raisebox{-.25cm}{\includegraphics[scale=.5]{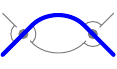}} & \raisebox{-.25cm}{\includegraphics[scale=.5]{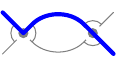}} \\[.2cm]
& \raisebox{-.25cm}{\includegraphics[scale=.5]{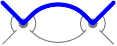}} & \raisebox{-.25cm}{\includegraphics[scale=.5]{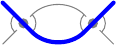}} & \raisebox{-.25cm}{\includegraphics[scale=.5]{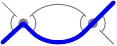}} & \raisebox{-.25cm}{\includegraphics[scale=.5]{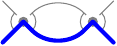}} & \raisebox{-.25cm}{\includegraphics[scale=.5]{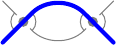}} & \raisebox{-.25cm}{\includegraphics[scale=.5]{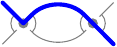}} \\[.6cm]
\text{$\b{g}$-vectors} & \compactVectorD{1}{0} & \compactVectorD{0}{1} & \compactVectorD{-1}{1} & \compactVectorD{-1}{0} & \compactVectorD{0}{-1} & \compactVectorD{1}{-1} \\[.6cm]
\text{facet}		& 1 & -1 & 1 & 0 & 0 & 0 \\
\text{defining}		& 0 & 1 & -1 & 1 & 0 & 0 \\
\text{inequalities}	& 0 & 0 & 1 & -1 & 1 & 0 \\
					& 0 & 0 & 0 & 1 & -1 & 1 \\
					& 1 & 0 & 0 & 0 & 1 & -1 \\[.2cm]
\end{array}
\]

\medskip
\noindent
(since the fans are the same, the type cones are the same and we therefore present them together).
\cref{fig:nonkissingFans}\,(right) gives another example of brick but $2$-cyclic gentle quiver for which \cref{prop:interestingExchangeablePairsNKC,prop:extremalExchangeablePairsNKC} fail.






\end{remark}


\newpage
\part{Grothendieck group and relations between $\b{g}$-vectors}
\label{part:algebra}


\section{Relations for~$\b{g}$-vectors in finite type cluster algebras \\ via~$2$-Calabi--Yau triangulated categories}
\label{sec:clusterCategories}

In~\cite{Butler} and~\cite{Auslander1984}, M.~Butler and M.~Auslander showed that the relations in the Grothendieck group of an Artin algebra are generated by the ones given by almost-split sequences precisely when the algebra is of finite representation type.  This result was generalized to triangulated categories with certain finiteness conditions by J.~Xiao and B.~Zhu in \cite{XiaoZhu}, to~$\Hom{}$-finite, Krull-Schmidt exact categories with enough projectives by H.~Enomoto in~\cite{Enomoto}, to $(n+2)$-angulated categories by F.~Fedele in~\cite{Fedele}, and to $\Hom{}$-finite, Krull--Schmidt triangulated categories with a cogenerator by J.~Haugland in~\cite{Haugland}. Relations in the Grothendieck group of cluster categories have also been studied in~\cite{Palu-grothendieckGroup}.  Inspired by these results, we will show in this Section that a similar statement holds (see \cref{thm:relations-g-vecteurs}) for the~$\b{g}$-vectors (or indices) of objects in a~$2$-Calabi--Yau triangulated category.  We will further generalize these results in Section~\ref{sec:extricats} to the setting of Ext-finite, Krull-Schmidt, extriangulated categories with Auslander--Reiten--Serre duality admitting projective objects with certain properties.


\subsection{Setting}
\label{sec:setting}

Let~$\field$ be a field. Let~$\cat$ be a~$\field$-linear triangulated category with suspension functor~$\susp$. We fix a collection~$\ind(\cat)$ of representatives of isomorphism classes of indecomposable objects of~$\cat$. We will assume the following:
\begin{itemize}
\item $\cat$ is essentially small (in particular,~$\ind(\cat)$ is a set);
\item $\cat$ is~$\Hom{}$-finite: for each pair of objects~$X$ and~$Y$, the~$\field$-vector space~$\cat(X,Y)$ is finite-dimensional;
\item $\cat$ is Krull--Schmidt: the endomorphism algebra of any indecomposable object is local;
\item $\cat$ is~$2$-Calabi--Yau: for each pair of objects~$X$ and~$Y$, there is an isomorphism of bifunctors
\[
\cat(X,\susp Y) \to D\cat(Y,\susp X)
\]
where~$D = \Hom{\field}(-,\field)$ is the usual duality of vector spaces;
\item $\cat$ contains a basic cluster-tilting object~$T = \bigoplus_{i=1}^n T_i$:
\[
\textrm{for any object~$X$,} \quad \cat(T, \susp X) = 0 \quad \textrm{if and only if} \quad X \in \add(T),
\]
where~$\add(T)$ is the smallest additive full subcategory of~$\cat$ containing the~$T_i$'s and closed under isomorphisms.
\end{itemize}


\subsection{Statement of the theorem}
\label{subsec:Statement cluster cats}

We need to introduce some notations before stating the main result of this section. Let~$\Lambda\eqdef \End{\cat}(T)$, and let~$F$ be the functor
\[
F = \cat(T,-): \cat \xrightarrow{} \MOD \Lambda.
\]

\begin{proposition}[\cite{BuanMarshReiten, KellerReiten}]\label{prop:functor-F}

The functor~$F$ induces an equivalence of~$\field$-linear categories
\[
 F:\cat/(\susp T) \xrightarrow{} \MOD \Lambda,
\]
where~$(\susp T)$ is the ideal of morphisms factoring through an object of~$\add(\susp T)$ and~$\MOD \Lambda$ is the category of finite-dimensional right~$\Lambda$-modules. This equivalence induces further equivalences between~$\add(T)$ and the category of projective modules, and between~$\add(\susp^2 T)$ and that of injective modules.
\end{proposition}

For categories with a cluster-tilting object, the~$2$-Calabi--Yau condition implies other duality results which we shall need.

\begin{proposition}[\cite{Palu}]
For any pair of objects~$X$ and~$Y$ in~$\cat$, there is an isomorphism of bifunctors
\[
(\susp T)\cat(X, \susp Y) \xrightarrow{\cong} \cat(Y, \susp X)/(\susp T),
\]
where~$(\susp T)\cat(X, \susp Y)$ is the space of morphisms from~$X$ to~$Y$ factoring through~$(\susp T)$.
\end{proposition}

\begin{remark}
Although the field~$\field$ is assumed to be algebraically closed in~\cite{Palu}, this assumption is not needed in the proof, and the result is valid over any field.
\end{remark}

Finally, we need the existence of almost-split triangles in~$\cat$. Recall that a triangle
\[
X\xrightarrow{f} Y \xrightarrow{g} Z \xrightarrow{h} \susp X
\]
in~$\cat$ is \defn{almost-split} if~$X$ and~$Z$ are indecomposable,~$h$ is non-zero, and any non-section~$X\to X'$ factors through~$f$  (or equivalently, any non-retraction~$Z'\to Z$ factors through~$g$). We say that a triangulated category \defn{has almost-split triangles} if there is an almost-split triangle as above for any indecomposable object~$X$. 

\begin{proposition}[{\cite[Prop.~I.2.3]{ReitenVandenbergh}}]
Any triangulated category admitting a Serre functor has almost-split triangles. In particular, any~$2$-Calabi--Yau triangulated category has almost-split triangles.
\end{proposition}

\begin{definition}
\label{def:grothendieck-group}
\begin{enumerate}
\item Let~$\Ksp(\cat)$ be the \defn{split Grothendieck group} of~$\cat$, that is, the free abelian group generated by symbols~$[X]$, where~$[X]$ denotes the isomorphism class of~$X$ in~$\cat$, modulo the following relations: for any objects~$Y$ and~$Z$, we let~$[Y\oplus Z] = [Y] + [Z]$. 
  
\item Let~$K_0(\cat ; T)$ be the quotient of~$\Ksp(\cat)$ by the relations~$[X]+[Z]-[Y]$ for all triangles
\[
X\xrightarrow{} Y \xrightarrow{} Z \xrightarrow{h} \susp X
\]
with~$h \in (\susp T)$. Denote by~$\b{g}:\Ksp(\cat) \to K_0(\cat ; T)$ the canonical projection.
\end{enumerate}
\end{definition}

In particular,~$\Ksp(\cat)$ is isomorphic to a free abelian group over the set~$\ind(\cat)$. Considering the modified group~$K_0(\cat ; T)$ instead of the usual Grothendieck group is motivated by \cref{sec:applications-g-vectors}, where we study relations between~$\b{g}$-vectors.

\begin{definition}
\label{def:bilinear form}
For any two objects~$X$ and~$Y$ of~$\cat$, define
\[
\langle X, Y \rangle \eqdef \dim_{\field} \Hom{\Lambda}(FX,FY).
\]
This defines a bilinear form
\[
\langle -,-\rangle : \Ksp(\cat) \times \Ksp(\cat) \xrightarrow{} \Z.
\]
\end{definition}

\begin{notation}
\label{notation:ell_X}
\begin{enumerate}
\item For any indecomposable object~$X$ of~$\cat$, let
\[
X \to E \to \susp^{-1} X \to \susp X
\]
be an almost split triangle (unique up to isomorphism). We let
\[
\ell_X \eqdef [X] + [\susp^{-1} X]  - [E] \in \Ksp(\cat).
\]
 
\item For any indecomposable object~$Y$ of~$\cat$, let
\[
\susp Y \to E' \to  Y \to \susp^2 Y
\]
be an almost split triangle (unique up to isomorphism). We let
\[
r_Y \eqdef [Y] + [\susp Y]  - [E'] \in \Ksp(\cat).
\]
\end{enumerate}
\end{notation}

We can finally state the main theorem of this section, which is an analogue of the main result of~\cite{Auslander1984}.

\begin{theorem}
\label{thm:relations-g-vecteurs}
Let~$\cat$ be a category satisying the hypotheses of \cref{sec:setting}. Then~$\cat$ has only finitely many isomorphism classes of indecomposable objects if and only if the set
\[
L \eqdef \bigset{\ell_X}{X \in \ind(\cat) \ssm \add(\susp T)}
\]
generates the kernel of~$\b{g}:\Ksp(\cat) \to K_0(\cat ; T)$. 
In this case, the set~$L$ is a basis of the kernel of~$\b{g}$, and for any $x \in \ker(\b{g})$, we have that
\[
x= \sum_{A \in \ind(\cat) \ssm \add(\susp T)} \frac{\langle x, A \rangle}{\langle \ell_A, A \rangle} \ell_A.
\]
\end{theorem}

\begin{corollary}
\label{coro:meshes positively generate cluster cats}
Assume that~$\ind(\cat)$ is finite. Let~$X\xrightarrow{} E \xrightarrow{} Y \xrightarrow{h} \susp X$ be a triangle with~$h \in (\susp T)$. Then the element~$x=[X]+[Y]-[E]$ of the kernel of~$\b{g}$ is a non-negative linear combination of the~$\ell_A$, with~$A \in \ind(\cat)\ssm \add(\susp T)$.
\end{corollary}

\begin{proof}
We know that~$x=\sum_{A \in \ind(\cat) \ssm \add(\susp T)} \frac{\langle x, [A] \rangle}{\langle \ell_A, [A] \rangle} \ell_A$; since~$\langle \ell_A, [A] \rangle$ is positive by \cref{lem:bilinear-form}, we only need to show that each~$\langle x, [A] \rangle$ is non-negative. 
The functor~$F=\cat(T,-)$ induces an exact sequence
\[
FX \to FE \to FY \to 0,
\]
which in turn induces an exact sequence
\[
0\to \Hom{\Lambda}(FY, FA) \to \Hom{\Lambda}(FE, FA) \xrightarrow{f} \Hom{\Lambda}(FX, FA) \to \coker(f) \to 0.
\]
Therefore~$\langle x, [A] \rangle = \dim_{\field} \coker(f) \ge 0$.
\end{proof}


\subsection{Proof of \cref{thm:relations-g-vecteurs}}

Before we can prove \cref{thm:relations-g-vecteurs}, we need to recall the following results and definition from~\cite{DehyKeller,Palu}.

\begin{proposition}
Let~$X$ be an object of~$\cat$. Then there exists a triangle
\[
T_1^X \to T_0^X \to X \to \susp T_1^X
\]
with~$T_1^X$ and~$T_0^X$ in~$\add(T)$.
\end{proposition}

\begin{definition}
The \defn{index} of an object~$X$ is the element
\[
\ind_T(X) \eqdef [T_0^X] - [T_1^X] \in \Ksp \big( \add(T) \big).
\]
\end{definition}

The notion of index is very close to the definition of the map~$\b{g}:\Ksp(\cat)\to K_0(\cat ; T)$ of \cref{def:grothendieck-group}. The link is given by the following result.

\begin{proposition}[\cite{Palu}]
Let~$X\xrightarrow{} Y \xrightarrow{} Z \xrightarrow{h} \susp X$ be a triangle.
Then
\[
\ind_T(X) + \ind_T(Z) - \ind_T(Y) = 0
\]
if and only if~$h \in (\susp T)$.
\end{proposition}

\begin{corollary}
\label{coro:grothendieck-g-vectors}
There is an isomorphism~$\phi:K_0(\cat ; T) \to \Ksp \big( \add(T) \big)$ such that~$\ind_T = \phi\circ \b{g}$. In particular,~$K_0(\cat ; T)$ is a free abelian group generated by the~$[T_i]$.
\end{corollary}

We can now begin the proof of \cref{thm:relations-g-vecteurs}

\begin{lemma}
\label{lem:bilinear-form-suspention-T}
If~$X$ or~$Y$ lie in~$\add(\susp T)$, then~$\langle X, Y \rangle = 0$.
\end{lemma}

\begin{proof}
This is because~$F\susp T = 0$.
\end{proof}

\begin{lemma}
\label{lem:almost-split}
Let~$X\xrightarrow{} Y \xrightarrow{} \susp^{-1} X \xrightarrow{h} \susp X$ be an almost-split triangle. Then~$X\notin \add(\susp T)$ if and only if~$h \in (\susp T)$.
\end{lemma}

\begin{proof}
If~$X \in \add(\susp T)$, then~$h$ cannot be in~$(\susp T)$, otherwise it would be zero since~$\cat(T, \susp T) =0$. 
 
Assume now that~$X\notin \add(\susp T)$. Let~$\field_X$ be the residue field of the algebra~$\End{\cat}(X)$. 
By definition of an almost-split triangle,~$h$ is in the socle of the right~$\End{\cat}(X)$-module~$\cat(\susp^{-1} X, \susp X)$. Moreover, this socle is a one-dimensional~$\field_X$-vector space; indeed, the~$2$-Calabi--Yau condition gives an isomorphism
\[
\cat(\susp^{-1} X, \susp X) \cong D\cat(X, X).
\]
Thus the socle of the right module~$\cat(\susp^{-1} X, \susp X)$ has the same~$\field_X$ dimension as the top of the left module~$\cat(X,X)$. Since~$X$ is indecomposable,~$\cat(X,X)$ is local, and its top is one-dimensional over~$\field_X$.
 
Now,~$(\susp T)\cat(\susp^{-1} X, \susp X)$ is a sub-module of~$\cat(\susp^{-1} X, \susp X)$. Therefore, if~$(\susp T)\cat(\susp^{-1} X, \susp X)$ is non-zero, then it contains the one-dimensional socle of~$\cat(\susp^{-1} X, \susp X)$, and thus contains~$h$.  
By~\cite{Palu},
\[
(\susp T)\cat(\susp^{-1} X, \susp X) \cong D\cat(X, X)/(\susp T). 
\]
The identity morphism of~$X$ is not in~$(\susp T)$, since~$X$ is not in~$\add(\susp T)$. Thus the right-hand side is non-zero, and so neither is the left-hand side. By the above, this implies~${h \in (\susp T)\cat(\susp^{-1} X, \susp X)}$, which finishes the proof.
\end{proof}

\begin{lemma}
\label{lem:bilinear-form}
Let~$A$ and~$B$ be two indecomposable objects of~$\cat$.
\begin{enumerate}
\item If~$A \notin \add(\susp T)$, then
\[
\langle \ell_A, B \rangle = 
\begin{cases}
0 & \textrm{if~$A\ncong B$;} \\
\dim_{\field}\field_A & \textrm{if~$A\cong B$.}
\end{cases}
\]
 
\item If~$B \notin \add(\susp T)$, then
\[
\langle A, r_B \rangle =
\begin{cases}
0 & \textrm{if~$A\ncong B$;} \\
\dim_{\field}\field_B & \textrm{if~$A\cong B$.}
\end{cases}
\]
\end{enumerate} 
\end{lemma}

\begin{proof}
We only prove the first assertion; the second one is proved dually. Assume that ${A\notin\add(\susp T)}$. Let
\[
A \xrightarrow{f} E \xrightarrow{} \susp^{-1} A \xrightarrow{h} \susp A
\]
be an almost-split triangle. By \cref{lem:almost-split}, the morphism~$h$ is in~$(\susp T)$. Applying the functor ${F=\cat(T,-)}$, we get an exact sequence
\[
FA\xrightarrow{Ff} FE \xrightarrow{} F\susp^{-1} A \xrightarrow{} 0.
\]
Applying now the functors~$\cat(-, B)$ and~$\Hom{\Lambda}(-, FB)$, we get a commutative diagram whose rows are exact sequences and whose vertical maps are surjective by \cref{prop:functor-F}.
\[
\hspace{-.3cm}
\xymatrix{ & \cat(\susp^{-1} A, B)\ar[r]\ar@{->>}[d] & \cat(E, B)\ar[r]^{f^*}\ar@{->>}[d] & \cat(A,B)\ar[r]\ar@{->>}[d] & \coker(f^*)\ar[r]\ar@{->>}[d] & 0 \\
0 \ar[r] & \Hom{\Lambda}(F\susp^{-1} A, FB)\ar[r] & \Hom{\Lambda}(FE, FB)\ar[r]^{Ff^*} & \Hom{\Lambda}(FA, FB)\ar[r] & \coker(Ff^*)\ar[r] & 0.
}
\]

If~$B\ncong A$, then the definition of an almost-split triangle implies that~$f^*$ is surjective. Thus $\coker(f^*)=0$, so that~$\coker(Ff^*) = 0$, and by additivity of the dimension in exact sequences, we get that~$\langle \ell_A, B \rangle =0$.
 
If~$B\cong A$, then~$\coker(f^*)$ is isomorphic to the residue field~$\field_A$ of~$\cat(A,A)$. Since~$A$ is not in~$\add(\susp T)$, the ideal of endomorphisms of~$A$ factoring through an object of~$\add(\susp T)$ is contained in the maximal ideal of~$\cat(A,A)$. Therefore, the rightmost morphism in the commutative diagram is an isomorphism, so~$\coker(Ff^*)$ is isomorphic to~$\field_A$.
\end{proof}

\begin{lemma}
\label{lem:coefficients}
Let~$x \in \Ksp(\cat)$, and write~\[x=\sum_{A \in \ind(\cat)} \lambda_A [A].\]  Then for any~$A\in\ind(\cat) \ssm \add(\susp T)$, we have that
\[
\lambda_A = \frac{\langle \ell_A, x \rangle}{\langle \ell_A, A \rangle} = \frac{\langle x, r_A \rangle}{\langle A, r_A \rangle}.
\]
\end{lemma}

\begin{proof}
Let~$B \in \ind(\cat) \ssm \add(\susp T)$. Applying \cref{lem:bilinear-form}, we get that
\[
\langle \ell_B, x \rangle = \Big\langle \ell_B, \sum_{A \in \ind(\cat)} \lambda_A [A] \Big\rangle = \sum_{A \in \ind{\cat}} \lambda_A \langle \ell_B, [A] \rangle = \lambda_B \langle \ell_B, [B] \rangle.
\]
The equality~$\lambda_A = \frac{\langle x, r_A \rangle}{\langle A, r_A \rangle}$ is proved in a similar way.
\end{proof}

\begin{corollary}
\label{lem:test-susp-T}
Let~$x \in \Ksp(\cat)$. Then the following are equivalent.
\begin{enumerate}
\item $x \in \Ksp \big( \add(\susp T) \big)$;
\item $\langle x, [A]\rangle = 0$ for all~$A \in \ind(\cat)$;
\item $\langle [A], x \rangle = 0$ for all~$A \in \ind(\cat)$.
\end{enumerate}
\end{corollary}

\begin{proof}
We will only proove that (1) is equivalent to (2); the proof that (1) is equivalent to (3) is similar. 
 
Assume that (2) holds. Then, by \cref{lem:coefficients}, we have that
\[
x = \sum_{A \in \ind(\cat) \ssm \add(\susp T)} \frac{\langle \ell_A, x \rangle}{\langle \ell_A, A \rangle}[A] + \sum_{i=1}^n \lambda_{\susp T_i} [\susp T_i] = \sum_{i=1}^n \lambda_{\susp T_i} [\susp T_i].
\]
Thus~$x \in \Ksp \big( \add(\susp T) \big)$, and (1) holds.
 
Assume now that (1) holds. Then~$\langle x, [A]\rangle = 0$ for any~$A$ by \cref{lem:bilinear-form-suspention-T}. Thus (2) holds.
\end{proof}

\begin{proposition}
\label{prop:free-set}
\begin{enumerate}
\item The set~$\set{[\ell_A]}{A \in \ind(\cat) \ssm \add(\susp T)} \cup \set{[\susp T_i]}{i \in [n]}$ is free in~$\Ksp(\cat)$.
\item The set~$\set{[r_A]}{A \in \ind(\cat) \ssm \add(\susp T)} \cup \set{[\susp T_i]}{i \in [n]}$ is free in~$\Ksp(\cat)$.
\end{enumerate}
\end{proposition}

\begin{proof}
We only prove (1); the proof of (2) is similar. Assume that~
\[
x= \sum_{A \in \ind(\cat) \ssm \add(\susp T)} \lambda_A \ell_A + \sum_{i=1}^n \lambda_i [\susp T_i] = 0.
\]
Then~$\langle x, [A] \rangle = 0$ for all~$A\in\ind(\cat) \ssm \add(\susp T)$. But~$\langle x, [A] \rangle = \lambda_A$ by \cref{lem:bilinear-form}. Thus~
\[
x = \sum_{i=1}^n \lambda_i [\susp T_i] = 0.
\]
But the~$[\susp T_i]$ are linearly independent in~$\Ksp(\cat)$. Thus~$\lambda_i = 0$ for all~$i \in [n]$. This finishes the proof.
\end{proof}

\begin{proposition}
Assume that~$\ind(\cat)$ is finite. Then the set~$\set{[\ell_A]}{A \in \ind(\cat) \ssm \add(\susp T)}$ is a basis of the kernel of~$\b{g} : \Ksp(\cat) \to K_0(\cat ; T)$. Moreover, for any~$x \in \ker \b{g}$, we have that
\[
x= \sum_{A \in \ind(\cat) \ssm \add(\susp T)} \frac{\langle x, A \rangle}{\langle \ell_A, A \rangle} \ell_A.
\]
\end{proposition}

\begin{proof}
By \cref{prop:free-set}, the set is free. Let~$x \in \ker \b{g}$. Consider the element
\[
z = x- \sum_{A \in \ind(\cat) \ssm \add(\susp T)} \frac{\langle x, A \rangle}{\langle \ell_A, A \rangle} \ell_A.
\]
Then for any~$B \in \ind{\cat}$, we have that
\[
\langle  z, [B] \rangle = \Big\langle x- \sum_{A \in \ind(\cat) \ssm \add(\susp T)} \frac{\langle x, A \rangle}{\langle \ell_A, A \rangle} \ell_A, [B] \Big\rangle = \langle x, [B] \rangle - \langle x, [B] \rangle = 0, 
\]
where the second equality is obtained by using \cref{lem:bilinear-form}. By \cref{lem:test-susp-T}, this implies that~$z \in \Ksp \big( \add(\susp T) \big)$. Since~$z \in \ker(\b{g})$ and since~$\b{g}$ is injective on~$\Ksp \big( \add (\susp T) \big)$ by \cref{coro:grothendieck-g-vectors}, we get that~$z=0$. This finishes the proof.
\end{proof}

\begin{corollary}
If~$\ind(\cat)$ is finite, then the set
\[
\set{[\ell_A]}{A \in \ind(\cat) \ssm \add(\susp T)} \cup \set{[\susp T_i]}{i \in [n]}
\]
is a basis of~$\Ksp(\cat)$.
\end{corollary}

\begin{proof}
By \cref{prop:free-set}, the set is free. It suffices to prove that it generates~$\Ksp(\cat)$. Let~${x \in \Ksp(\cat)}$. Consider
\[
z = x - \sum_{A \in \ind{\cat}} \frac{\langle x, A \rangle}{\langle \ell_A, A\rangle} \ell_A.
\]
Then for any~$B \in \ind(\cat)$, we have that~$\langle z, [B] \rangle = 0$. By \cref{lem:test-susp-T}, this implies that ${z \in \Ksp \big( \add(\susp T) \big)}$, and finishes the proof.
\end{proof}

All that remains is to prove the converse in the statement of \cref{thm:relations-g-vecteurs}.

\begin{proposition}
Assume that the set~$\set{[\ell_A]}{A \in \ind(\cat) \ssm \add(\susp T)}$ is a basis of the kernel of~$\b{g}$. Then~$\ind(\cat)$ is finite.
\end{proposition}

\begin{proof}
Let~$x \in \ker \b{g}$, and write~$x = \sum_{[A] \in \ind(\cat)} \lambda_A[A]$, where the sum has finite support. For any~$B$ not in the support of the sum, we have that~$\langle x, [B] \rangle = 0$.
 
Now,~$[T]+[\susp T]$ is in the kernel of~$\b{g}$, but~$\langle [T] + [\susp T], [B] \rangle = \langle [T], [B] \rangle = 0$ if and only if ${B \in \add(\susp T)}$. Thus~$\ind(\cat)$ has to be finite.
\end{proof}


\subsection{Application to~$\b{g}$-vectors of cluster algebras of finite type}
\label{sec:applications-g-vectors}

In this section, we apply the results of \cref{subsec:Statement cluster cats} in order to prove \cref{prop:uniqueExchangePropertyCA,prop:meshMutations} of \cref{subsec:typeConeCA}.

We first recall the results on categorification of cluster algebras that we will use in the proofs.
\begin{theorem}
\label{thm:categorification}
Let~$\cat$ be the cluster category of a valued quiver of Dynkin type~$A,B,C,D,E,F$ or~$G$. Let~$T$ be a basic cluster-tilting object in~$\cat$, and let~$Q$ be the valued Gabriel quiver of~$\End{\cat}(T)$. Then there is a bijection
\[
\CC: \ind(\cat) \xrightarrow{} \{\textrm{cluster variables in the cluster algebra of~$Q$}\},
\]
expressing cluster variables as Laurent polynomials in the variables of the cluster~$\CC(T)$ and which has the follwing properties.
\begin{itemize}
\item[(i)] \emph{(}\cite[Prop.~3.2]{BuanMarshReinekeReitenTodorov}\emph{)} For any~$X,Y \in \ind(\cat)$, $\CC(X)$ and~$\CC(Y)$ are compatible if and only if~$\cat(X, \susp Y) = 0$.
\item[(ii)] \emph{(}\cite[Prop.~4.3]{FuKeller} and \cite[Prop.~2.2]{Palu}, where the proofs also work over an arbitrary base field\emph{)} For any~$X \in \ind(\cat)$, the~$\b{g}$-vector of~$\CC(X)$ is~$\b{g}([X])$, where we identify~$\Z^n$ and~$K_0(\cat ; T)$ via the isomorphism sending~$(a_1, \dots, a_n)$ to~$\sum_{i=1}^n a_i [T_i]$.
\item[(iii)] \emph{(}\cite[Thm.~7.5]{BuanMarshReinekeReitenTodorov}\emph{)} For any~$X,Y \in \ind(\cat)$, $\CC(X)$ and~$\CC(Y)$ are exchangeable if and only if~$\dim_{\field_X}\cat(X, \susp Y) = \dim_{\field_Y}\cat(X, \susp Y) = 1$.
\item[(iv)]\emph{(}\cite{BuanIyamaReitenScott} and \cite[Thm.~5.1]{BuanMarshReiten-mutation}, adapted to the cluster categories of species of Dynkin type\emph{)} Let~$T'$ be any basic cluster-tilting object of~$\cat$. Then the matrix~$B'$ of the associated seed~$(\CC(T'),B')$ is given by the multiplicities of the middle terms of the exchange triangles. 
\end{itemize}
\end{theorem}

The following definition is a cluster-tilting version of the positive mutation for~$\tau$-tilting modules, or for~$2$-term silting complexes.

\begin{definition}
\label{def:positive mutation}
Let~$x\in\cluster$ be a cluster variable in a given cluster and let~$x' \in \cluster'$ be obtained by mutating~$\cluster$ at~$x$.
The pair~$(x,x')$ is said to be a \defn{positive mutation} with respect to some initial cluster~$\cluster_\circ$ if there are~$X,Y\in \ind(\cat)$ such that~$\CC(X)=x$,~$\CC(Y)=x'$,~$\dim_{\field_X}\cat(X, \susp Y) = \dim_{\field_Y}\cat(X, \susp Y) = 1$ and, for any non-split triangle
\[
X\xrightarrow{} E \xrightarrow{} Y \xrightarrow{h} \susp X 
\]
(this triangle is unique up to isomorphism) we have~$h\in (\susp T)$, where~$T\in\cat$ is any basic cluster-tilting object such that~$\CC(T)=\cluster_\circ$.

The mutation is a \defn{mesh mutation} if the triangle can be chosen to be an almost-split triangle (with no assumption on~$h$).
\end{definition}

\begin{lemma}
\label{lem:exaclty one mesh mutation}
For every non-initial cluster variable~$x$, there exists precisely one pair~$(x,x')$ such that there exists a positive mesh mutation changing~$x$ into~$x'$.
\end{lemma}

\begin{proof}
Up to isomorphism, there is precisely one almost-split triangle of~$\cat$ starting in~$X$, where~$X$ is the object such that~$\phi(X) = x$.  Thus there is precisely one mesh relation starting at~$x$. By \cref{lem:almost-split}, the pair~$(x,x')$ is a positive mutation if and only~$X\notin \add(\susp T)$, that is, if and only if~$x$ is not an initial variable.
\end{proof}

\begin{lemma}
\label{lem:mesh mutations and meshes}
The set of positive mesh mutations~$(x,x')$ of \cref{def:positive mutation} is in bijection with the set~$\meshes$ of \cref{def:meshMutation}.
\end{lemma}

\begin{proof}
\enlargethispage{.3cm}
If~$(x,x')$ is a positive mesh mutation, then~$\{x,x'\}$ is in~$\meshes$.  Indeed, if~$X$ and~$X'$ are the corresponding objects in~$\cat$, then there is an almost-split triangle
\[
X\to E\to X' \to \susp E,
\]
so that~$X' = \tau^{-1}X$.  Moreover, there is a slice of the Auslander--Reiten quiver of~$\cat$ containing~$X$ and all indecomposable direct factors of~$E$.  The direct sum of indecomposable objects in this slice is a cluster-tilting object in which~$X$ is a source, and mutation at~$X$ gives~$X'$.  Thus~$\{x,x'\}$ is in~$\meshes$.  This gives an injective map from the set of positive mesh mutations to~$\meshes$.
 
\smallskip
Next, assume that~$\{x,x'\}\in \meshes$ starts in~$x$.  Then there is a cluster-tilting object~$T$ in~$\cat$ having~$X$ as a direct factor, and such that~$X$ is a source in~$T$ and mutation at~$X$ yields~$X'$.  Since~$X$ is a source, the mutation triangle starting and ending in~$X$ are
\[
X\to E\to X' \to \susp X \quad \textrm{and} \quad  X' \to 0 \to X \to \susp X'.
\]
Therefore~$X' = \susp^{-1}X = \tau^{-1}X$.  Thus the first triangle is the almost-split triangle starting in~$X$, since the dimension of~$\cat(X',\susp X)$ is~$1$ over~$\field_{X'}$.  Since the pair~$\{x,x'\}$ is not initial,~$\phi(X')$ is not an initial variable; in other words,~$X'$ is not in~$\add(T)$.  Therefore,~$X = \tau X' = \susp X'$ is not in~$\add\susp T$, so the mutation is a positive mesh mutation by \cref{lem:almost-split}.  This gives an injective map from~$\meshes$ to the set of positive mutations. 
\end{proof}

\begin{corollary}[\cref{lem:bijectionMeshMutations}]
\label{coro:number of positive meshes}
We have that~$|\meshes|$ is the number~$|\variables[\B_\circ]|-n$ of non-initial cluster variables.
\end{corollary}

\begin{proof}
This is a direct consequence of \cref{lem:exaclty one mesh mutation,lem:mesh mutations and meshes}.
\end{proof}

\begin{lemma}
\label{lemma:positive mutation triangles}
Let~$\mathcal{A}$ be any cluster algebra admitting a categorification by a Hom-finite,~$2$-Calabi--Yau, Krull--Schmidt, triangulated category with cluster-tilting objects.
Let~$x\in\cluster$ be a cluster variable in a given cluster and let~$x' \in \cluster'$ be obtained by mutating~$\cluster$ at~$x$.
Then precisely one of the two mutations~$\mu_x(\cluster)$ and~$\mu_{x'}(\cluster')$ is a positive mutation.
\end{lemma}

\begin{proof}
In view of \cref{thm:categorification}\,(iii), it follows from~\cite[Lem.~3.3]{Palu}.
\end{proof}

\begin{corollary}\emph{(\cref{prop:meshMutations})}
\label{corollary:proof of prop meshMutations}
For any initial exchange matrix~$\B_\circ$ of mutation type~$A, B, C, D$, $E, F$ or~$G$, the linear dependence between the $\b{g}$-vectors of any mutation can be decomposed into positive combinations of linear dependences between $\b{g}$-vectors of non-initial mesh mutations.
\end{corollary}

\begin{proof}
Let~$\cat$ be a cluster category of Dynkin type~$A,B,C,D,E,F$ or~$G$, and let $T\in\cat$ be some basic cluster-tilting object categorifying the given initial exchange matrix~$\B_\circ$. 
By \cref{lemma:positive mutation triangles}, we may assume that the mutation under consideration is a positive mutation.
In view of \cref{thm:categorification}\,(ii) and (iv) it is enough to show that the relation in~$K_0(\cat ; T)$ given by any triangle ${X\to E\to Y\xrightarrow{h} \susp X}$ with~$h\in(\susp T)$ is a positive linear combination of relations coming from Auslander--Reiten triangles with third term not in~$\add(\susp T)$.
This categorified statement is precisely \cref{coro:meshes positively generate cluster cats}.
\end{proof}

It also follows from \cref{thm:categorification} above that any cluster algebra of finite type has the unique exchange relation property (see \cref{def:uerp}).

\begin{corollary}
\label{corollary:UERPforCAproof}
\emph{(\cref{prop:uniqueExchangePropertyCA})}
Let~$\B_\circ$ be any finite type exchange matrix, and let~$\mathcal{A}(\B_\circ)$ be the associated cluster algebra without coefficients.
Then, for any exchangeable cluster variables~$x$ and~$x'$, the linear dependence $\sum_{y \in \cluster \cup \cluster'} \coefficient[{\b{g}_y}][\cluster][\cluster'] \, \b{g}_y = 0$ only depends on the pair~$(x,x')$ and not on the specific choice of clusters~$\cluster$ and~$\cluster'$ containing~$x$ and~$x'$ respectively and such that~${\cluster \ssm \{x\} = \cluster' \ssm \{x'\}}$.
\end{corollary}

\begin{proof}
By the results in additive categorification of cluster algebras recalled in \cref{thm:categorification}, the statement follows from \cref{lem:uerp categorified} below.
\end{proof}

\begin{lemma}
\label{lem:uerp categorified}
Let~$\cat$ be a cluster category of Dynkin type~$A,B,C,D,E,F$ or~$G$, and let~$X,Y \in \cat$ be such~$\CC(Y)$ is obtained by performing a positive mutation at~$\CC(X)$ in some cluster containing it.
Then the linear dependence between the~$\b{g}$-vectors given by the positive mutation is
\[
\b{g}_{\CC(X)} + \b{g}_{\CC(Y)} = \sum_i\b{g}_{\CC(E_i)}, 
\]
where~$X\to \bigoplus_i E_i \to Y \xrightarrow{h} \susp X$ is a non-split triangle and each~$E_i$ is indecomposable.
\end{lemma}

\begin{proof}
Let~$T\in\cat$ be a basic cluster tilting object such that~$\CC(T)$ is the initial cluster.
According to \cref{lemma:positive mutation triangles}, the morphism~$h$ belongs to the ideal~$(\susp T)$.
Therefore, the triangle ${X\to \bigoplus_i E_i \to Y \xrightarrow{h} \susp X}$ induces the relation~$[X]+[Y]=\sum_i [E_i]$ in~$K_0(\cat ; T)$.
\cref{thm:categorification}\,(ii) gives the equality ${\b{g}_{\CC(X)} + \b{g}_{\CC(Y)} = \sum_i\b{g}_{\CC(E_i)}}$, where the~$\b{g}$-vectors are computed with respect to the initial seed~$\CC(T)$.
By \cref{thm:categorification}\,(iii), the middle term of the triangle is uniquely defined, up to isomorphism, by~$X$, and~$Y$, and the triangle is therefore an exchange triangle, with respect to any mutation at~$X$: it does not depend on the choice of a cluster tilting object containing~$X$.
By \cref{thm:categorification}\,(iv) this equality is one of the two equalities in \cref{lem:linearDependencegvectorsCA}.
\end{proof}


\section{Relations for~$\b{g}$-vectors in brick algebras via extriangulated categories}
\label{sec:extricats}

The main definitions and first properties on extriangulated categories are recalled after the statements of the main results of this section.
See \cref{sec:recollections extricats} below.
The reader that does not want to dwell into details can safely skip \cref{sec:recollections extricats} and focus on any of the first three examples given in \cref{rem:examples extricats}: triangulated categories, exact categories (that are small or have enough projectives or injectives) and extension-closed subcategories of triangulated categories, such as the category~$K^{[-1,0]}(\proj \Lambda)$, are examples of extriangulated categories.


\subsection{Setting}
\label{setting:extricat}

We indifferently call \emph{conflations}, written~$X\infl Y\defl Z$, or \emph{extriangles}, written ${X\infl Y\defl Z \overset{\delta}{\dashrightarrow}}$, the analogues of short exact sequences or triangles in an extriangulated category.

\smallskip
\para{Setting for Section~\ref{sec:prelim on extricats}}
In that section, we let~$\cat$ be an extriangulated category with a fixed full additive subcategory~$\tc$, stable under isomorphisms, under taking direct summands, and satisfying the following three properties:
\begin{enumerate}
\item Every~$T\in\tc$ is projective in~$\cat$;
\item For each~$T\in\tc$, the morphism~$T\to 0$ is an inflation for the extriangulated structure of~$\cat$;
\item For each~$X\in\cat$, there is an extriangle~$T_1^X\infl T_0^X \defl X \overset{\delta_X}{\dashrightarrow}$ in~$\cat$ with~$T_0^X, T_1^X$ in~$\tc$.
\end{enumerate}

\smallskip
\para{Setting for Section~\ref{sec:Statment extricats}}
In that section, we keep the previous setting, but assume moreover that~$\cat$ is Krull--Schmidt, $\field$-linear, Ext-finite, and has Auslander--Reiten--Serre duality, and that the subcategory~$\tc$ is of the form~$\add T$, where~$T=T_1\oplus\cdots\oplus T_n$ is a basic object.

\begin{remark}
\label{rem:examples extricats}
Examples of categories satisfying the properties above are
\begin{itemize}
\item $2$-Calabi--Yau triangulated categories admitting a cluster-tilting object (see \cref{sec:2CYTriangulated});
\item for any Artin algebra~$\Lambda$, the category~$K^{[-1,0]}(\proj \Lambda)$ of complexes of finitely generated projective~$\Lambda$-modules concentrated in degrees~$-1$ and~$0$, with morphisms considered up to homotopy (see \cref{sec:Kbproj});
\item more generally, extriangulated categories~$\cat$ constructed as follows: if~$\tc$ is a rigid subcategory of an extriangulated category~$\ec$ (with some assumption ensuring that condition (2) holds) we let~$\cat$ be the full subcategory of~$\ec$ whose objects satisfy condition (3), equipped with the extriangulated structure obtained from that of~$\ec$ by considering those extriangles in~$\ec$ all whose terms belong to~$\cat$ and whose deflation is~$\tc$-epi (see \cref{sec:prT}).
\end{itemize}
\end{remark}


\subsection{Statement of preliminary results on extriangulated categories}
\label{sec:prelim on extricats}

So as to be able to state the main theorem of \cref{sec:extricats}, we need a few results on extriangulated categories.
However, all proofs are postponed to \cref{sec:proofs extricats}.

We let~$\cat$ be as in \cref{setting:extricat}.

\begin{notation}
For any object~$T\in\tc$, we fix an extriangle~$T\infl 0 \defl \susp T \dashrightarrow$.
\end{notation}

\begin{remark}
As will be proven below (\cref{rem:equivalence Sigma} and \cref{coro:Sigma T injectives}), this notation extends to an equivalence of categories from the category~$\tc$ of projective objects in~$\cat$ to the category~$\susp\tc$ of injective objects in~$\cat$.
\end{remark}

\begin{notation}
We let~$\Modt$ denote the category of additive functors from~$\tc$ to Abelian groups, and~$\modt$ its full subcategory of functors that are finitely presented, \ie cokernels of morphisms between representable functors.
We let~$F:\cat\to\Modt$ be the functor defined on objects by sending~$X\in\cat$ to~$\cat(-,X)|_\tc$.
\end{notation}

\begin{lemma}
\label{lem:finitely presented}
For any~$X\in\cat$, the functor~$FX$ is finitely presented.
We thus have a functor
\[
F:\cat\to\modt.
\]
\end{lemma}

\cref{prop:KRextricat} below extends similar results from~\cite{BuanMarshReiten,KellerReiten,KoenigZhu,IyamaYoshino} (see also \cref{prop:functor-F}) to the setting under consideration. We note that the proof requires minor modifications.

\begin{proposition}
\label{prop:KRextricat}
The functor~$F$ induces an equivalence of categories
\[
F: \cat/(\susp\tc) \to \modt
\]
where~$(\susp\tc)$ is the ideal of morphisms factoring through an object of the form~$\susp T$, for some~$T\in\tc$.
\end{proposition}

\begin{remark}
Because the category~$\cat$ does not have weak kernels in general, the category~$\modt$ might not be abelian.
However, in all our applications, the subcategory~$\tc$ is of the form~$\add(T)$ for some object~$T$.
In that case,~$\modt$ is equivalent to~$\MOD\End{\cat}(T)$ and thus abelian.
\end{remark}

\begin{definition}
We let~$\kzero{\cat}$ denote the \defn{Grothendieck group} of~$\cat$, that is, the quotient of the free abelian group generated by symbols~$[X]$, for each $X\in\cat$, by the relations~$[X]-[Y]+[Z]$, for each conflation~$X\infl Y\defl Z$ in~$\cat$.
\end{definition}

\begin{remark}
Since~$\tc$ is extension-closed in~$\cat$, it inherits an extriangulated structure.
Because~$\tc$ is made of projective objects in~$\cat$, its extriangulated structure splits and we have~$\kzero{\tc}\cong\Ksp(\tc)$.
\end{remark}

The notion of index from~\cite{DehyKeller,Palu} generalises to our current setting.

\begin{definition}
For any object~$X\in\cat$, fix some extriangle
\[
T_1^X \infl T_0^X \defl X \overset{\delta_X}{\dashrightarrow}
\]
and define the \defn{index} of~$X$ by
\[
\ind_\tc X = [T_0^X]-[T_1^X]\in\kzero{\tc}.
\]
\end{definition}

\begin{proposition}
\label{prop:index iso extricat}
 The assignment~$X\mapsto\ind_\tc X$ is well-defined and induces an isomorphism
\[
\ind_\tc : \kzero{\cat} \overset{\cong}{\longrightarrow} \kzero{\tc}
\]
of abelian groups.
\end{proposition}


\subsection{Statement of the theorem}
\label{sec:Statment extricats}

Let~$\cat$ be as in \cref{setting:extricat}.

\begin{remark}
Assume moreover that~$\cat$ is $\field$-linear and Ext-finite.
By the preliminary results of \cref{sec:recollections extricats}, we can apply~\cite[Prop.~4.2]{IyamaNakaokaPalu} to obtain that~$\modt$ is Hom-finite.
In particular, if $\tc=\add T$, then $\End{\cat}(T)$ is a finite-dimensional~$\field$-algebra.
It thus also follows that \cref{def:bilinear form} still makes sense in this more general setup: For any~$X,Y\in\cat$,
\[
\langle X,Y\rangle \eqdef \dim_\field \Hom{\tc}(FX,FY) = \dim_\field \cat/(\susp\tc)(X,Y).
\]
\end{remark}

When this makes sense, we make use of \cref{notation:ell_X}: if
\[
X \infl E \defl Y \dashrightarrow
\]
is an almost-split sequence (see~\cite{IyamaNakaokaPalu} for a definition in this setting), we let
\[
\ell_X \eqdef [X]+[Y]-[E]\in\Ksp(\cat).
\]

\begin{theorem}
\label{thm:extricats}
Assume that~$\cat$ is a~$\field$-linear, Ext-finite, Krull--Schmidt, extriangulated category with Auslander--Reiten--Serre duality.
Assume that~$T$ is a projective object of~$\cat$ such that any~$X\in\cat$ admits a conflation~$T_1^X\infl T_0^X\defl X$ with~$T_0^X, T_1^X\in\add(T)$, and the morphism~$T\to 0$ is an inflation. Fix a conflation~$T\to 0 \to \susp T$. Then~$\cat$ has only finitely many isomorphism classes of indecomposable objects if and only if the set
\[
L \eqdef \bigset{\ell_X}{X \in \ind(\cat) \ssm \add(\susp T)}
\]
generates the kernel of the canonical projection~$\b{g}:\Ksp(\cat) \to \kzero{\cat}$. 
In this case, the set~$L$ is a basis of the kernel of~$\b{g}$, and for any $x \in \ker(\b{g})$, we have that
\[
x= \sum_{A \in \ind(\cat) \ssm \add(\susp T)} \frac{\langle x, A \rangle}{\langle \ell_A, A \rangle} \ell_A.
\]
\end{theorem}

\begin{corollary}
\label{coro:meshes positively generate extricats}
Assume that~$\ind(\cat)$ is finite.
Let~$X\infl E\defl Y\dashrightarrow$ be any extriangle.
Then the element~$x=[X]+[Y]-[E]$ of the kernel of~$\b{g}$ is a non-negative linear combination of the~$\ell_A$, with~$A\in\ind(\cat)\ssm \add(\susp T)$.
\end{corollary}


\subsection{Recollections on extriangulated categories}
\label{sec:recollections extricats}

Extriangulated categories, recently introduced in~\cite{NakaokaPalu}, axiomatize extension-closed subcategories of triangulated categories in a (moderately) similar way that Quillen's exact categories axiomatize extension-closed subcategories of abelian categories.
They appear in representation theory in relation with cotorsion pairs~\cite{ChangZhouZhu-ClusterSubalgebras,ZhaoHuang-Phantom,LiuNakaoka-Hearts,Liu-LocalizationsHearts}, with Auslander--Reiten theory~\cite{IyamaNakaokaPalu}, with cluster algebras, mutations, or cluster-tilting theory~\cite{ChangZhouZhu-ClusterSubalgebras,Pressland,ZhouZhu-TriangulatedQuotient,LiuZhou-Abelian,LiuZhou-AbelianII,LiuZhou-Gorenstein}, with Cohen--Macaulay dg-modules in the remarkable~\cite{Jin}.
We also note the generalization, called~$n$-exangulated categories~\cite{HerschendLiuNakaokaI,HerschendLiuNakaokaII}, to a version suited for higher homological algebra.

An extriangulated category is the data of an additive category~$\cat$, of an additive bifunctor ${\mathbb{E}:\cat^\mathrm{op}\times\cat\to \mathit{Ab}}$ modelling the~$\operatorname{Ext}^1$-bifunctor, and of an additive realization~$\mathfrak{s}$ sending each element~$\delta\in\mathbb{E}(Z,X)$ to some (equivalence classe of) diagram~$X\to Y\to Z$ modelling the short exact sequences or triangles.
Some axioms, inspired from the case of extension-closed subcategories of triangulated categories have to be satisfied.

More specifically: fix an additive category~$\cat$, and an additive bifunctor~$\mathbb{E}: \cat^\mathrm{op}\times\cat\to\mathit{Ab}$, where $\mathit{Ab}$ is the category of abelian groups.

\begin{definition}
For any~$X,Z\in\cat$, an element~$\delta\in\mathbb{E}(Z,X)$ is called an~\defn{$\mathbb{E}$-extension}.
A \defn{split}~$\mathbb{E}$-extension is a zero element~$0\in\mathbb{E}(Z,X)$, for some objects~$X,Z\in\cat$.
For any two~$\mathbb{E}$-exten\-sions~${\delta\in\mathbb{E}(Z,X)}$, $\delta'\in\mathbb{E}(Z',X')$, the additivity of~$\cat$,~$\mathbb{E}$ permits to define the~$\mathbb{E}$-extension
\[
\delta\oplus\delta'\in\mathbb{E}(Z\oplus Z',X\oplus X').
\]
\end{definition}

\begin{remark}
Let~$\delta\in\mathbb{E}(Z,X)$ be an~$\mathbb{E}$-extension. By functoriality, any morphisms~$f\in\cat(X,X')$ and~$h\in\cat(Z',Z)$ induce~$\mathbb{E}$-extensions~$\mathbb{E}(Z,f)(\delta)\in\mathbb{E}(Z,X') \text{ and } \mathbb{E}(h,X)(\delta)\in\mathbb{E}(Z',X)$.
For short, we write~$f_\ast\delta$ and~$h^\ast\delta$ instead.
Using those notations, we have, in~$\mathbb{E}(Z',X')$
\[
\mathbb{E}(h,f)(\delta)=h^\ast f_\ast\delta=f_\ast h^\ast\delta.
\]
\end{remark}

\begin{definition}
A \defn{morphism}~$(f,h):\delta\to\delta'$ of~$\mathbb{E}$-extensions~$\delta\in\mathbb{E}(Z,X)$,~$\delta'\in\mathbb{E}(Z',X')$ is a pair of morphisms~$f\in\cat(X,X')$ and~$h\in\cat(Z,Z')$ in~$\cat$, such that~$f_\ast\delta=h^\ast\delta'.$
\end{definition}

\begin{definition}
\label{DefSqEquiv}
Let~$X,Z\in\cat$ be any two objects. Two sequences of morphisms in~$\cat$
\[
X \overset{x}{\longrightarrow}Y\overset{y}{\longrightarrow}Z \text{ and } X\overset{x'}{\longrightarrow}Y'\overset{y'}{\longrightarrow}Z
\]
are said to be \defn{equivalent} if there exists an isomorphism~$g\in\cat(Y,Y')$ such that the following diagram commutes.
\[
\xy
(-20,0)*+{X}="X";
(5,0)*+{}="1";
(0,6)*+{Y}="Y";
(0,-6)*+{Y'}="Y'";
(-5,0)*+{}="2";
(20,0)*+{Z}="Z";
{\ar^{x} "X";"Y"};
{\ar^{y} "Y";"Z"};
{\ar_{x^{\prime}} "X";"Y'"};
{\ar_{y^{\prime}} "Y'";"Z"};
{\ar^{g}_{\cong} "Y";"Y'"};
{\ar@{}|{} "X";"1"};
{\ar@{}|{} "2";"Z"};
\endxy
\]
The equivalence class of~$X \xrightarrow{x} Y \xrightarrow{y} Z$ is denoted by~$[X \xrightarrow{x} Y \xrightarrow{y} Z]$.
\end{definition}

\begin{notation}
For any~$X,Y,Z,A,B,C\in\cat$, and any~$[X \xrightarrow{x} Y \xrightarrow{y} Z]$,~$[A \xrightarrow{a} B \xrightarrow{b} C]$, we let
\[ 0=[X\overset{\left[\bsm1\\0\esm\right]}{\longrightarrow} X\oplus Y \overset{\left[\bsm0\;1\esm\right]}{\longrightarrow} Y]
\]
and
\[
[X \xrightarrow{x} Y \xrightarrow{y} Z]\oplus [A \xrightarrow{a} B \xrightarrow{b} C] = [X\oplus A \overset{\left[\bsm x \; 0 \\ 0\; a\esm\right]}{\longrightarrow} Y\oplus B \overset{\left[\bsm y\;0\\0\;b \esm\right]}{\longrightarrow} Z\oplus C].
\]
\end{notation}

\begin{definition}
An \defn{additive realization}~$\mathfrak{s}$ is a correspondence associating, with~$\mathbb{E}$-extension ${\delta\in\mathbb{E}(Z,X)}$, an equivalence class~$\mathfrak{s}(\delta)=[X\xrightarrow{x}Y\xrightarrow{y}Z]$ and satisfying the following condition: 
\begin{itemize}
\item[$(\ast)$] Let~$\delta\in\mathbb{E}(Z,X)$ and~$\delta'\in\mathbb{E}(Z',X')$ be any pair of~$\mathbb{E}$-extensions, with
\[\mathfrak{s}(\delta)=[X\xrightarrow{x}Y\xrightarrow{y}Z] \text{ and } \mathfrak{s}(\delta')=[X'\xrightarrow{x'}Y'\xrightarrow{y'}Z'].\]
Then, for any morphism~$(f,h):\delta\to\delta'$, there exists~$g\in\cat(Y,Y')$ such that the following diagram commutes:
\[
\xy
(-12,6)*+{X}="0";
(0,6)*+{Y}="2";
(12,6)*+{Z}="4";
(-12,-6)*+{X'}="10";
(0,-6)*+{Y'}="12";
(12,-6)*+{Z'}="14";
{\ar^{x} "0";"2"};
{\ar^{y} "2";"4"};
{\ar_{f} "0";"10"};
{\ar^{g} "2";"12"};
{\ar^{h} "4";"14"};
{\ar_{x'} "10";"12"};
{\ar_{y'} "12";"14"};
{\ar@{}|{} "0";"12"};
{\ar@{}|{} "2";"14"};
{\ar@{}|\circlearrowright "0";"12"};
{\ar@{}|\circlearrowright "2";"14"};
\endxy
\]
\end{itemize}

The sequence~$X\xrightarrow{x}Y\xrightarrow{y}Z$ is said to \defn{realize}~$\delta$ if~$\mathfrak{s}(\delta)=[X\xrightarrow{x}Y\xrightarrow{y}Z]$, and the triple~$(f,g,h)$ is said to realize~$(f,h)$ if the diagram in~$(\ast)$ commutes.
\end{definition}

\begin{definition}
A realization of~$\mathbb{E}$ is called an \defn{additive realization} if the following conditions are satisfied:
\begin{enumerate}
\item For any~$X,Z\in\cat$, the realization of the split~$\mathbb{E}$-extension~$0\in\mathbb{E}(Z,X)$ is given by~$\mathfrak{s}(0)=0$.
\item For any two~$\mathbb{E}$-extensions~$\delta\in\mathbb{E}(Z,X)$ and~$\delta'\in\mathbb{E}(Z',X')$, the realization of~$\delta\oplus\delta'$ is given by~$\mathfrak{s}(\delta\oplus\delta')=\mathfrak{s}(\delta)\oplus\mathfrak{s}(\delta')$.
\end{enumerate}
\end{definition}

\begin{definition}(\cite[Def.~2.12]{NakaokaPalu})
A triple~$(\cat,\mathbb{E},\mathfrak{s})$ is called an \defn{extriangulated category} if the following holds:
\begin{itemize}
\item[{\rm (ET1)}] $\mathbb{E}:\cat^{\mathrm{op}}\times\cat\to\mathit{Ab}$ is an additive bifunctor;
\item[{\rm (ET2)}] $\mathfrak{s}$ is an additive realization of~$\mathbb{E}$;
\item[{\rm (ET3)}] Let~$\delta\in\mathbb{E}(Z,X)$ and~$\delta'\in\mathbb{E}(Z',X')$ be~$\mathbb{E}$-extensions, respectively realized by~$X\xrightarrow{x}Y\xrightarrow{y}Z$ and~$X'\xrightarrow{x'}Y'\xrightarrow{y'}Z'$. Then, for any commutative square
\[
\xy
(-12,6)*+{X}="0";
(0,6)*+{Y}="2";
(12,6)*+{Z}="4";
(-12,-6)*+{X'}="10";
(0,-6)*+{Y'}="12";
(12,-6)*+{Z'}="14";
{\ar^{x} "0";"2"};
{\ar^{y} "2";"4"};
{\ar_{f} "0";"10"};
{\ar^{g} "2";"12"};
{\ar_{x'} "10";"12"};
{\ar_{y'} "12";"14"};
{\ar@{}|{} "0";"12"};
{\ar@{}|\circlearrowright "0";"12"};
\endxy 
\]
in~$\cat$, there exists a morphism~$(f,h):\delta\to\delta'$ satisfying~$h\circ y=y'\circ g$.
\item[{\rm (ET3)$^{\mathrm{op}}$}] Dual of {\rm (ET3)}.
\item[{\rm (ET4)}] Let~$\delta\in\mathbb{E}(Z',X)$ and~$\delta'\in\mathbb{E}(X',Y)$ be~$\mathbb{E}$-extensions realized respectively by
\[
X\overset{f}{\longrightarrow}Y\overset{f'}{\longrightarrow}Z' \quad\text{and}\quad Y\overset{g}{\longrightarrow}Z\overset{g'}{\longrightarrow}X'.
\]
Then there exist an object~$Y'\in\cat$, a commutative diagram in~$\cat$
\[
\xy
(-18,6)*+{X}="0";
(-6,6)*+{Y}="2";
(6,6)*+{Z'}="4";
(18,6)="6";
(-18,-6)*+{X}="10";
(-6,-6)*+{Z}="12";
(6,-6)*+{Y'}="14";
(18,-6)="16";
(-6,-18)*+{X'}="22";
(6,-18)*+{X'}="24";
(-6,-28)="32";
(6,-28)="34";
{\ar^{f} "0";"2"};
{\ar^{f'} "2";"4"};
{\ar@{=} "0";"10"};
{\ar_{g} "2";"12"};
{\ar^{d} "4";"14"};
{\ar^{h} "10";"12"};
{\ar^{h'} "12";"14"};
{\ar_{g'} "12";"22"};
{\ar^{e} "14";"24"};
{\ar@{=} "22";"24"};
{\ar@{}|{} "0";"12"};
{\ar@{}|{} "2";"14"};
{\ar@{}|{} "12";"24"};
{\ar@{-->}^\delta "4";"6"};
{\ar@{-->}^{\delta''} "14";"16"};
{\ar@{-->}^{\delta'} "22";"32"};
{\ar@{-->}^{f'_\ast \delta'} "24";"34"};
{\ar@{}|\circlearrowright "0";"12"};
{\ar@{}|\circlearrowright "2";"14"};
{\ar@{}|\circlearrowright "12";"24"};
\endxy
\]
and an~$\mathbb{E}$-extension~$\delta''\in\mathbb{E}(Y',X)$ realized by~$X\overset{h}{\longrightarrow}Z\overset{h'}{\longrightarrow}Y'$, which satisfy the following compatibilities.
  \begin{enumerate}[(i)]
    \item $Z'\overset{d}{\longrightarrow}Y'\overset{e}{\longrightarrow}X'$ realizes~$f'_\ast\delta'$,
    \item $d^\ast\delta''=\delta$,
    \item $f_\ast\delta''=e^\ast\delta'$. 
  \end{enumerate}
\item[{\rm (ET4)$^{\mathrm{op}}$}] Dual of {\rm (ET4)}.
\end{itemize}
\end{definition}

\enlargethispage{-.8cm}
We use the following terminology.

\begin{notation}
Let~$(\cat,\mathbb{E},\mathfrak{s})$ be an extriangulated category.
\begin{enumerate}
\item A sequence~$X\xrightarrow{x}Y\xrightarrow{y}Z$ is called a conflation if it realizes some~$\mathbb{E}$-extension in~$\mathbb{E}(Z,X)$. In which case the morphism~$X\xrightarrow{x}Y$ is called an inflation, written~$X\infl Y$, and the morphism~$Y\xrightarrow{y}Z$ is called a deflation, witten~$Y\defl Z$.
\item An extriangle is a diagram~$X\overset{x}{\infl} Y\overset{y}{\defl} Z\overset{\delta}{\dashrightarrow}$ where~$X\overset{x}{\infl} Y\overset{y}{\defl} Z$ is a conflation realizing the~$\mathbb{E}$-extension~$\delta\in\mathbb{E}(Z,X)$.
\item Similarly, we call morphism of extriangles any diagram
\[
\xy
(-12,6)*+{X}="0";
(0,6)*+{Y}="2";
(12,6)*+{Z}="4";
(24,6)*+{}="6";
(-12,-6)*+{X'}="10";
(0,-6)*+{Y'}="12";
(12,-6)*+{Z'}="14";
(24,-6)*+{}="16";
{\ar^{x} "0";"2"};
{\ar^{y} "2";"4"};
{\ar@{-->}^{\delta} "4";"6"};
{\ar_{f} "0";"10"};
{\ar^{g} "2";"12"};
{\ar^{h} "4";"14"};
{\ar_{x'} "10";"12"};
{\ar_{y'} "12";"14"};
{\ar@{-->}_{\delta'} "14";"16"};
{\ar@{}|{} "0";"12"};
{\ar@{}|{} "2";"14"};
\endxy
\]
where~$(f,h):\delta\to\delta'$ is a morphism of~$\mathbb{E}$-extensions realized by~$(f,g,h)$.
\end{enumerate}
\end{notation}

The axioms above ensure that any extriangle induce long-ish exact sequences after application of some covariant or contravariant Hom-functor.
In particular in any conflation, the inflation is a weak kernel of the deflation, and the deflation is a weak cokernel of the inflation.

\begin{proposition}[{\cite[Props.~3.3~\&~3.11]{NakaokaPalu}}]
\label{prop:extricat long exact sequences}
Assume that $(\cat,\mathbb{E},\mathfrak{s})$ is an extriangulated category, and let $X\xrightarrow{x}Y\xrightarrow{y}Z\overset{\delta}{\dashrightarrow}$ be an extriangle.
Then the following sequences of natural transformations are exact:
\[
\cat(Z,-)\overset{-\circ y}{\longrightarrow}\cat(Y,-)\overset{-\circ x}{\longrightarrow}\cat(X,-)\overset{\delta^\sharp}{\longrightarrow}\mathbb{E}(Z,-)\overset{y^\ast}{\longrightarrow}\mathbb{E}(Y,-)\overset{x^\ast}{\longrightarrow}\mathbb{E}(X,-),
\]
\[
\cat(-,X)\overset{x\circ-}{\longrightarrow}\cat(-,Y)\overset{y\circ-}{\longrightarrow}\cat(-,Z)\overset{\delta_\sharp}{\longrightarrow}\mathbb{E}(-,X)\overset{x_\ast}{\longrightarrow}\mathbb{E}(-,Y)\overset{y_\ast}{\longrightarrow}\mathbb{E}(-,Z),
\]
where~$\delta^\sharp(f)=f_\ast\delta$ and~$\delta_\sharp(g)=g^\ast\delta$.
\end{proposition}

Any variant of the axiom (ET4) that would hold in an extension-closed subcategory of a triangulated category by applying the octahedron axiom also holds in any extriangulated category.
See~\cite[Sect.~3.2]{NakaokaPalu} for more details.

\begin{definition}
An \defn{almost-split}~$\mathbb{E}$-extension~$\delta\in\mathbb{E}(Z,X)$ is a non-split~$\mathbb{E}$-extension such that:
\begin{itemize}
\item[{\rm (AS1)}] $f_\ast\delta=0$ for any non-section $f\in\cat(X,X')$.
\item[{\rm (AS2)}] $g^\ast\delta=0$ for any non-retraction $g\in\cat(Z',Z)$.
\end{itemize}
\end{definition}

\begin{definition}
A non-zero object $X\in\cat$ is said to be \defn{endo-local} if $\cat(X,X)$ is local.
\end{definition}

\begin{proposition}[{\cite[Prop.~2.5]{IyamaNakaokaPalu}}]
For any non-split $\mathbb{E}$-extension $\delta\in\mathbb{E}(Z,X)$, the following holds.
\begin{enumerate}
\item If $\delta$ satisfies {\rm (AS1)}, then $X$ is endo-local.
\item If $\delta$ satisfies {\rm (AS2)}, then $Z$ is endo-local.
\end{enumerate}
\end{proposition}

\begin{definition}
 An object~$P\in\cat$ is called \defn{projective} if, for any~$X\in\cat$, we have~$\mathbb{E}(P,X)=0$.
 Dually, an object~$I\in\cat$ is called \defn{injective} if for any~$X\in\cat$, we have~$\mathbb{E}(X,I)=0$.
\end{definition}

\begin{proposition}[{\cite[Prop.~8.1]{IyamaNakaokaPalu}}]
Let~$\Lambda$ be an Artin algebra and let~$K^{[-1,0]}(\proj\Lambda)$ be the full subcategory of the homotopy category~$K^b(\proj\Lambda)$ consisting of complexes concentrated in degrees~$-1$ and~$0$ (using cohomological conventions).
Then any endo-local non-projective object~$X\in\cat$ permits an almost-split conflation~$\tau X\infl E\defl X$ and any endo-local non-injective object~$Y\in\cat$ permits an almost-split conflation~$Y\infl E'\defl \tau^- Y$.
\end{proposition}


\subsection{Proofs of preliminary results and of \cref{thm:extricats}}
\label{sec:proofs extricats}

We let~$\cat$ and~$\tc$ satisfy the assumptions (1), (2) and (3) of \cref{setting:extricat}.
We will make it explicit when more assumptions are needed.

\enlargethispage{-.1cm}
We start by stating some immediate implications of the existence of conflations~$T\infl 0 \defl \susp T$, for~$T\in\tc$.

\begin{lemma}
\label{lem:first consequences}
The following holds:
\begin{enumerate}[(i)]
\item An object belongs to~$\tc$ if and only if it is projective in~$\cat$.
\item For any~$T,T'\in\tc$, we have~$\cat(T,\susp T')=0$.
\item For any~$X\in\cat$ and~$T\in\tc$,~$\cat(X,\susp T)\cong\mathbb{E}(X,T)$.
\item For any~$X\in\cat$ and~$T\in\tc$,~$\cat(T,X)\cong\mathbb{E}(\susp T,X)$.
\item The full subcategory~$\susp\tc$ is rigid in~$\cat$.
\end{enumerate}
\end{lemma}

\begin{proof}
\begin{enumerate}[(i)]
\item If~$X$ is projective in~$\cat$, then the deflation~$T_0^X\defl X$ splits and~$X$ is a summand of~$T_0^X$.
\item This follows from the fact that~$T$ is projective and~$0\defl\susp T'$ is a deflation.
\item and (iv) are obtained from the exact sequences of \cref{prop:extricat long exact sequences} associated with the extriangle~$T\infl 0\defl \susp T\dashrightarrow$.
\addtocounter{enumi}{1}
\item For any~$T,T'\in\tc$, we have~$\mathbb{E}(\susp T,\susp T') \cong \cat(T,\susp T')=0$.
\qedhere
\end{enumerate}
\end{proof}

\begin{lemma}
\label{lem:shifts of extriangles}
Let~$X\in\cat$ and let
\[
T_1^X\overset{f_X}{\infl} T_0^X\overset{g_X}{\defl} X \overset{\delta_X}{\dashrightarrow}, \quad T_0^X\infl 0\defl \susp T_0^X\overset{\delta_0}{\dashrightarrow}, \quad T_1^X\infl 0\defl \susp T_1^X\overset{\delta_1}{\dashrightarrow}
\]
be extriangles with~$T_0^X,T_1^X\in\tc$.
Then there are extriangles
\[
T_0^X\overset{g_X}{\infl} X \overset{h_X}{\defl} \susp T_1^X \overset{(f_X)_\ast\delta_1}{\dashrightarrow}, \quad X \overset{h_X}{\infl} \susp T_1^X \overset{\susp f_X}{\defl} \susp T_0^X \overset{(g_X)_\ast\delta_0}{\dashrightarrow}
\]
where~$h_X$ is the unique morphism in~$\cat(X,\susp T_1^X)$ satisfying~$(h_X)^\ast\delta_1=\delta_X$ and~$\susp f_X$ is the unique morphism in~$\cat(\susp T_1^X,\susp T_0^X)$ satisfying~$(\susp f_X)^\ast\delta_0 = (f_X)_\ast\delta_1$.
\end{lemma}

\begin{proof}
The first extriangle is obtained by applying the dual of~\cite[Prop.~3.15]{NakaokaPalu} to~$\delta_X$ and~$\delta_1$, and the second one is obtained similarly from~$\delta_0$ and ~$(f_X)_\ast \delta_1$.
\end{proof}

\begin{remark}
\label{rem:equivalence Sigma}
For any~$T\in\tc$, fix an extriangle~$T\infl 0\defl\susp T\overset{\eps_T}{\dashrightarrow}$.
As noticed in \cref{lem:shifts of extriangles}, for any morphism~$T\overset{f}{\to}T'$ in~$\tc$, there is a unique morphism~$\susp f\in\cat(\susp T,\susp T')$ such that~${f_\ast\eps_T=(\susp f)^\ast\eps_{T'}}$ (\ie such that~$(f,\susp f)$ is a morphism from~$\eps_T$ to~$\eps_{T'}$), as illustrated below:
\[
\xy
(-12,6)*+{T}="0";
(0,6)*+{0}="2";
(12,6)*+{\susp T}="4";
(24,6)*+{}="6";
(-12,-6)*+{T'}="10";
(0,-6)*+{0}="12";
(12,-6)*+{\susp T'}="14";
(24,-6)*+{}="16";
{\ar "0";"2"};
{\ar "2";"4"};
{\ar@{-->}^{\eps_T} "4";"6"};
{\ar_{f} "0";"10"};
{\ar "2";"12"};
{\ar^{\susp f} "4";"14"};
{\ar "10";"12"};
{\ar "12";"14"};
{\ar@{-->}_{\eps_{T'}} "14";"16"};
\endxy
\]
This gives an equivalence of additive categories
\[
 \susp: \tc \to \susp\tc
\]
where~$\susp\tc$ is the full subcategory of~$\cat$ whose objects are isomorphic to objects of the form~$\susp T$, for some~$T\in\tc$.
\end{remark}

\begin{lemma}
\label{lem: mphs in T are inflations}
Every morphism with domain in~$\tc$ is an inflation.
\end{lemma}

\begin{proof}
Let~$T \overset{t}{\to} X$ be a morphism with~$T \in \tc$.
Using the notation of \cref{rem:equivalence Sigma}, there is an extriangle~$T\infl 0\defl \susp T \overset{\eps_T}{\dashrightarrow}$.
Consider the extriangle~$X\overset{g}{\infl} Y \overset{h}{\defl} \susp T \overset{t_\ast\eps_T}{\dashrightarrow}$.
By~\cite[Prop.~3.17]{NakaokaPalu}, there is a commutative diagram of extriangles
\[
\xy
(-7,7)*+{T}="2";
(7,7)*+{0}="4";
(21,7)*+{\susp T}="6";
(35,7)*+{}="8";
(-7,-7)*+{X}="12";
(7,-7)*+{Y}="14";
(21,-7)*+{\susp T}="16";
(35,-7)*+{}="18";
(-7,-21)*+{Y}="22";
(7,-21)*+{Y}="24";
(-7,-32)*+{}="32";
(7,-32)*+{}="34";
{\ar@{->} "2";"4"};
{\ar@{->} "4";"6"};
{\ar@{-->}^{\eps_T} "6";"8"};
{\ar@{->}_{t'} "2";"12"};
{\ar "4";"14"};
{\ar@{=} "6";"16"};
{\ar^{g} "12";"14"};
{\ar^{h} "14";"16"};
{\ar@{-->}^{t_\ast\eps_T} "16";"18"};
{\ar_{g} "12";"22"};
{\ar^{1} "14";"24"};
{\ar@{=} "22";"24"};
{\ar@{-->}^{h^\ast\eps_T} "22";"32"};
{\ar@{-->} "24";"34"};
{\ar@{}|\circlearrowright "2";"14"};
{\ar@{}|\circlearrowright "4";"16"};
{\ar@{}|\circlearrowright "12";"24"};
\endxy
\]
satisfying~$(t')_\ast\eps_T=t_\ast\eps_T$.
We thus have~$t'=t$.
\end{proof}

\begin{remark}
It follows from the proof of \cref{lem: mphs in T are inflations} that conditions (1), (2), (3) in \cref{setting:extricat} are equivalent to conditions (1), (2), (3'), where
\[
{\rm (3')} \quad \cat = \tc\ast\susp\tc.
\]
\end{remark}

\begin{proof}[Proof of \cref{lem:finitely presented}]
Applying the functor~$F$ to the extriangle~$T_1^X\infl T_0^X\defl X\overset{\delta_X}{\dashrightarrow}$ induces an exact sequence~$FT_1^X \to FT_0^X \to FX \to \mathbb{E}(T,T_1^X)=0$, showing that~$FX$ is finitely presented.
\end{proof}

\begin{proof}[Proof of \cref{prop:KRextricat}]
Recall that~$F:\cat\to\modt$ denotes the functor sending an object~$X$ to the module~$\cat(-,X)|_\tc$.
By \cref{lem:first consequences}\,(v), the ideal~$(\susp\tc)$ of morphisms factoring through an object in~$\susp\tc$ is included in the kernel of~$F$.
Conversely, let~$X\xrightarrow{f}Y$ be such that~$Ff=0$.
By \cref{lem:shifts of extriangles}, there is a conflation~$T_0\overset{g}{\infl} X\overset{h}{\defl} \susp T_1$ with~$T_0,T_1\in\tc$.
By assumption, the composition~$f\circ g$ vanishes, which implies that~$f$ factors through~$\susp T_1$.
The functor~$F$ thus induces a faithful functor, still denoted~$F$, from~$\cat/(\susp\tc)$ to~$\modt$.
It is well-known (see for instance~\cite[Lem.~3.1]{AssemSimsonSkowronski}) that~$F$ restricts to an equivalence of categories~$\tc\to\proj\tc$.
Because~$\tc$ consists of projective objects in~$\cat$, the image under~$F$ of every conflation is a right exact sequence.
The fact that~$F$ is full and dense can thus be proven as in the case of triangulated categories, by using \cref{lem: mphs in T are inflations}.
\end{proof}

\begin{corollary}
\label{coro:Sigma T injectives}
The objects in~$\susp\tc$ are precisely the injective objects in~$\cat$.
\end{corollary}

\begin{proof}
First note that any injective object belongs to~$\susp \tc$, by \cref{lem:shifts of extriangles}.
Indeed, if~$X$ is injective, then the extriangle
\[
X \overset{h_X}{\infl} \susp T_1^X \overset{\susp f_X}{\defl} \susp T_0^X \overset{(g_X)_\ast\delta_0}{\dashrightarrow}
\]
splits, showing that~$X$ is a summand of~$\susp T_1^X$.
Let us prove that the converse holds.
Let~$\susp T \infl X \overset{f}{\defl} Y\overset{\delta}{\dashrightarrow}$ be an extriangle with~$T\in\tc$.
Since~$\cat(-,\susp T)|_\tc = 0$ and
$\mathbb{E}(-,\susp T)|_\tc =0$ (by \cref{lem:first consequences}),
the morphism~$\cat(-,f)|_\tc$ is an isomorphism in~$\modt$.
By \cref{prop:KRextricat}, there are morphisms~$Y\overset{s}{\to} X$,~$Y\xrightarrow{h}\susp T'$,~$\susp T'\xrightarrow{t} Y$, for some~$T'\in\tc$ satisfying
\[
fs = id_Y + th.
\]
By \cref{lem:first consequences}\,(v), $\susp\tc$ is a rigid subcategory, so that~$t^\ast\delta=0$.
Equivalently, there is a morphism~$\susp T'\xrightarrow{k} X$ such that~$t=fk$.
We thus have~$f(s-kh)=id_Y$, showing that the extriangle splits.
\end{proof}

\begin{lemma}
\label{lem:index well-defined extricat}
 The index is well-defined. More precisely, for any object~$X\in\cat$ and any two conflations~$T_1\overset{f}{\infl} T_0 \overset{g}{\defl} X$,~$T'_1\overset{f'}{\infl} T'_0 \overset{g'}{\defl} X$ with~$T_0,T_1,T'_0,T'_1\in\tc$, we have~$[T_0]-[T_1]=[T'_0]-[T'_1]$ in~$\kzero{\tc}$.
\end{lemma}

\begin{proof}
Applying~\cite[Prop.~3.15]{NakaokaPalu} gives a commutative diagram made of conflations
\[
\xy
(-7,21)*+{T_1}="-12";
(7,21)*+{T_1}="-14";
(-21,7)*+{T'_1}="0";
(-7,7)*+{Y}="2";
(7,7)*+{T_0}="4";
(-21,-7)*+{T'_1}="10";
(-7,-7)*+{T'_0}="12";
(7,-7)*+{X}="14";
{\ar@{=} "-12";"-14"};
{\ar "-12";"2"};
{\ar "-14";"4"};
{\ar "0";"2"};
{\ar "2";"4"};
{\ar@{=} "0";"10"};
{\ar "2";"12"};
{\ar "4";"14"};
{\ar "10";"12"};
{\ar "12";"14"};
{\ar@{}|\circlearrowright "-12";"4"};
{\ar@{}|\circlearrowright "0";"12"};
{\ar@{}|\circlearrowright "2";"14"};
\endxy
\]
Since the objects in~$\tc$ are projective, the two conflations with middle term~$Y$ split and there are isomorphisms~$T'_1\oplus T_0 \cong Y \cong T_1\oplus T'_0$.
The equality~$[T'_1]+[T_0]=[T_1]+[T'_0]$ thus holds in~$\kzero{\tc}$.
\end{proof}

\begin{remark}
We note that \cref{lem:index well-defined extricat} only requires the subcategory~$\tc$ to be rigid, and not made of projective objects.
However, this stronger assumption is used in the proof of \cref{lem:index kzero extricat} below.
\end{remark}

\begin{lemma}
\label{lem:index kzero extricat}
The index descends to~$\kzero{\cat}$: for any conflation~$X\infl Y\defl Z$ in~$\cat$, we have~$\ind_\tc Y = \ind_\tc X + \ind_\tc Z$ in~$\kzero{\tc}$.
\end{lemma}

\begin{proof}
Let~$X\infl Y\defl Z$ be a conflation in~$\cat$.
Fix conflations~$T_1^X \infl T_0^X \defl X$,~$T_1^Z\infl T_0^Z \defl Z$ with~$T_0^X, T_1^X, T_0^Z, T_1^Z\in\tc$.
By~\cite[Prop.~3.15]{NakaokaPalu} and axiom (ET4)$^{\mathrm{op}}$, there are commutative diagrams of conflations
\[
\xy
(-7,21)*+{T^Z_1}="-12";
(10,21)*+{T^Z_1}="-14";
(-24,7)*+{X}="0";
(-7,7)*+{X\oplus T^Z_0}="2";
(10,7)*+{T^Z_0}="4";
(-24,-7)*+{X}="10";
(-7,-7)*+{Y}="12";
(10,-7)*+{Z}="14";
{\ar@{=} "-12";"-14"};
{\ar "-12";"2"};
{\ar "-14";"4"};
{\ar "0";"2"};
{\ar "2";"4"};
{\ar@{=} "0";"10"};
{\ar "2";"12"};
{\ar "4";"14"};
{\ar "10";"12"};
{\ar "12";"14"};
{\ar@{}|\circlearrowright "-12";"4"};
{\ar@{}|\circlearrowright "0";"12"};
{\ar@{}|\circlearrowright "2";"14"};
\endxy
\quad \quad \ \
\xy
(-24,21)*+{T^X_1}="0";
(-7,21)*+{B}="2";
(14,21)*+{T^Z_1}="4";
(-24,7)*+{T^X_1}="10";
(-7,7)*+{T^X_0\oplus T^Z_0}="12";
(14,7)*+{X\oplus T^Z_0}="14";
(-7,-7)*+{Y}="22";
(14,-7)*+{Y}="24";
{\ar "0";"2"};
{\ar "2";"4"};
{\ar@{=} "0";"10"};
{\ar "2";"12"};
{\ar "4";"14"};
{\ar "10";"12"};
{\ar "12";"14"};
{\ar "12";"22"};
{\ar "14";"24"};
{\ar@{=} "22";"24"};
{\ar@{}|\circlearrowright "0";"12"};
{\ar@{}|\circlearrowright "2";"14"};
{\ar@{}|\circlearrowright "12";"24"};
\endxy
\]
where we have used the fact that~$T_0^Z$ is projective in~$\cat$.
Moreover, the conflation with middle term~$B$ splits, which implies that~$B$ is isomorphic to~$T_1^X\oplus T_1^Z$.
\end{proof}

\begin{proof}[Proof of \cref{prop:index iso extricat}]
By \cref{lem:index well-defined extricat} and \cref{lem:index kzero extricat}, the index induces an well-defined morphism of abelian groups~$\ind_\tc:\kzero{\cat}\to\kzero{\tc}$.
It remains to prove that it is an isomorphism.
Let~$\pi:\kzero{\tc}\to\kzero{\cat}$ send a class~$[T]\in\kzero{\tc}$ to the class~$[T]\in\kzero{\cat}$.
The existence, for any~$T\in\tc$, of a conflation~$0\infl T\defl T$ shows that~$\pi$ is right inverse to~$\ind_\tc$.
To see that it is also a left inverse, note that if~$T_1\infl T_0\defl X$ is a conflation in~$\cat$, then the equality~$[X]=[T_0]-[T_1]$ holds in~$\kzero{\cat}$.
\end{proof}

\begin{proof}[Proof of \cref{thm:extricats}]
The proof is now similar to the proof of \cref{thm:relations-g-vecteurs}, by using \cref{prop:KRextricat} instead of \cref{prop:functor-F}, \cref{prop:index iso extricat} instead of \cref{coro:grothendieck-g-vectors}, and the fact that~$\tc$ is made of projective objects instead of \cref{lem:almost-split}.
\end{proof}

\begin{proof}[Proof of \cref{coro:meshes positively generate extricats}]
The proof is similar to that of \cref{coro:meshes positively generate cluster cats}.
\end{proof}


\subsection{Application to 2-term complexes of projectives and to gentle algebras}
\label{sec:Kbproj}

Let~$\Lambda$ be an Artin algebra.
Then the homotopy category~$K^b(\proj \Lambda)$ is a triangulated category, and we consider its full subcategory~$K^{[-1,0]}(\proj \Lambda)$ of complexes concentrated in degrees~$-1$ and~$0$ (with cohomological conventions).

\begin{proposition}
\label{prop:2kbproj-is-extriangulated}
Let~$\Lambda$ be any Artin algebra. Then the category~$K^{[-1,0]}(\proj \Lambda)$ is an extriangulated category satisfying the assumptions (1) to (3) of \cref{setting:extricat}, wiht~$\tc = \add(\Lambda)$. If, moreover,~$\Lambda$ is finite-dimensional over a field~$\field$, then~$K^{[-1,0]}(\proj \Lambda)$ satisfies all assumptions of \cref{setting:extricat}. 
\end{proposition}

\begin{proof}
The category~$K^{[-1,0]}(\proj \Lambda)$ is closed under extensions in~$K^b(\proj \Lambda)$, so it is an extriangulated category whose conflations are induced by triangles in~$K^b(\proj\Lambda)$ by~\cite[Rem.~2.18]{NakaokaPalu}. Now, the object~$\Lambda$ is projective in~$K^{[-1,0]}(\proj \Lambda)$, since if~$X\infl Y\defl Z$ is a conflation, then we get a long exact sequence
\[
\Hom{K^b}(\Lambda,Y)\to \Hom{K^b}(\Lambda,Z) \to  \Hom{K^b}(\Lambda, \susp X),
\]
and the last term vanishes, since~$\susp X$ is a complex concentrated in degrees~$-2$ and~$-1$, while~$\Lambda$ is in degree~$0$. This proves that~$\Lambda$ is projective.
 
The morphism~$\Lambda \to 0$ is an inflation, since~$\Lambda\infl 0 \defl \susp \Lambda$ is a conflation. Moreover, if~$X= P_1\xrightarrow{f}P_0$ is an object of~$K^{[-1,0]}(\proj\Lambda)$, then it lies in a conflation~$P_1\stackrel{f}{\infl} P_0 \defl X$ by definition. We have proved that properties~(1) to (3) of \cref{setting:extricat} are satisfied.
 
Finally, assume that~$\Lambda$ is finite-dimensional over a field~$\field$. Then~$K^{[-1,0]}(\proj\Lambda)$ is Krull--Schmidt,~$\field$-linear and Ext-finite, since this is true for~$K^b(\proj\Lambda)$ for any finite-dimensional algebra~$\Lambda$. Moreover, it has Auslander--Reiten--Serre duality by~\cite[Prop.~6.1]{IyamaNakaokaPalu}.
\end{proof}

The extriangulated category~$K^{[-1,0]}(\proj \Lambda)$ is the setting in which we can finally prove point~(ii) of \cref{prop:exchangeablePairsNKC}, using results of~\cite{DemonetIyamaJasso} on~$\b{g}$-vectors. 
Before doing so, we introduce the notion of mutation conflation for two-term silting objects.

\begin{definition}
Recall that a two-term silting object is a complex~$T$ in~$K^{[-1,0]}(\proj \Lambda)$ such that~$\Hom{K^b}(T,\susp T)=0$ and the number of isomorphism classes of indecomposable summands of~$T$ is the same as that of~$\Lambda$.
\end{definition}

\begin{remark}
The definition given above is only equivalent to the usual definition of a silting object for two-term complexes of projectives.
\end{remark}

\begin{definition}
\label{def:mutation conflation}
A conflation $X\infl E\defl Y$ of~$K^{[-1,0]}(\proj \Lambda)$ is called a \defn{mutation conflation} if there are basic, two-term, silting objects $X\oplus R$, $Y\oplus R$, with~$X$ and~$Y$ indecomposable, such that the inflation $X\infl E$ is a left ($\add R$)-approximation (\ie any morphism from~$X$ to an object in~$\add R$ factors through~$X\infl E$).
\end{definition}

\begin{remark}
In \cref{def:mutation conflation}, the requirement that~$X$ and~$Y$ are indecomposable implies the map~$E\defl Y$ is a right~($\add R$)-approximation, and that both approximations are minimal.
\end{remark}

\begin{proposition}
\label{prop:exchange conflations are unique}
Let~$\Lambda$ be a finite-dimensional~$\field$-algebra, where~$\field$ is a field. Let~$X$ and~$Y$ be objects of~$K^{[-1,0]}(\proj \Lambda)$. 
\begin{enumerate}
\item If there is a mutation conflation (see \cref{def:mutation conflation}) of the form~$X\infl E\defl Y$, then there can be no mutation conflation of the form~$Y\infl E'\defl X$.
\item Assume that~$X\infl E\defl Y$ and~$X\infl E'\defl Y$ are two mutation conflations. 
Then~$E$ and~$E'$ are isomorphic. 
\end{enumerate} 
\end{proposition}

\begin{proof}
We first prove~(1). Assume that~$X\infl E\defl Y$ and~$Y\infl E'\defl X$ are two mutation conflations. Applying the functor~$\Hom{K^b}(\susp^{-1} Y, -)$ to the second sequence, we get an exact sequence of abelian groups
\[
\Hom{K^b}(\susp^{-1} Y, E') \to \Hom{K^b}(\susp^{-1} Y, X) \to \Hom{K^b}(\susp^{-1} Y, \susp Y).
\]
By definition of a mutation conflation, we have that~$\Hom{K^b}(\susp^{-1} Y, E') = 0$. Moreover, since~$Y$ is concentrated in homological degrees~$-1$ and~$0$, we get that~$\Hom{K^b}(\susp^{-1} Y, \susp Y) = 0$. Therefore,~$\Hom{K^b}(\susp^{-1} Y, X) = 0$. 
However, since the conflation~$X\infl E\defl Y$ is not split, we have that~$\Hom{K^b}(\susp^{-1} Y, X) \neq 0$, a contradiction. This proves~(1). 

\smallskip
We now prove~(2). Let~$X\infl E\defl Y$ and~$X\infl E'\defl Y$ be two mutation conflations. By~\cite[Thm.~6.5]{DemonetIyamaJasso} (based on~\cite[Thm~2.3]{DehyKeller}), rigid objects in~$K^{[-1,0]}(\proj \Lambda)$ are determined by their~$\b{g}$-vector. The datum of the~$\b{g}$-vector of an object of~$K^{[-1,0]}(\proj \Lambda)$ is equivalent to that of its class in the Grothendieck group~$K_0 \big( K^{[-1,0]}(\proj \Lambda) \big)$, where~$K^{[-1,0]}(\proj \Lambda)$ is viewed as an extriangulated category (see \cref{prop:2kbproj-is-extriangulated}). In this Grothendieck group, we have that~$[E] = [E'] = [X]+[Y]$. Thus~$E$ and~$E'$ are isomorphic.
\end{proof}

\begin{corollary}
\label{coro:exchangeablePairsNKC}
Point~(ii) of \cref{prop:exchangeablePairsNKC} holds.
\end{corollary}

\begin{proof}
Let~$\omega$ and~$\omega'$ be in facets~$F$ and~$F'$ as in the statement of \cref{prop:exchangeablePairsNKC}. For every walk~$\omega''$, let~$P(\omega'')$ be the object of~$K^{[-1,0]}(\proj \Lambda)$ corresponding to~$\omega''$. Then the walks~$\mu$ and~$\nu$ defined in the same statement yield a mutation conflation
\[
P(\omega) \infl P(\mu) \oplus P(\nu) \defl P(\omega') \quad \textrm{or} \quad P(\omega') \infl P(\mu) \oplus P(\nu) \defl P(\omega).
\]
Assume that there are other facets~$F_0$ and~$F_0'$ also satisfying the hypotheses of \cref{prop:exchangeablePairsNKC}. Then they give rise to walks~$\mu'$ and~$\nu'$ and an exchange conflation
\[
P(\omega) \infl P(\mu') \oplus P(\nu') \defl P(\omega') \quad \textrm{or} \quad P(\omega') \infl P(\mu') \oplus P(\nu') \defl P(\omega).
\]
By \cref{prop:exchange conflations are unique}, we get that~$P(\mu)\oplus P(\nu)$ and~$P(\mu')\oplus P(\nu')$ are isomorphic, and that the mutation conflations either both have the form~$P(\omega) \infl P(\mu) \oplus P(\nu) \defl P(\omega')$ or both have the form~$P(\omega') \infl P(\mu) \oplus P(\nu) \defl P(\omega)$. Without loss of generality, assume that they have the first form.
 
In particular, the leftmost morphisms of both mutation conflations are minimal ${\add \big( P(\mu) \oplus P(\nu) \big)}$-approximations of~$P(\omega)$. As such, they are isomorphic as morphisms. They determine common substrings of~$\omega$ and~$\mu$ and of~$\omega$ and~$\nu$, which in turn determine~$\sigma$. Point (ii) of \cref{prop:exchangeablePairsNKC} is proved.
\end{proof}

We now come to the main result of this section.

\begin{theorem}
\label{thm:brick-algebra-condition}
Let~$\Lambda$ be an Artin algebra all of whose indecomposable objects are ~$\tau$-rigid bricks. Then the almost-split conflations of~$K^{[-1,0]}(\proj \Lambda)$ are mutation conflations if and only if for any non-projective indecomposable~$\Lambda$-module~$M$, the space~$\Hom{\Lambda}(M, \tau^2 M)$ vanishes.
\end{theorem}

\begin{remark}
Assume that~$\Lambda$ is a finite-dimensional brick algebra which is representation-finite.
Then~\cite[Thm.~6.2]{DemonetIyamaJasso} implies that the assumptions of \cref{thm:brick-algebra-condition} are satisfied.
\end{remark}

\begin{corollary}
\label{coro:silting simplicial}
Let~$\Lambda$ be a finite dimensional brick algebra of finite representation type.
Assume moreover that, for any indecomposable~$\Lambda$-module~$M$, the space~$\Hom{\Lambda}(M, \tau^2 M)$ vanishes.
Then the support~$\tau$-tilting fan of~$\Lambda$ has the unique exchange relation property and its type cone is simplicial.
\end{corollary}

As a consequence of \cref{coro:silting simplicial}, the method of \cref{part:geometry} apply and give an explicit description of all realizations of the support~$\tau$-tilting fan of~$\Lambda$.
Another consequence is an algebraic proof of \cref{coro:simplicialTypeConeNKC}, that we restate below.

\begin{corollary}[\cref{coro:simplicialTypeConeNKC}]
\label{coro:simplicialTypeConeNKCalgebraic}
For any brick and $2$-acyclic gentle algebra, the type cone of its support~$\tau$-tilting fan is simplicial.
\end{corollary}

The remaining part of this section is devoted to proving \cref{thm:brick-algebra-condition,coro:simplicialTypeConeNKCalgebraic}.

\begin{lemma}
\label{lem:AR-conflations-and-sequences}
Let~$\Lambda$ be any Artin algebra. Let~$X\overset{a}{\infl} Y \overset{b}{\defl} Z$ be an almost-split conflation in~$K^{[-1,0]}(\proj \Lambda)$ such that~$Z$ is not in~$\add(\susp \Lambda)$. Denote by~$F$ the functor~$\Hom{K^b}(\Lambda, -):K^{[-1,0]}(\proj \Lambda) \to \MOD \Lambda$. Then there is an almost-split sequence in~$\MOD\Lambda$:
\[
0\to FX\to FY \to FZ \to 0
\]
\end{lemma}

\begin{proof}
The almost-split conflation induces an exact sequence
\[
F\susp^{-1}Y\to F\susp^{-1}Z \to FX \to FY\to FZ\to F\susp X.
\]
Because~$X$ is concentrated in non-positive degrees,~$F\susp X=0$.
Let us verify that the map ${F\susp^{-1}Y\to F\susp^{-1}Z}$ is surjective.
Let~$\Lambda \xrightarrow{f} \susp^{-1} Z$ be a morphism in~$K^b(\proj \Lambda)$.
Since~$Y\overset{b}{\to} Z$ is right almost-split, and~$Z$ is indecomposable not in~$\add(\susp\Lambda)$, the morphism~$\susp f$ factors through~$b$.
This implies the required surjectivity, and we have a short exact sequence
\[
 0 \to FX \xrightarrow{Fa} FY\xrightarrow{Fb} FZ\to 0.
\]
Because~$\MOD\Lambda$ and~$K^b(\proj \Lambda)$ are Krull--Schmidt, it is immediate to check that~$Fb$ (resp.~$Fa$) inherits from~$b$ (resp.~$a$) the property of being right (resp. left) almost-split.
\end{proof}

\begin{lemma}
\label{lem:then-AR-seq-are-mutation}
Let~$\Lambda$ be an Artin algebra, let~$\tau X\infl E\defl X$ be an Auslander--Reiten conflation in~$K^{[-1,0]}(\proj\Lambda)$ and denote~$FX=M$.
\begin{enumerate}
\item $X\notin\add(\susp\Lambda)$;
\item $M$ and~$\tau M$ are~$\tau$-rigid;
\item $\Hom{\Lambda}(M,\tau^2 M)=0$;
\item $\tau M$ is a brick.
\end{enumerate}
Then~$\tau X\infl E\defl X$ is a mutation conflation.
\end{lemma}

\begin{proof}
First note that the results~\cite[Lems.~3.4~\&~3.5]{AdachiIyamaReiten}, even though stated for finite dimensional algebras, hold with the same proofs for any Artin algebra.
For convenience of the reader, we recall those two results: Let~$Y, Z \in K^{[-1,0]}(\proj\Lambda)$ be indecomposable and not in~$\add(\susp\Lambda)$, and let~$P\in\proj\Lambda$. Then~$\Hom{K^b}(Y,\susp Z)=0$ if and only if~$\Hom{\Lambda}(FZ,\tau FY)=0$ and~${\Hom{K^b}(P,Y)=0}$ if and only if~$\Hom{\Lambda}(P,FY)=0$.

Note also that \cref{lem:AR-conflations-and-sequences} implies that~$F\tau X$ is isomorphic to $\tau M$.

By assumption,~$X$ is not in~$\add(\susp\Lambda)$.
The same is true of~$\tau X$ since objects in~$\add(\susp\Lambda)$ are injective in~$K^{[-1,0]}(\proj\Lambda)$ and an Auslander--Reiten conflation does not split.
It thus follows from~\cite[Lem.~3.4 ]{AdachiIyamaReiten} that~$X$ and~$\tau X$ are rigid in $K^{[-1,0]}(\proj\Lambda)$.
We have to prove that both~$E\oplus X$ and~$E\oplus\tau X$ are rigid.

\vspace{7pt}
\noindent (1) $\Hom{K^b}(E,\susp X) = 0$.
\vspace{4pt}

The Auslander--Reiten conflation induces an exact sequence
\[
\Hom{K^b}(X,\susp X) \to \Hom{K^b}(E,\susp X) \to \Hom{K^b}(\tau X,\susp X).
\]
As already remarked,~$X$ is rigid so that the left-most term vanishes.
By assumption
\[
\Hom{\Lambda}(FX,\tau F\tau X) = \Hom{\Lambda}(M,\tau^2 M)=0,
\]
so that the right-most term also vanishes by~\cite[Lem.~3.4 ]{AdachiIyamaReiten}.

\vspace{7pt}
\noindent (2) $\Hom{K^b}(\tau X,\susp E) = 0$.
\vspace{4pt}

Since~$\Lambda$ is projective in~$K^{[-1,0]}(\proj\Lambda)$, the sequence~$\tau M \to FE \to M$ is short exact.
It induces a sequence
\[
\Hom{\Lambda}(M,\tau^2 M) \to \Hom{\Lambda}(FE,\tau^2 M) \to \Hom{\Lambda}(\tau M,\tau^2 M)
\]
from which we deduce that~$\Hom{\Lambda}(FE,\tau^2 M)=0$.
This implies~$\Hom{K^b}(\tau X,\susp E) = 0$ by~\cite[Lem.~3.4 ]{AdachiIyamaReiten}.

\vspace{7pt}
\noindent (3) $\Hom{K^b}(X,\susp E) = 0$.
\vspace{4pt}

Let us show that~$\Hom{\Lambda}(FE,\tau M)=0$.
Let~$FE \xrightarrow{f} \tau M$ be any morphism, and consider the short exact sequence~$0\to \tau M \xrightarrow{a} FE \xrightarrow{b} M \to 0$ induced by the Auslander--Reiten conflation.
Since~$\tau M$ is a brick, the composition~$fa$ is zero.
The morphism~$f$ then factors through the cokernel~$b$, which implies~$f=0$ since~$M$ is~$\tau$-rigid.

\vspace{7pt}
\noindent (4) $\Hom{K^b}(E,\susp E) = 0$.
\vspace{4pt}

The Auslander--Reiten conflation induces an exact sequence
\[
\Hom{K^b}(X,\susp E) \to \Hom{K^b}(E,\susp E) \to \Hom{K^b}(\tau X,\susp E).
\]
We can conclude by (2) and (3).

\vspace{7pt}
\noindent (5) $\Hom{K^b}(E,\susp \tau X) = 0$.
\vspace{4pt}

We have an exact sequence
\[
\Hom{K^b}(E,E) \to \Hom{K^b}(E,X) \to \Hom{K^b}(E,\susp \tau X) \to \Hom{K^b}(E,\susp E).
\]
The first map is surjective since~$E\to X$ is right almost-split (note that~$E$ cannot contain any summand isomorphic to~$X$).
We can conclude by (4).
\end{proof}

\begin{lemma}
\label{lem:if AR are mutations}
Let~$\Lambda$ be an Artin algebra, and let~$Y\infl E\defl X$ be a conflation in~$K^{[-1,0]}(\proj\Lambda)$. Then:
\begin{enumerate}
\item If~$\Hom{K^b}(Y,\susp E)=0$, we have~$\Hom{\Lambda}(FX,\tau FY)=0$.
\item If moreover~$X\notin\add(\susp\Lambda)$ and the conflation is almost-split, then~$\Hom{\Lambda}(FX,\tau^2 FX)=0$.
\end{enumerate}
\end{lemma}

\begin{proof}
By \cref{lem:AR-conflations-and-sequences}, \emph{(2)} follows from \emph{(1)}.
Let us prove \emph{(1)}.
The conflation induces an exact sequence~$\Hom{K^b}(Y,\susp E) \to \Hom{K^b}(Y,\susp X) \to \Hom{K^b}(Y,\susp^2,Y)$.
The complex~$Y$ being two-term, we have~$\Hom{K^b}(Y,\susp^2,Y)=0$.
Assuming that~$\Hom{K^b}(Y,\susp E)=0$, we deduce that~$\Hom{K^b}(Y,\susp X)=0$, and hence that~$\Hom{\Lambda}(FX,\tau FY) = 0$ by~\cite[Lem.~3.4]{AdachiIyamaReiten}.
\end{proof}

%

\begin{lemma}
\label{lem:initial meshes}
Let~$\Lambda$ be an Artin algebra whose indecomposable injective modules are~$\tau$-rigid, and let~$P$ be an indecomposable projective~$\Lambda$-module.
Then the Auslander--Reiten conflation~${Y\infl E\defl \susp P}$ in~$K^{[-1,0]}(\proj\Lambda)$ is a mutation conflation if~$FY$ is a brick.
\end{lemma}

\begin{proof}
In this proof, we implicitly identify the projective indecomposable~$\Lambda$-module~$P$ with the complex~$P$ concentrated in degree~$0$.

We check that the conflation~$Y\infl E\defl\susp P$ is a mutation conflation.

(0) Since~$P$ is concentrated in only one degree, it is rigid in~~$K^{[-1,0]}(\proj\Lambda)$ and so is~$\susp P$.
Moreover,~$FY$ is indecomposable and injective.
By assumption, it is~$\tau$-rigid, hence~$Y$ is rigid in~$K^{[-1,0]}(\proj\Lambda)$ by~\cite[Lem.~3.4]{AdachiIyamaReiten}.

(1) We have~$\Hom{K^b}(E,\susp^2 P)=0$ since~$E$ is concentrated in degrees~$-1$ and~$0$.

(2) $\Hom{K^b}(Y,\susp E)=0$: There is an exact sequence
\[
\Hom{K^b}(Y,\susp Y) \to \Hom{K^b}(Y,\susp E) \to \Hom{K^b}(Y,\susp^2 P),
\]
where the first term vanishes by (0) and the last term vanishes since~$Y$ is in degrees~$-1, 0$ while~$\susp^2 P$ is concentrated in degree~$-2$.

(3) $\Hom{K^b}(E,\susp Y) = 0$: Consider the exact sequence
\[
 \Hom{K^b}(\susp Y,\susp Y) \to \Hom{K^b}(\susp P,\susp Y) \to \Hom{K^b}(E,\susp Y) \to \Hom{K^b}(Y,\susp Y).
\]
Its last term vanishes by (0).
Let us prove that its first morphism is surjective.
Composition with~$\alpha$ and~$F\alpha$ yields a commutative square
\[
\xy
(-20,7)*+{\Hom{K^b}(Y,Y)}="0";
(20,7)*+{\Hom{K^b}(P,Y)}="2";
(-20,-7)*+{\Hom{\Lambda}(FY,FY)}="10";
(20,-7)*+{\Hom{\Lambda}(FP,FY)}="12";
{\ar@{->>} "0";"10"};
{\ar^{\alpha^\ast} "0";"2"};
{\ar@{->>} "2";"12"};
{\ar_{(F\alpha)^\ast} "10";"12"};
{\ar@{}|\circlearrowright "0";"12"};
\endxy
\]
where the vertical maps are surjective.
Since~$\Lambda$ is rigid and~$Y$ is two-term, we can apply~\cite[Prop.~6.2]{IyamaYoshino} to the triangulated category~$K^b(\proj\Lambda)$ (or \cref{prop:KRextricat} to the extriangulated $K^{[-1,0]}(\proj\Lambda)$).
As a consequence, the right-most vertical map is bijective.
Because~$FY$ is a brick,~$\End{\Lambda}(FY)$ is a division ring.
Moreover,~$FY$ is the indecomposable injective with simple socle the simple top of~$P$, so that the~$\End{\Lambda}(FY)$-module~$\Hom{\Lambda}(FP,FY)$ is one-dimensional over~$\End{\Lambda}(FY)$.
It follows that the vertical map~$(F\alpha)^\ast$ is either zero or bijective.
But~$\alpha$ is non-zero with domain in~$\add\Lambda$ so that~$F\alpha$ is non-zero ; and~$(F\alpha)^\ast(1_{FY})=F\alpha$ so that~$(F\alpha)^\ast$ is non-zero.
By commutativity, this implies the surjectivity of the top vertical morphism~$\alpha^\ast$.
Since the shift functor is an automorphism, the morphism~$\Hom{K^b}(\susp Y,\susp Y) \to \Hom{K^b}(\susp P,\susp Y)$ is also surjective, and~$\Hom{K^b}(E,\susp Y) = 0$.

(4) $\Hom{K^b}(E,\susp E)=0$: In the exact sequence
\[
 \Hom{K^b}(E,\susp Y) \to \Hom{K^b}(E, \susp E) \to \Hom{K^b}(E,\susp^2 P),
\]
the first term is zero by (3) and the last term vanishes since~$P$ is concentrated in degree zero.

(5) $\Hom{K^b}(\susp P,\susp E)= 0$:
Let~$P\xrightarrow{f} E$ be any morphism.
\[
\xy
(7,14)*+{P}="P";
(-21,0)*+{P}="0";
(-7,0)*+{Y}="2";
(7,0)*+{E}="4";
(21,0)*+{\susp P}="6";
(-21,-14)*+{Y}="Y";
(-7,-14)*+{E}="E";
{\ar^f "P";"4"};
{\ar^0 "P";"6"};
{\ar@{-->}_g "P";"2"};
{\ar^\alpha "0";"2"};
{\ar^\beta "2";"4"};
{\ar "4";"6"};
{\ar_g "0";"Y"};
{\ar@{-->}_h "2";"Y"};
{\ar_\beta "Y";"E"};
{\ar_{\beta h\!} "2";"E"};
{\ar@{-->}^{h'} "4";"E"};
\endxy
\]
Since~$P$ is rigid,~$f$ factors through the inflation~$\beta$ in the conflation~$Y\infl E\defl \susp P$, and there is some morphism~$g$ such that~$f=\beta g$.
By (3), the morphism~$g$ factors through~$P \xrightarrow{\alpha} Y$: there is some~$h:Y\to Y$ such that~$g= h \alpha$.
We then have~$f = \beta g = \beta (g-\alpha) = \beta (h-1) \alpha$.
Since~$Y$ is endo-local, the morphism~$h$ or~$h-1$ is an automorphism, and we may assume that~$h$ is an automorphism.
The morphism~$\beta h$ is not a section (otherwise~$\beta$ would be a section) and~$\beta$ is left-almost split: there is an endomorphism~$h'$ of~$E$ such that~$\beta h = h'\beta$.
Therefore, we have~$f=\beta h \alpha = h'\beta\alpha = 0$.
\end{proof}

\begin{proof}[Proof of \cref{thm:brick-algebra-condition}]
The theorem follows from the slightly more general statements of \cref{lem:then-AR-seq-are-mutation,lem:if AR are mutations,lem:initial meshes}.
\end{proof}


\begin{lemma}
\label{lem:characterization Gentle 2-acyclic}
Let~$\Lambda$ be a brick and $2$-acyclic gentle algebra.
Then, for any indecomposable~$\Lambda$-module~$M$ we have~$\Hom{\Lambda}(M,\tau^2 M)=0$.
\end{lemma}

\begin{proof}
Let~$\Lambda=\field Q / I$ be any gentle algebra of finite representation type, and let~$M$ be an indecomposable~$\Lambda$-module.
Write~$\sigma$ for the string corresponding to~$\tau M$.
Then the walks (\ie the maximal strings in the blossoming bound quiver of~$(Q,I)$) associated with~$M$ and~$\tau M$ are respectively~$\omega_M = \hh\sigma = p^{-1}\alpha\sigma\beta^{-1}q$ and~$\omega_{\tau M}=\cc\sigma=p'\alpha'^{-1}\sigma\beta'q'^{-1}$, where~$p,p',q,q'$ are paths and~$\alpha,\alpha',\beta,\beta'$ are arrows in the blossoming quiver.

Assume that~$\Hom{\Lambda}(M,\tau^2 M)\neq0$.
By~\cite[Theorem 2.46]{PaluPilaudPlamondon-nonkissing} or~\cite[Theorem 5.1]{BrustleDouvilleMousavandThomasYildirim}, the walk~$\omega_M$ kisses the walk~$\omega_{\tau M}$: there is a substring~$\rho$ of~$\omega_M$ and~$\omega_{\tau M}$ that is strictly on top of~$\omega_M$ and strictly at the bottom of~$\omega_{\tau M}$. Remark that~$\rho$ cannot start nor end at some endpoint of~$\sigma$ (because~$\rho$ is a top substring of~$\omega_M$).
 
If the substring~$\rho$ of~$\omega_M$ is contained in~$\sigma$, by the previous remark it is strictly below~$\sigma$. This implies that~$\omega_M$ kisses itself, hence that~$M$ is not~$\tau$-rigid. Similarly, if~$\rho$ is contained in~$\sigma$ as a substring of~$\omega_{\tau M}$, then~$\tau M$ is not~$\tau$-rigid. In both cases, the brick-$\tau$-rigid correspondence of~\cite{DemonetIyamaJasso} implies that~$\Lambda$ is not a brick algebra.
 
Note that~$\rho$ cannot contain~$\sigma$, because~$\sigma$ is the only common substring of~$\omega_M$ and~$\omega_{\tau M}$ that contains~$\sigma$.
Assume now that~$\rho$, viewed either as a substring of~$\omega_M$ or of~$\omega_{\tau M}$ is not contained in~$\sigma$.
Since~$\rho$ is strictly on top of~$\omega_M$, it contains the source of~$\alpha$ or the source of~$\beta$.
Similarly,~$\rho$ contains the target of~$\alpha'$ or the target of~$\beta'$.
Because the target of~$\alpha$ is the source of~$\alpha'$ (and the target of~$\beta$ is the source of~$\beta'$), this implies that~$Q$ contains either a non-oriented cycle with at most one relation, in which case~$\Lambda$ is not a brick algebra, or a~$2$-cycle, in which case~$\Lambda$ is not~$2$-acyclic.
\end{proof}

\begin{proof}[Proof of \cref{coro:simplicialTypeConeNKCalgebraic}]
The unique exchange relation property follows from \cref{prop:exchange conflations are unique} (2).
By \cref{lem:characterization Gentle 2-acyclic}, the assumptions of \cref{thm:brick-algebra-condition} are satisfied.
We can thus apply \cref{thm:extricats} and its corollary, which shows that there is the same number of extremal exchange pairs of the support~$\tau$-tilting fan (\ie the non-kissing fan) as there are meshes in the Auslander--Reiten quiver of~$\Lambda$.
Hence the type cone of the support~$\tau$-tilting fan is simplicial.
\end{proof}


\subsection{Application to other examples of categories}

In this section, we give two more examples of families of categories to which \cref{thm:extricats,coro:meshes positively generate extricats} apply.


\subsubsection{$2$-Calabi--Yau triangulated categories}
\label{sec:2CYTriangulated}

Let~$\cat$ be a triangulated category with shift functor~$\susp$.
We assume moreover that~$\cat$ is~$\field$-linear, Krull--Schmidt, Hom-finite, 2-Calabi--Yau, with a cluster-tilting object~$T$.
Then~$\cat$ satisfies all the assumptions of \cref{setting:extricat} for \cref{sec:Statment extricats} but, possibly, assumption (1).

We fix this issue by considering a relative extriangulated structure.
Let~$\mathbb{E}_T$ be the additive sub-bifunctor of~$\cat(-,\susp -)$ formed by those morphisms that factor through~$\add\susp T$.
Then~$(\cat,\mathbb{E}_T)$ becomes extriangulated by either~\cite[Prop.~3.16 or Prop.3.19]{HerschendLiuNakaokaI} (this can also be checked directly).
The extriangles of~$\cat$ are precisely the triangles~$X\overset{f}{\rightarrow}Y\overset{g}{\rightarrow} Z \overset{h}{\dashrightarrow}$, for which the morphism~$h$ factors through~$\add\susp T$.
By construction, Condition (1) is now satisfied.
Moreover, the triangles appearing in Conditions (2) and (3) are extriangles for this relative structure, which shows that those two conditions are still satisfied for the extriangulated structure~$(\cat,\mathbb{E}_T)$.

\begin{remark}
The index of an object~$X$ in the triangulated category~$\cat$, coincide with its class in the Grothendieck group of~$\cat$ for the relative extriangulated structure given by~$\mathbb{E}_T$.
\end{remark}

\begin{remark}
The main results of \cref{subsec:Statement cluster cats} can thus be viewed as specific cases of the results of \cref{sec:extricats}: \cref{thm:relations-g-vecteurs} and \cref{coro:meshes positively generate cluster cats} are consequences of \cref{thm:extricats} and \cref{coro:meshes positively generate extricats}.
\end{remark}

\subsubsection{Objects presented by a subcategory}
\label{sec:prT}

We now describe a setup satisfying the assumptions of \cref{setting:extricat} that, at the same time, subsumes all examples given so far, while being easily checked in practice.

Let~$(\ec,\mathbb{E},\mathfrak{s})$ be an extriangulated category (for example, a small exact category or a triangulated category), and let~$\tc$ be an essentially small subcategory of~$\ec$.

For any~$X,Y\in\ec$, let~$\mathbb{E}_\tc(X,Y)$ be the subset of~$\mathbb{E}(X,Y)$ consisting of those~$\delta$ that satisfy~$f^\ast\delta = 0$, for any~$T\in\tc$ and any~$T\xrightarrow{f}X$.
It is easily seen that~$\mathbb{E}_\tc$ is an additive subfunctor of~$\mathbb{E}$.
Consider $(\ec,\mathbb{E}_\tc,\mathfrak{s}_\tc)$ endowed with the resctriction of the additive realization~$\mathfrak{s}$.
Since the deflations in~$(\ec,\mathbb{E}_\tc,\mathfrak{s}_\tc)$ are precisely the deflations in~$(\ec,\mathbb{E},\mathfrak{s})$ that are~$\tc$-epic, they are closed under composition.
We can thus apply~\cite[Prop.~3.16]{HerschendLiuNakaokaI} to obtain that~$(\ec,\mathbb{E}_\tc,\mathfrak{s}_\tc)$ is an extriangulated category.
With that relative structure, the objects in~$\tc$ become projective
(alternatively, one can directly apply~\cite[Prop.~3.19]{HerschendLiuNakaokaI}).
Define~$\cat$ to be the full subcategory of~$\ec$ whose objects~$X$ admit an~$\mathfrak{s}_\tc$-conflation~$T_1\infl T_0\defl X$.

\begin{lemma}
\label{lem: prT is extriangulated}
The full subcategory~$\cat$ is extension-closed in~$(\ec,\mathbb{E}_\tc,\mathfrak{s}_\tc)$, and thus inherits an extriangulated structure.
\end{lemma}

\begin{proof}
This follows from the proof of \cref{lem:index kzero extricat}.
\end{proof}

\begin{notation}
Write~$\tc^\perp$ for the full subcategory of~$\ec$ whose objects are those objects~$X\in\ec$ that satisfy~$\ec(T,X)=0$, for any~$T\in\tc$.
\end{notation}

\begin{proposition}
\label{prop:prT}
Assume that, for any~$T\in\tc$, there is an~$\mathfrak{s}$-conflation~$T\infl 0\defl X_T$ in~$\ec$ with~$X_T\in\tc^\perp$.
Then the extriangulated category~$\cat$ satisfies the assumptions (1) to (3) of \cref{setting:extricat}.
\end{proposition}

\begin{proof}
For any~$T\in\tc$ the existence of an extriangle~$0\infl T\defl T\overset{0}{\dashrightarrow}$ shows that~$\tc$ is a full subcategory of~$\cat$.
By \cref{lem: prT is extriangulated}, the category~$\cat$ is extriangulated and, by definition of this extriangulated structure, the objects in~$\tc$ are projective in~$\cat$.
Assumption (3) is satisfied by the definition of~$\cat$, and assumption (2) is equivalent to that of the proposition.
\end{proof}

\begin{remark}
 The categories considered in \cref{sec:2CYTriangulated,sec:Kbproj} all fall into the setting of \cref{prop:prT}.
\end{remark}


\subsection{Examples}

We conclude with a couple of examples for two associative algebras.



\begin{example}[A gentle algebra]
Consider the algebra~$\Lambda$ given the the path algebra of the quiver
\[
\begin{tikzcd}
	1 \arrow[r, "\alpha"] & 2\arrow[r, "\beta"] & 3\arrow[r, "\gamma"] & 4
\end{tikzcd}
\]
modulo the relation~$\gamma\beta\alpha = 0$. The Auslander--Reiten quiver of the extriangulated category $K^{[-1, 0]}(\proj\Lambda)$ is depicted below.
\[
\hspace*{-1cm}
{\tiny
\begin{tikzcd}
	&&&& (0\to P_4)\arrow[ddr] && (P_2\to 0)\arrow[ddr] && \\
	&& (0\to P_3)\arrow[dr] && (P_1\to 0)\arrow[dr] &&&& \\
	& (0\to P_2)\arrow[dr]\arrow[ur] && (P_1\to P_3)\arrow[uur]\arrow[ur]\arrow[dr] && (P_2\to P_4)\arrow[uur]\arrow[dr] && (P_3\to 0)\arrow[dr] & \\
	(0\to P_1) \arrow[ur] && (P_1\to P_2)\arrow[ur] && (P_2\to P_3)\arrow[ur]  && (P_3\to P_4)\arrow[ur] && (P_4\to 0)
\end{tikzcd}
}
\]
The Auslander--Reiten quiver of the module category is obtained by taking the quotient by the ideal generated by the shifted projective modules.
\[
\hspace*{-1cm}
{\footnotesize
\begin{tikzcd}
    &&&& \begin{matrix} 4\\3\\2 \end{matrix}\arrow[ddr] && \\
    && \begin{matrix} 3\\2\\1 \end{matrix}\arrow[dr] &&&& \\
    & \begin{matrix} 2\\1 \end{matrix}\arrow[dr]\arrow[ur] && \begin{matrix} 3\\2 \end{matrix}\arrow[uur]\arrow[dr] && \begin{matrix} 4\\3 \end{matrix}\arrow[dr] & \\
    \begin{matrix} 1 \end{matrix} \arrow[ur] && \begin{matrix} 2 \end{matrix}\arrow[ur] && \begin{matrix} 3 \end{matrix}\arrow[ur] && \begin{matrix}4 \end{matrix}
\end{tikzcd}
}
\]
The conditions of \cref{thm:brick-algebra-condition} are satisfied.
\end{example}

\begin{example}[Preprojective algebra of type~$A_3$]\label{exm:preprojective silting objects}
Let~$\Lambda$ be the preprojective algebra of type~$A_3$: it is the path algebra of the quiver
\[
\begin{tikzcd}
	1\arrow[r, shift left = 2pt, "\alpha"] & 2\arrow[l, shift left = 2pt, "\bar \alpha"] \arrow[r, shift left = 2pt, "\beta"]& 3\arrow[l, shift left = 2pt, "\bar \beta"]
\end{tikzcd}
\] 
modulo the relations~$\alpha \bar \alpha = 0$, $\bar \beta \beta = 0$ and~$\bar \alpha \alpha + \beta \bar \beta = 0$. 
The Auslander--Reiten quiver of the extriangulated category~$K^{[-1,0]}(\proj \Lambda)$ is depicted below. Note that it is periodic.
 
\[
\hspace*{-2cm}
{\tiny
\begin{tikzcd}
    (P_1 \to P_3)\arrow[dr] && (P_3\to P_2)\arrow[r]\arrow[dr] & (0\to P_3)\arrow[r]& (P_1\to P_3)\arrow[r]\arrow[dr] & (P_1\to 0)\arrow[r] & (P_2\to P_1)\arrow[dr]&& (P_1\to P_2) \\
    (0\to P_2)\ar[r] & (P_2\to P_2)\arrow[ur]\arrow[dr]\arrow[r] & (P_2\to 0)\arrow[r] & (P_1\oplus P_3\to P_2)\arrow[ur]\arrow[dr] && (P_2\to P_1\oplus P_3)\arrow[ur]\arrow[dr]\arrow[r] & (0\to P_2)\arrow[r] & (P_2\to P_2)\arrow[ur]\arrow[dr]\arrow[r] & (P_2\to 0) \\
    (P_2\to P_1)\arrow[ur] && (P_1\to P_2)\arrow[ur]\arrow[r] & (0\to P_1)\arrow[r] & (P_3\to P_1)\arrow[r]\arrow[ur] & (P_3\to 0)\arrow[r] & (P_2\to P_3)\arrow[ur] && (P_3\to P_2)
\end{tikzcd}
}
\]

\medskip
The Auslander--Reiten quiver of the module category~$\MOD(\Lambda)$ can be obtained by applying the functor~$\Hom{}(\Lambda, -)$.
 
\[
\hspace*{-2cm}
{\footnotesize
\begin{tikzcd}
    \begin{matrix} 3 \end{matrix}\arrow[dr] && \begin{matrix} 2 \\ 1 \end{matrix}\arrow[r]\arrow[dr] & \begin{matrix} 3 \\ 2 \\ 1 \end{matrix}\arrow[r]& \begin{matrix} 3 \\ 2 \end{matrix}\arrow[dr] && \begin{matrix} 1 \end{matrix}\arrow[dr]&& \begin{matrix} 2 \\ 3 \end{matrix} \\
    \begin{matrix} 2 \\ 13 \\ 2 \end{matrix}\ar[r] & \begin{matrix} 2\\ 13 \end{matrix}\arrow[ur]\arrow[dr] && \begin{matrix} 2 \end{matrix}\arrow[ur]\arrow[dr] && \begin{matrix} 13\\2 \end{matrix}\arrow[ur]\arrow[dr]\arrow[r] & \begin{matrix} 2\\13\\2 \end{matrix}\arrow[r] & \begin{matrix} 2\\13 \end{matrix}\arrow[ur]\arrow[dr] &  \\
    \begin{matrix} 1 \end{matrix}\arrow[ur] && \begin{matrix} 2\\3 \end{matrix}\arrow[ur]\arrow[r] & \begin{matrix} 1\\2\\3 \end{matrix}\arrow[r] & \begin{matrix} 1\\2 \end{matrix}\arrow[ur] && \begin{matrix} 3\\1 \end{matrix}\arrow[ur] && \begin{matrix} 2\\1 \end{matrix}
\end{tikzcd}
}
\]
 
\medskip
One can check directly that the conditions of \cref{thm:brick-algebra-condition} is not satisfied.
For example, the module~$\bsm2\\1\esm$ is~$\tau$-rigid since~$\tau(\bsm2\\1\esm) = \bsm3\\1\esm$, while~$\Hom{\Lambda}\left(\bsm2\\1\esm,\tau^2(\bsm2\\1\esm)\right) = \Hom{\Lambda}(\bsm2\\1\esm,\bsm1\\2\esm) \ne 0$.
\end{example}

\addtocontents{toc}{\vspace{.1cm}}
\section*{Acknowledgments}

We thank H.~Thomas for comments on a preliminary version of the paper and for informing us about the upcoming extended version of~\cite{BazierMatteDouvilleMousavandThomasYildirim}.
We are grateful to C.~Stump for reporting a typo in \cref{exm:typeConeCA}.


\bibliographystyle{alpha}
\bibliography{typeConeAssociahedra}

\newcommand{\etalchar}[1]{$^{#1}$}
\begin{thebibliography}{BMDM{\etalchar{+}}18}

\bibitem[ABD10]{ArdilaBenedettiDoker}
Federico Ardila, Carolina Benedetti, and Jeffrey Doker.
\newblock Matroid polytopes and their volumes.
\newblock {\em Discrete Comput.~Geom.}, 43(4):841--854, 2010.

\bibitem[ACEP20]{ArdilaCastilloEurPostnikov}
Federico Ardila, Federico Castillo, Christopher Eur, and Alexander Postnikov.
\newblock Coxeter submodular functions and deformations of {C}oxeter
  permutahedra.
\newblock {\em Adv. Math.}, 365:107039, 36, 2020.

\bibitem[AHBC{\etalchar{+}}16]{ArkaniHamedBourjailyCachazoGoncharovPostnikovTrnka-GrassmannianGeometryScatteringAmplitudes}
Nima Arkani-Hamed, Jacob Bourjaily, Freddy Cachazo, Alexander Goncharov,
  Alexander Postnikov, and Jaroslav Trnka.
\newblock {\em Grassmannian geometry of scattering amplitudes}.
\newblock Cambridge University Press, Cambridge, 2016.

\bibitem[AHBHY18]{ArkaniHamedBaiHeYan}
Nima Arkani-Hamed, Yuntao Bai, Song He, and Gongwang Yan.
\newblock Scattering forms and the positive geometry of kinematics, color and
  the worldsheet.
\newblock {\em J. High Energy Phys.}, (5):096, front matter+75, 2018.

\bibitem[AHBL17]{ArkaniHamedBaiLam-PositiveGeometries}
Nima Arkani-Hamed, Yuntao Bai, and Thomas Lam.
\newblock Positive geometries and canonical forms.
\newblock {\em J. High Energy Phys.}, (11):039, front matter+121, 2017.

\bibitem[AHT14]{ArkaniHamedTrnka-Amplituhedron}
Nima Arkani-Hamed and Jaroslav Trnka.
\newblock The amplituhedron.
\newblock {\em J. High Energy Phys.}, (10):30, 2014.

\bibitem[AIR14]{AdachiIyamaReiten}
Takahide Adachi, Osamu Iyama, and Idun Reiten.
\newblock {$\tau$}-tilting theory.
\newblock {\em Compos. Math.}, 150(3):415--452, 2014.

\bibitem[ASS06]{AssemSimsonSkowronski}
Ibrahim Assem, Daniel Simson, and Andrzej Skowro\'nski.
\newblock {\em Elements of the representation theory of associative algebras.
  {V}ol. 1}, volume~65 of {\em London Mathematical Society Student Texts}.
\newblock Cambridge University Press, Cambridge, 2006.
\newblock Techniques of representation theory.

\bibitem[Aus84]{Auslander1984}
Maurice Auslander.
\newblock Relations for {G}rothendieck groups of {A}rtin algebras.
\newblock {\em Proc. Amer. Math. Soc.}, 91(3):336--340, 1984.

\bibitem[Bar01]{Baryshnikov}
Yuliy Baryshnikov.
\newblock On {S}tokes sets.
\newblock In {\em New developments in singularity theory ({C}ambridge, 2000)},
  volume~21 of {\em NATO Sci. Ser. II Math. Phys. Chem.}, pages 65--86. Kluwer
  Acad. Publ., Dordrecht, 2001.

\bibitem[BDM{\etalchar{+}}20]{BrustleDouvilleMousavandThomasYildirim}
Thomas Br\"{u}stle, Guillaume Douville, Kaveh Mousavand, Hugh Thomas, and Emine
  Y\i ld\i~r\i m.
\newblock On the combinatorics of gentle algebras.
\newblock {\em Canad. J. Math.}, 72(6):1551--1580, 2020.

\bibitem[BFS90]{BilleraFillimanSturmfels}
Louis~J. Billera, Paul Filliman, and Bernd Sturmfels.
\newblock Constructions and complexity of secondary polytopes.
\newblock {\em Adv.~Math.}, 83(2):155--179, 1990.

\bibitem[BIRS09]{BuanIyamaReitenScott}
A.~B. Buan, O.~Iyama, I.~Reiten, and J.~Scott.
\newblock Cluster structures for 2-{C}alabi-{Y}au categories and unipotent
  groups.
\newblock {\em Compos. Math.}, 145(4):1035--1079, 2009.

\bibitem[BLR19]{BanerjeeLaddhaRaman}
Pinaki Banerjee, Alok Laddha, and Prashanth Raman.
\newblock Stokes polytopes: the positive geometry for {$\phi^4$} interactions.
\newblock {\em J. High Energy Phys.}, (8):067, 34, 2019.

\bibitem[BMDM{\etalchar{+}}18]{BazierMatteDouvilleMousavandThomasYildirim}
V\'eronique Bazier-Matte, Guillaume Douville, Kaveh Mousavand, Hugh Thomas, and
  Emine Y\i{}ld\i{}r\i{}m.
\newblock {ABHY} {A}ssociahedra and {N}ewton polytopes of ${F}$-polynomials for
  finite type cluster algebras.
\newblock Preprint,
  \href{http://arxiv.org/abs/1808.09986}{\texttt{arXiv:1808.09986}}, 2018.

\bibitem[BMR{\etalchar{+}}06]{BuanMarshReinekeReitenTodorov}
Aslak~Bakke Buan, Robert Marsh, Markus Reineke, Idun Reiten, and Gordana
  Todorov.
\newblock Tilting theory and cluster combinatorics.
\newblock {\em Adv. Math.}, 204(2):572--618, 2006.

\bibitem[BMR07]{BuanMarshReiten}
Aslak~Bakke Buan, Robert~J. Marsh, and Idun Reiten.
\newblock Cluster-tilted algebras.
\newblock {\em Trans. Amer. Math. Soc.}, 359(1):323--332, 2007.

\bibitem[BMR08]{BuanMarshReiten-mutation}
Aslak~Bakke Buan, Robert~J. Marsh, and Idun Reiten.
\newblock Cluster mutation via quiver representations.
\newblock {\em Comment. Math. Helv.}, 83(1):143--177, 2008.

\bibitem[BPR12]{BergeronPrevilleRatelle}
Fran\c{c}ois Bergeron and Louis-Fran\c{c}ois Pr\'eville-Ratelle.
\newblock Higher trivariate diagonal harmonics via generalized {Tamari} posets.
\newblock {\em Journal of Combinatorics}, 3(3):317--341, 2012.

\bibitem[BR87]{ButlerRingel}
M.~C.~R. Butler and Claus~Michael Ringel.
\newblock Auslander-{R}eiten sequences with few middle terms and applications
  to string algebras.
\newblock {\em Comm. Algebra}, 15(1-2):145--179, 1987.

\bibitem[But81]{Butler}
M.~C.~R. Butler.
\newblock Grothendieck groups and almost split sequences.
\newblock In {\em Integral representations and applications ({O}berwolfach,
  1980)}, volume 882 of {\em Lecture Notes in Math.}, pages 357--368. 1981.

\bibitem[CD06]{CarrDevadoss}
Michael~P. Carr and Satyan~L. Devadoss.
\newblock Coxeter complexes and graph-associahedra.
\newblock {\em Topology Appl.}, 153(12):2155--2168, 2006.

\bibitem[CFZ02]{ChapotonFominZelevinsky}
Fr{\'e}d{\'e}ric Chapoton, Sergey Fomin, and Andrei Zelevinsky.
\newblock Polytopal realizations of generalized associahedra.
\newblock {\em Canad. Math. Bull.}, 45(4):537--566, 2002.

\bibitem[Cha16]{Chapoton-quadrangulations}
Fr\'ed\'eric Chapoton.
\newblock Stokes posets and serpent nests.
\newblock {\em Discrete Math. Theor. Comput. Sci.}, 18(3), 2016.

\bibitem[CP17]{ChatelPilaud}
Gr\'egory Chatel and Vincent Pilaud.
\newblock {C}ambrian {H}opf {Algebras}.
\newblock {\em Adv. Math.}, 311:598--633, 2017.

\bibitem[CSZ15]{CeballosSantosZiegler}
Cesar Ceballos, Francisco Santos, and G\"unter~M. Ziegler.
\newblock Many non-equivalent realizations of the associahedron.
\newblock {\em Combinatorica}, 35(5):513--551, 2015.

\bibitem[CZZ18]{ChangZhouZhu-ClusterSubalgebras}
Wen Chang, Panyue Zhou, and Bin Zhu.
\newblock Cluster subalgebras and cotorsion pairs in frobenius extriangulated
  categories.
\newblock {\em Algebr. Represent. Theory}, 2018.

\bibitem[DIJ19]{DemonetIyamaJasso}
Laurent Demonet, Osamu Iyama, and Gustavo Jasso.
\newblock {$\tau$}-tilting finite algebras, bricks, and {$g$}-vectors.
\newblock {\em Int. Math. Res. Not. IMRN}, (3):852--892, 2019.

\bibitem[DK08]{DehyKeller}
Raika Dehy and Bernhard Keller.
\newblock On the combinatorics of rigid objects in 2-{C}alabi-{Y}au categories.
\newblock {\em Int. Math. Res. Not. IMRN}, (11):Art. ID rnn029, 17, 2008.

\bibitem[DRS10]{DeLoeraRambauSantos}
Jesus~A. {De Loera}, J\"org Rambau, and Francisco Santos.
\newblock {\em Triangulations: Structures for Algorithms and Applications},
  volume~25 of {\em Algorithms and {C}omputation in Mathematics}.
\newblock Springer Verlag, 2010.

\bibitem[Eno19]{Enomoto}
Haruhisa Enomoto.
\newblock Relations for {G}rothendieck groups and representation-finiteness.
\newblock {\em J. Algebra}, 539:152--176, 2019.

\bibitem[Fed21]{Fedele}
Francesca Fedele.
\newblock Grothendieck groups of triangulated categories via cluster tilting
  subcategories.
\newblock {\em Nagoya Math. J.}, 244:204--231, 2021.

\bibitem[FK10]{FuKeller}
Changjian Fu and Bernhard Keller.
\newblock On cluster algebras with coefficients and 2-{C}alabi-{Y}au
  categories.
\newblock {\em Trans. Amer. Math. Soc.}, 362(2):859--895, 2010.

\bibitem[FS05]{FeichtnerSturmfels}
Eva~Maria Feichtner and Bernd Sturmfels.
\newblock Matroid polytopes, nested sets and {B}ergman fans.
\newblock {\em Port. Math. (N.S.)}, 62(4):437--468, 2005.

\bibitem[FZ02]{FominZelevinsky-ClusterAlgebrasI}
Sergey Fomin and Andrei Zelevinsky.
\newblock Cluster algebras. {I}. {F}oundations.
\newblock {\em J. Amer. Math. Soc.}, 15(2):497--529, 2002.

\bibitem[FZ03a]{FominZelevinsky-ClusterAlgebrasII}
Sergey Fomin and Andrei Zelevinsky.
\newblock Cluster algebras. {II}. {F}inite type classification.
\newblock {\em Invent. Math.}, 154(1):63--121, 2003.

\bibitem[FZ03b]{FominZelevinsky-YSystems}
Sergey Fomin and Andrei Zelevinsky.
\newblock {$Y$}-systems and generalized associahedra.
\newblock {\em Ann. of Math. (2)}, 158(3):977--1018, 2003.

\bibitem[FZ07]{FominZelevinsky-ClusterAlgebrasIV}
Sergey Fomin and Andrei Zelevinsky.
\newblock Cluster algebras. {IV}. {C}oefficients.
\newblock {\em Compos. Math.}, 143(1):112--164, 2007.

\bibitem[GKZ08]{GelfandKapranovZelevinsky}
Israel Gelfand, Mikhail Kapranov, and Andrei Zelevinsky.
\newblock {\em Discriminants, resultants and multidimensional determinants}.
\newblock Modern Birkh\"auser Classics. Birkh\"auser Boston Inc., Boston, MA,
  2008.
\newblock Reprint of the 1994 edition.

\bibitem[GM18]{GarverMcConville}
Alexander Garver and Thomas McConville.
\newblock Oriented flip graphs of polygonal subdivisions and noncrossing tree
  partitions.
\newblock {\em J. Combin. Theory Ser. A}, 158:126--175, 2018.

\bibitem[GMM20]{GarverMcConvilleMousavand}
Alexander Garver, Thomas McConville, and Kaveh Mousavand.
\newblock A categorification of biclosed sets of strings.
\newblock {\em J. Algebra}, 546:390--431, 2020.

\bibitem[Hai84]{Haiman}
Mark Haiman.
\newblock Constructing the associahedron.
\newblock Unpublished manuscript, 11 pages, available at
  \url{http://www.math.berkeley.edu/~mhaiman/ftp/assoc/manuscript.pdf}, 1984.

\bibitem[Hau19]{Haugland}
Johanne Haugland.
\newblock Auslander-{R}eiten triangles and {G}rothendieck groups of
  triangulated categories.
\newblock Preprint,
  \href{http://arxiv.org/abs/1904.02506}{\texttt{arxiv.org/abs/1904.02506}},
  2019.

\bibitem[HL07]{HohlwegLange}
Christophe Hohlweg and Carsten Lange.
\newblock Realizations of the associahedron and cyclohedron.
\newblock {\em Discrete Comput.~Geom.}, 37(4):517--543, 2007.

\bibitem[HLN21]{HerschendLiuNakaokaI}
Martin Herschend, Yu~Liu, and Hiroyuki Nakaoka.
\newblock {$n$}-exangulated categories ({I}): {D}efinitions and fundamental
  properties.
\newblock {\em J. Algebra}, 570:531--586, 2021.

\bibitem[HLN22]{HerschendLiuNakaokaII}
Martin Herschend, Yu~Liu, and Hiroyuki Nakaoka.
\newblock {$n$}-exangulated categories ({II}): {C}onstructions from
  {$n$}-cluster tilting subcategories.
\newblock {\em J. Algebra}, 594:636--684, 2022.

\bibitem[HLT11]{HohlwegLangeThomas}
Christophe Hohlweg, Carsten Lange, and Hugh Thomas.
\newblock Permutahedra and generalized associahedra.
\newblock {\em Adv. Math.}, 226(1):608--640, 2011.

\bibitem[HPS18]{HohlwegPilaudStella}
Christophe Hohlweg, Vincent Pilaud, and Salvatore Stella.
\newblock Polytopal realizations of finite type $\mathbf{g}$-vector fans.
\newblock {\em Adv. Math.}, 328:713--749, 2018.

\bibitem[INP18]{IyamaNakaokaPalu}
Osamu Iyama, Hiroyuki Nakaoka, and Yann Palu.
\newblock Auslander--{R}eiten theory in extriangulated categories.
\newblock Preprint,
  \href{https://arxiv.org/abs/1805.03776}{\texttt{arXiv:1805.03776}}, 2018.

\bibitem[IY08]{IyamaYoshino}
Osamu Iyama and Yuji Yoshino.
\newblock Mutation in triangulated categories and rigid {C}ohen-{M}acaulay
  modules.
\newblock {\em Invent. Math.}, 172(1):117--168, 2008.

\bibitem[Jin20]{Jin}
Haibo Jin.
\newblock Cohen-{M}acaulay differential graded modules and negative
  {C}alabi-{Y}au configurations.
\newblock {\em Adv. Math.}, 374:107338, 59, 2020.

\bibitem[Kel01]{Keller-AinfinityAlgebras}
Bernhard Keller.
\newblock Introduction to $a$-infinity algebras and modules.
\newblock {\em Homology Homotopy Appl.}, 3(1):1--35., 2001.

\bibitem[KR07]{KellerReiten}
Bernhard Keller and Idun Reiten.
\newblock Cluster-tilted algebras are {G}orenstein and stably {C}alabi-{Y}au.
\newblock {\em Adv. Math.}, 211(1):123--151, 2007.

\bibitem[KZ08]{KoenigZhu}
Steffen Koenig and Bin Zhu.
\newblock From triangulated categories to abelian categories: cluster tilting
  in a general framework.
\newblock {\em Math. Z.}, 258(1):143--160, 2008.

\bibitem[Lee89]{Lee}
Carl~W. Lee.
\newblock The associahedron and triangulations of the {$n$}-gon.
\newblock {\em European J.~Combin.}, 10(6):551--560, 1989.

\bibitem[Liu20]{Liu-LocalizationsHearts}
Yu~Liu.
\newblock Localizations of the hearts of cotorsion pairs.
\newblock {\em Glasg. Math. J.}, 62(3):564--583, 2020.

\bibitem[LN19]{LiuNakaoka-Hearts}
Yu~Liu and Hiroyuki Nakaoka.
\newblock Hearts of twin cotorsion pairs on extriangulated categories.
\newblock {\em J. Algebra}, 528:96--149, 2019.

\bibitem[Lod04]{Loday}
Jean-Louis Loday.
\newblock Realization of the {S}tasheff polytope.
\newblock {\em Arch.~Math.~(Basel)}, 83(3):267--278, 2004.

\bibitem[LP18]{LangePilaud}
Carsten Lange and Vincent Pilaud.
\newblock Associahedra via spines.
\newblock {\em Combinatorica}, 38(2):443--486, 2018.

\bibitem[LR98]{LodayRonco}
Jean-Louis Loday and Mar{\'{\i}}a~O. Ronco.
\newblock Hopf algebra of the planar binary trees.
\newblock {\em Adv. Math.}, 139(2):293--309, 1998.

\bibitem[LZ20a]{LiuZhou-Abelian}
Yu~Liu and Panyue Zhou.
\newblock Abelian categories arising from cluster tilting subcategories.
\newblock {\em Appl. Categ. Structures}, 28(4):575--594, 2020.

\bibitem[LZ20b]{LiuZhou-AbelianII}
Yu~Liu and Panyue Zhou.
\newblock Abelian categories arising from cluster tilting subcategories {II}:
  quotient functors.
\newblock {\em Proc. Roy. Soc. Edinburgh Sect. A}, 150(6):2721--2756, 2020.

\bibitem[LZ21]{LiuZhou-Gorenstein}
Yu~Liu and Panyue Zhou.
\newblock Gorenstein dimension of abelian categories arising from cluster
  tilting subcategories.
\newblock {\em Czechoslovak Math. J.}, 71(146)(2):435--453, 2021.

\bibitem[McC17]{McConville}
Thomas McConville.
\newblock Lattice structure of {G}rid-{T}amari orders.
\newblock {\em J. Combin. Theory Ser. A}, 148:27--56, 2017.

\bibitem[McM73]{McMullen-typeCone}
P.~McMullen.
\newblock Representations of polytopes and polyhedral sets.
\newblock {\em Geometriae Dedicata}, 2:83--99, 1973.

\bibitem[McM77]{McMullen-Valuations}
P.~McMullen.
\newblock Valuations and {E}uler-type relations on certain classes of convex
  polytopes.
\newblock {\em Proc. London Math. Soc. (3)}, 35(1):113--135, 1977.

\bibitem[Mey74]{Meyer}
Walter Meyer.
\newblock Indecomposable polytopes.
\newblock {\em Trans. Amer. Math. Soc.}, 190:77--86, 1974.

\bibitem[MP19]{MannevillePilaud-accordion}
Thibault Manneville and Vincent Pilaud.
\newblock Geometric realizations of the accordion complex of a dissection.
\newblock {\em Discrete Comput. Geom.}, 61(3):507--540, 2019.

\bibitem[MTTV21]{MasudaThomasTonksVallette}
Naruki Masuda, Hugh Thomas, Andy Tonks, and Bruno Vallette.
\newblock The diagonal of the associahedra.
\newblock {\em J. \'{E}c. polytech. Math.}, 8:121--146, 2021.

\bibitem[NP19]{NakaokaPalu}
Hiroyuki Nakaoka and Yann Palu.
\newblock Extriangulated categories, {H}ovey twin cotorsion pairs and model
  structures.
\newblock {\em Cah. Topol. G\'{e}om. Diff\'{e}r. Cat\'{e}g.}, {L}{X}, 2019.

\bibitem[NZ12]{NakanishiZelevinsky}
Tomoki Nakanishi and Andrei Zelevinsky.
\newblock On tropical dualities in cluster algebras.
\newblock In {\em Algebraic groups and quantum groups}, volume 565 of {\em
  Contemp. Math.}, pages 217--226. Amer. Math. Soc., Providence, RI, 2012.

\bibitem[Pal08]{Palu}
Yann Palu.
\newblock Cluster characters for 2-{C}alabi-{Y}au triangulated categories.
\newblock {\em Ann. Inst. Fourier (Grenoble)}, 58(6):2221--2248, 2008.

\bibitem[Pal09]{Palu-grothendieckGroup}
Yann Palu.
\newblock Grothendieck group and generalized mutation rule for 2-{C}alabi-{Y}au
  triangulated categories.
\newblock {\em J. Pure Appl. Algebra}, 213(7):1438--1449, 2009.

\bibitem[Pil18]{Pilaud-brickAlgebra}
Vincent Pilaud.
\newblock Brick polytopes, lattice quotients, and {H}opf algebras.
\newblock {\em J. Combin. Theory Ser. A}, 155:418--457, 2018.

\bibitem[PK92]{PukhlikovKhovanskii}
A.~V. Pukhlikov and A.~G. Khovanski\u{\i}.
\newblock Finitely additive measures of virtual polyhedra.
\newblock {\em Algebra i Analiz}, 4(2):161--185, 1992.

\bibitem[Pos09]{Postnikov}
Alexander Postnikov.
\newblock Permutohedra, associahedra, and beyond.
\newblock {\em Int. Math. Res. Not. IMRN}, (6):1026--1106, 2009.

\bibitem[PPP19]{PaluPilaudPlamondon-surfaces}
Yann Palu, Vincent Pilaud, and Pierre-Guy Plamondon.
\newblock Non-kissing and non-crossing complexes for locally gentle algebras.
\newblock {\em J. Comb. Algebra}, 3(4):401--438, 2019.

\bibitem[PPP21]{PaluPilaudPlamondon-nonkissing}
Yann Palu, Vincent Pilaud, and Pierre-Guy Plamondon.
\newblock Non-kissing complexes and tau-tilting for gentle algebras.
\newblock {\em Mem. Amer. Math. Soc.}, 274(1343), 2021.

\bibitem[PPS10]{PetersenPylyavskyySpeyer}
T.~Kyle Petersen, Pavlo Pylyavskyy, and David~E. Speyer.
\newblock A non-crossing standard monomial theory.
\newblock {\em J. Algebra}, 324(5):951--969, 2010.

\bibitem[Pre17]{Pressland}
Matthew Pressland.
\newblock A categorification of acyclic principal coefficient cluster algebras.
\newblock Preprint,
  \href{https://arxiv.org/abs/1702.05352}{\texttt{arXiv:1702.05352}}, 2017.

\bibitem[PRV17]{PrevilleRatelleViennot}
Louis-Fran\c{c}ois Pr\'eville-Ratelle and Xavier Viennot.
\newblock The enumeration of generalized {T}amari intervals.
\newblock {\em Trans. Amer. Math. Soc.}, 369(7):5219--5239, 2017.

\bibitem[PRW08]{PostnikovReinerWilliams}
Alexander Postnikov, Victor Reiner, and Lauren~K. Williams.
\newblock Faces of generalized permutohedra.
\newblock {\em Doc.~Math.}, 13:207--273, 2008.

\bibitem[PS12]{PilaudSantos-brickPolytope}
Vincent Pilaud and Francisco Santos.
\newblock The brick polytope of a sorting network.
\newblock {\em European~J.~Combin.}, 33(4):632--662, 2012.

\bibitem[PS15]{PilaudStump-brickPolytope}
Vincent Pilaud and Christian Stump.
\newblock Brick polytopes of spherical subword complexes and generalized
  associahedra.
\newblock {\em Adv. Math.}, 276:1--61, 2015.

\bibitem[PS19]{PilaudSantos-quotientopes}
Vincent Pilaud and Francisco Santos.
\newblock Quotientopes.
\newblock {\em Bull. Lond. Math. Soc.}, 51(3):406--420, 2019.

\bibitem[Ram19]{Raman}
Prashanth Raman.
\newblock The positive geometry for {$\phi^p$} interactions.
\newblock {\em J. High Energy Phys.}, (10):271, 33, 2019.

\bibitem[Rea04]{Reading-latticeCongruences}
Nathan Reading.
\newblock Lattice congruences of the weak order.
\newblock {\em Order}, 21(4):315--344, 2004.

\bibitem[Rea06]{Reading-CambrianLattices}
Nathan Reading.
\newblock Cambrian lattices.
\newblock {\em Adv.~Math.}, 205(2):313--353, 2006.

\bibitem[RS09]{ReadingSpeyer}
Nathan Reading and David~E. Speyer.
\newblock Cambrian fans.
\newblock {\em J.~Eur.~Math.~Soc.}, 11(2):407--447, 2009.

\bibitem[RVdB02]{ReitenVandenbergh}
I.~Reiten and M.~Van~den Bergh.
\newblock Noetherian hereditary abelian categories satisfying {S}erre duality.
\newblock {\em J. Amer. Math. Soc.}, 15(2):295--366, 2002.

\bibitem[SS93]{ShniderSternberg}
Steve Shnider and Shlomo Sternberg.
\newblock {\em Quantum groups: From coalgebras to {D}rinfeld algebras}.
\newblock Series in Mathematical Physics. International Press, Cambridge, MA,
  1993.

\bibitem[SSW17]{SantosStumpWelker}
Francisco Santos, Christian Stump, and Volkmar Welker.
\newblock Noncrossing sets and a {G}rassmann associahedron.
\newblock {\em Forum Math. Sigma}, 5:e5, 49, 2017.

\bibitem[Sta63]{Stasheff}
Jim Stasheff.
\newblock Homotopy associativity of {H}-spaces {I} \& {II}.
\newblock {\em Trans. Amer. Math. Soc.}, 108(2):275--312, 1963.

\bibitem[Ste13]{Stella}
Salvatore Stella.
\newblock Polyhedral models for generalized associahedra via {C}oxeter
  elements.
\newblock {\em J. Algebraic Combin.}, 38(1):121--158, 2013.

\bibitem[Str12]{Street}
Ross Street.
\newblock Parenthetic remarks.
\newblock In {\em Associahedra, {T}amari lattices and related structures},
  volume 299 of {\em Prog. Math. Phys.}, pages 251--268. Birkh\"auser/Springer,
  Basel, 2012.

\bibitem[Tam51]{Tamari}
Dov Tamari.
\newblock {\em Monoides pr\'eordonn\'es et cha\^ines de Malcev}.
\newblock PhD thesis, Universit\'e Paris Sorbonne, 1951.

\bibitem[XZ02]{XiaoZhu}
Jie Xiao and Bin Zhu.
\newblock Relations for the {G}rothendieck groups of triangulated categories.
\newblock {\em J. Algebra}, 257(1):37--50, 2002.

\bibitem[Zel06]{Zelevinsky}
Andrei Zelevinsky.
\newblock Nested complexes and their polyhedral realizations.
\newblock {\em Pure Appl. Math. Q.}, 2(3):655--671, 2006.

\bibitem[ZH19]{ZhaoHuang-Phantom}
Tiwei Zhao and Zhaoyong Huang.
\newblock Phantom ideals and cotorsion pairs in extriangulated categories.
\newblock {\em Taiwanese J. Math.}, 23(1):29--61, 2019.

\bibitem[Zie98]{Ziegler-polytopes}
G{\"u}nter~M. Ziegler.
\newblock {\em Lectures on Polytopes}, volume 152 of {\em Graduate texts in
  Mathematics}.
\newblock Springer-Verlag, New York, 1998.

\bibitem[ZZ18]{ZhouZhu-TriangulatedQuotient}
Panyue Zhou and Bin Zhu.
\newblock Triangulated quotient categories revisited.
\newblock {\em J. Algebra}, 502:196--232, 2018.

\end{thebibliography}
\label{sec:biblio}

\end{document}